\documentclass[12pt,centertags,oneside]{amsart}

\usepackage{amsmath,amstext,amsthm,amssymb,amsfonts,amscd}
\usepackage{mathrsfs}
\usepackage[colorlinks=true]{hyperref}
\usepackage{graphicx}
\usepackage[T1]{fontenc}
\usepackage[latin1]{inputenc}
\usepackage{color,tocvsec2}
\usepackage{enumerate}
\usepackage{lscape}
\usepackage[all]{xy}
\usepackage{fouriernc}

\usepackage{multicol}        
\makeindex             

\usepackage[a4paper,width=14.66cm,top=2.54cm,bottom=2.54cm]{geometry}

\theoremstyle{plain}
\newtheorem{lem}{Lemme}[section]
\newtheorem{thm}[lem]{Theorem}
\newtheorem{prop}[lem]{Proposition}
\newtheorem{cor}[lem]{Corollary}
\newtheorem{ass}[lem]{Assumption}

\theoremstyle{definition}
\newtheorem{defin}[lem]{Definition}

\theoremstyle{remark}
\newtheorem{re}[lem]{Remark}

\numberwithin{equation}{section}
\numberwithin{figure}{section}

\newcommand{\cE}{\mathcal{E}}
\newcommand{\cF}{\mathcal{F}}

\newcommand{\cH}{\mathcal{H}}

\newcommand{\cO}{\mathcal{O}}

\newcommand{\wE}{\widetilde{E}}

\newcommand{\bZ}{\mathbf{Z}}
\newcommand{\bR}{\mathbf{R}}
\newcommand{\bC}{\mathbf{C}}

\newcommand{\fb}{\mathfrak{b}}
\newcommand{\fg}{\mathfrak{g}}
\newcommand{\fh}{\mathfrak{h}}
\newcommand{\fk}{\mathfrak{k}}
\newcommand{\fm}{\mathfrak{m}}
\newcommand{\fn}{\mathfrak{n}}
\newcommand{\fp}{\mathfrak{p}}

\newcommand{\ft}{\mathfrak{t}}
\newcommand{\fz}{\mathfrak{z}}
\newcommand{\fu}{\mathfrak{u}}
\newcommand{\fa}{\mathfrak{a}}
\newcommand{\fc}{\mathfrak{c}}
\newcommand{\fii}{\mathfrak{i}}

\DeclareMathOperator{\Tr}{\mathrm Tr}
\DeclareMathOperator{\Ad}{\mathrm Ad}

\DeclareMathOperator{\vol}{\mathrm vol}

\DeclareMathOperator{\ad}{\mathrm ad}
\DeclareMathOperator{\Sp}{\mathrm{Sp}}
\DeclareMathOperator{\Supp}{\mathrm Supp}
\renewcommand{\Re}{\mathrm{Re}\,}
\DeclareMathOperator{\im}{\mathrm Im}
\DeclareMathOperator{\End}{\mathrm End}
\DeclareMathOperator{\Hom}{\mathrm Hom}

\DeclareMathOperator{\rk}{\mathrm rk}
\DeclareMathOperator{\GL}{\mathrm GL}

\newcommand{\<}{\langle}
\renewcommand{\>}{\rangle}
\newcommand{\ol}{\overline}

\newcommand{\p}{\partial}

\renewcommand{\(}{\left(}
\renewcommand{\)}{\right)}
\renewcommand{\[}{\left[}
\renewcommand{\]}{\right]}
\renewcommand{\l}{\leqslant}
\newcommand{\g}{\geqslant}
\newcommand{\e}{\epsilon}

\newcommand{\bbS}{\mathbb{S}}

\newcommand{\Trs}{{\rm Tr}_{\rm s}}

\begin{document}

\title[Complex valued analytic torsion and dynamical zeta function]{ 
Complex valued analytic torsion and dynamical zeta function on 
locally symmetric spaces}
\author{Shu Shen}
 \address{Institut de Math\'ematiques de Jussieu-Paris Rive Gauche, 
Sorbonne Universit\'e,  4 place Jussieu, 75252 Paris Cedex 5, France.}
\email{shu.shen@imj-prg.fr}
\thanks {I am indebted to  Xiaolong Han, Xiaonan Ma,  Jianqing Yu, 
and Weiping Zhang 
for reading  the preliminary version of  this article and for  useful 
discussions.}

\makeatletter
\@namedef{subjclassname@2020}{%
  \textup{2020} Mathematics Subject Classification}
\makeatother

\date{\today}

\dedicatory{}

\begin{abstract}
We show that the Ruelle dynamical 
zeta function on a closed odd dimensional  locally symmetric space 
twisted by an arbitrary  flat vector bundle has a meromorphic 
extension to the whole complex plane 
and that its leading term in the  Laurent series at the zero point  
is  related to the regularised determinant of the flat Laplacian of 
Cappell-Miller. 
When the flat vector bundle is close to an acyclic and unitary one, we show that  the dynamical zeta function is 
regular at the zero point and that its value is equal 
to the complex valued analytic torsion of  Cappell-Miller. This generalises author's previous 
results  for unitarily flat vector bundles as well as   M\"uller and Spilioti's results on hyperbolic manifolds. 
\end{abstract}

\maketitle
\tableofcontents

\settocdepth{section}
\section*{Introduction}
The purpose of this article is to study the   relation between the 
complex valued analytic torsion of Cappell-Miller and the  value at the zero point of the Ruelle dynamical zeta 
function associated to a  flat vector bundle, which is not 
necessarily unitarily flat, on a closed  odd dimensional  locally symmetric space of reductive type. 

Let $Z$ be a smooth closed manifold. Let $F$ be a complex flat vector bundle on $Z$.
Let $H^{\scriptscriptstyle\bullet }(Z,F)$ be the cohomology of the 
sheaf of locally constant sections of $F$. We assume $H^{\scriptscriptstyle\bullet }(Z,F)=0$.

%

Let $g^{TZ},g^{F}$ be the metrics  on $TZ$ and $F$.
The analytic torsion is a real positive number  introduced by Ray and Singer \cite{RSTorsion}. It is a spectral invariant defined by the Hodge Laplacian associated with 
$g^{TZ},g^F$. If $\dim Z$ is odd, they showed that the analytic torsion does not depend on the metric data, and conjectured that, if  $F$ is unitarily flat (i.e., 
the holonomy representation of $F$ is unitary), there 
is an  equality between the  
analytic torsion and its topological counterpart, the Reidemeister 
torsion \cite{ReidemeisterTorsion,FranzTorsion, 
dRTorsion}. 

%
%
%


Ray and Singer's conjecture was proved by Cheeger \cite{Ch79} and M\"uller \cite{Muller78}, and is now known under the name of 
Cheeger-M\"uller Theorem. Bismut-Zhang \cite{BZ92} and M\"uller 
\cite{Muller2} simultaneously considered generalisations of this 
result. M\"uller \cite{Muller2} extended his result to  odd 
dimensional  oriented manifold and only $\det F$ is required to be unitarily flat. Using the 
Witten deformation \cite{Witten82}, Bismut and Zhang \cite{BZ92}  
generalised the original Cheeger-M\"uller theorem to arbitrary flat 
vector bundles with arbitrary Hermitian metrics on a manifold with arbitrary dimension.

On the other hand, Turaev \cite{Turaev89,FarberTuraev00} generalised the concept of the Reidemeister torsion to a complex valued invariant whose absolute value provides the original Reidemeister torsion, with the help of the so-called 
Euler structure. 

The analytic counterpart of the Turaev torsion was  introduced by 
Braverman-Kappeler, Burghelea-Haller, and Cappell-Miller 
separately.  In \cite{BravermanKappeler,BravermanKappeler07}, using 
the signature operator, Braverman and Kappeler defined what they called the refined analytic torsion 
for flat vector bundles on odd dimensional manifolds, and showed that 
it is equal to the Turaev torsion up to a multiplication by a complex number of absolute value one.  In 
\cite{BurgheleaHaller07,BurgheleaHaller08}, Burghelea and Haller defined a generalised analytic torsion associated to a nondegenerate symmetric bilinear form on a flat vector bundle over an arbitrary 
dimensional manifold. Following the idea of \cite{BZ92},  Su and Zhang 
\cite{SuZhang08} showed that the Burghelea-Haller torsion coincides 
with the square of the Turaev torsion. In 
\cite{BurgheleaHallerII10}, using different methods, Burghelea and 
Haller  proved Su-Zhang's result up to the sign,  in the  
case of odd dimension manifolds.
 In \cite{CappellMiller}, with the help of the flat  Laplacian, Cappell-Miller introduced  another version 
 of the analytic torsion associated to a flat vector bundle on an arbitrary dimensional 
manifold without assuming the existence of the 
nondegenerate symmetric bilinear form.  By adopting the proof of 
Su-Zhang \cite{SuZhang08}, Cappell-Miller showed that   their torsion 
invariant is equal to its topological counterpart when the manifold is oriented and is of odd dimension. 

Milnor \cite{MilnorZcover} initiated  the study of the relation 
between the  torsion invariant and a 
dynamical system.  Fried \cite{FriedRealtorsion} showed that the 
Ray-Singer analytic torsion of an acyclic unitarily flat vector bundle on  a closed odd dimensional  hyperbolic manifold  is equal 
to the value at the zero point of the Ruelle dynamical zeta function of 
the geodesic flow. He conjectured \cite[p.~66, Conjecture]{Friedconj} 
that similar results should hold true for more general flows. In 
\cite{Shfried}, following previous contributions by Moscovici-Stanton 
\cite{MStorsion}, using Bismut's orbital integral formula \cite{B09}, the author affirmed the Fried conjecture for
geodesic flows on closed odd dimensional locally symmetric manifolds equipped with a unitarily flat vector bundle.  In \cite{Shen_Yu}, the 
authors made a further generalisation to closed  odd 
dimensional locally symmetric orbifolds. We refer the reader to \cite{Ma_bourbaki} for an 
introduction to the technique used in \cite{Shfried}.






The Fried conjecture for arbitrary flat vector bundle  is recently considered  by Spilioti \cite{Spilioti15,Spilioti18,Spilioti20} and 
by M\"uller \cite{Muller20} for closed odd dimensional hyperbolic 
manifolds. When the underlying manifold is hyperbolic, Spilioti \cite{Spilioti15,Spilioti18} showed that the Ruelle dynamical  zeta function has a meromorphic extension to $\bC$. M\"uller \cite{Muller20} 
related  its leading coefficient in the  Laurent series at the   zero 
point to the complex valued analytic torsion of Cappell-Miller. When 
the flat vector bundle is close to an acyclic and unitary one, it is 
shown \cite{Muller20,Spilioti20} that  the dynamical  zeta function  is regular at the 
zero point and its value  is equal to the Cappell-Miller analytic torsion. 

In this article, we extend the above results of M\"uller and Spilioti 
to closed odd dimensional locally symmetric spaces of arbitrary ranks. 

%

We refer the reader to  \cite{S19a} for a survey 
on the Fried conjecture on different flows, like Morse-Smale flow \cite{ShenYuMorseSmale} and Anosov flow \cite{ShVariationRuelle}.

Now, we will describe our results in more detail, and explain the techniques used in their proofs.


\subsection{The analytic torsions of Ray-Singer and Cappell-Miller}
Let $Z$ \index{Z@$Z$}   be a smooth closed manifold of dimension $m$, 
and let $(F,\nabla^{F})$\index{F@$F$} be a complex flat vector bundle 
of rank $r$ on $Z$ with flat connection $\nabla^{F}$. Let 
$\Gamma$\index{G@$\Gamma$} be the fundamental group of $Z$. Let 
 $\rho:\Gamma\to \GL_{r}(\bC)$ be the  holonomy representation  of $F$.  
Let 
$\(\Omega^{\scriptscriptstyle\bullet }(Z,F),d^{Z}\)$ 
\index{O@$\Omega^\cdot(Z,F)$} be the de Rham complex  with values in 
$F$. Let $H^{\scriptscriptstyle\bullet }(Z,F)$ be the corresponding 
cohomology. 

If $X$ is the universal cover 
of $Z$, if $\(\Omega^{\scriptscriptstyle\bullet }(X)\otimes 
	\bC^{r}\)^{\Gamma}$ is the space of $\bC^{r}$-valued 
	$\Gamma$-invariant forms on $X$, then  
\begin{align}\label{eqOXR}
	\Omega^{\scriptscriptstyle\bullet 
	}(Z,F)=\(\Omega^{\scriptscriptstyle\bullet }(X)\otimes 
	\bC^{r}\)^{\Gamma}.
\end{align} 

Let $g^{TZ}$ \index{G@$g^{TZ}$} be a Riemannian metric on $TZ$.  Let  
$\Box^{Z}$ be the associated flat Laplacian acting on 
$\Omega^{\scriptscriptstyle\bullet }(Z,F)$. It can be obtained by the restriction of the Hodge Laplacian on $X$ to 
the $\Gamma$-invariant subspace \eqref{eqOXR}.

If $F$ has a flat 
metric $g^{F}$, $\Box^{Z}$ is just the Hodge Laplacian associated 
with $g^{TZ},g^{F}$. By Hodge theory, we have
\begin{align}
	\ker \Box^{Z}=H^{\scriptscriptstyle\bullet }(Z,F). 
\end{align} 
The Ray-Singer analytic torsion is a positive real number defined by the following weighted product of the zeta regularised  determinants 
\begin{align}\label{eq:inn}
T_{\rm RS}(F)=\prod_{i=1}^{m}{\rm 
det}^{*}\(\Box^Z|_{\Omega^{i}(Z,F)}\)^{(-1)^ii/2}.
\end{align}
 Formally, $\det^{*}$  means the product of non zero   
 eigenvalues  counted  with  multiplicities. By a fundamental result  
 of  Ray-Singer \cite{RSTorsion} (see also \cite[Theorem 0.1]{BZ92}), if $\dim Z$ is odd and if 
 $H^{\scriptscriptstyle\bullet }(Z,F)=0$, $T_{\rm RS}(F)$ is independent of  $g^{TZ}$ and $g^{F}$. 
 
For general $F$, $\Box^{Z}$ is a second order 
elliptic differential operator, which is not  necessarily 
self-adjoint. Still $\Box^{Z}$ has a  Hodge-like theory. Namely, if 
$\Omega_{0}^{\scriptscriptstyle\bullet }(Z,F)$ is the characteristic space of $\Box^{Z}$ associated with the eigenvalue $0$, then 
$(\Omega_{0}^{\scriptscriptstyle\bullet }(Z,F),d^{Z})$ is a finite 
dimensional complex such that  
\begin{align}\label{eqOm0}
	H^{\scriptscriptstyle\bullet 
	}\(\Omega_{0}^{\scriptscriptstyle\bullet }(Z,F),d^{Z}\)\simeq 
	H^{\scriptscriptstyle\bullet }(Z,F). 
\end{align} 

By \eqref{eqOm0}, if $\Box^{Z}$ is invertible, then $H^{\scriptscriptstyle\bullet 
}(Z,F)=0$. In this case, the Cappell-Miller torsion $T_{\rm CM}(F)\in 
\bC$ is a non vanishing complex number  given by  the following weighted product of the zeta regularised 
determinants (see Section \ref{srDet} for a definition) \begin{align}\label{eq:inn}
T_{\rm CM}(F)=\prod_{i=1}^{m}{\rm 
det}^{*}\(\Box^Z|_{\Omega^{i}(Z,F)}\)^{(-1)^ii}.
\end{align}
 Note that with 
 our convention, Cappell-Miller's original definition 
 \cite{CappellMiller} corresponds to the inverse of $T_{\rm CM}(F)$.  


For a general $F$, the Cappell-Miller torsion $T_{\rm CM}(F)$ is an 
element in the dual of  the  determinant line of the graded space 
$H^{\scriptscriptstyle\bullet }(Z,F\oplus F^{*})$  (see \cite[Section 
8]{CappellMiller}). It has been shown that \cite[Theorem 8.3]{CappellMiller} if $\dim Z$ is odd, $T_{\rm 
CM}(F)$ does not depend on $g^{TZ}$ and is a topological invariant.

\subsection{The  dynamical zeta function of Ruelle}
Let us recall the definition of the Ruelle dynamical zeta 
function  for geodesic flows introduced by Fried \cite[Section 5]{Friedconj} (see also \cite[Section 2]{S19a}).

Let $(Z,g^{TZ})$ be a connected manifold with nonpositive sectional 
curvature. 
Let $[\Gamma]$ be the set of the conjugacy classes of the fundamental 
group $\Gamma$. 
For $[\gamma]\in [\Gamma]$ \index{G@$[\gamma]$}, let 
$B_{[\gamma]}$\index{B@$B_{[\gamma]}$} be the set of closed 
geodesics in the free homotopy class associated to $[\gamma]$. It is easy to see that all the 
elements in $B_{[\gamma]}$ have the same length $\ell_{[\gamma]}$. 


For simplicity, 
assume that  all the $B_{[\gamma]}$ are  
smooth finite dimensional  submanifolds of the loop space of $Z$. 
This is the case if $(Z,g^{TZ})$ has a negative sectional curvature 
or if $Z$ is locally symmetric. If $\gamma=1$, $B_{[1]}=Z$ is the 
space of trivial geodesics. 
If $\gamma\neq 1$, the group 
$\bbS^1$ acts locally freely on  $B_{[\gamma]}$ by rotation, so that 
$ B_{[\gamma]}/\bbS^1$ is an orbifold. Let $\chi_{\rm orb}(B_{[\gamma]}/\bbS^1)\in \mathbf{Q}$ \index{C@$\chi_{\rm orb}$} be the orbifold Euler characteristic \cite{SatakeGaussB}. Denote by \index{M@$m_{[\gamma]}$}
\begin{align}
m_{[\gamma]}=\left|\ker\big(\bbS^1\to {\rm 
Diff}\big(B_{[\gamma]}\big)\big)\right|\in \mathbf{N}^{*}
\end{align}
the multiplicity of a generic element in $B_{[\gamma]}$. Let 
$\e_{[\gamma]}=\pm1$ be the Lefschetz index of the Poincar\'e return map along a closed geodesic in $B_{[\gamma]}$  induced by the geodesic flow (see 
\cite[(2.17)]{S19a} for a precise definition). If $Z$ is locally 
symmetric, then $\e_{[\gamma]}=1$. 

Let $\rho:\Gamma\to \GL_{r}(\bC)$ be a representation of $\Gamma$.  The formal dynamical zeta function of Ruelle is defined for $\sigma\in \bC$ by
\begin{align}\label{eq:ZetaE}
R_{\rho}(\sigma)=\exp\(\sum_{[\gamma]\in [\Gamma_{+}]} 
\e_{[\gamma]}\Tr[\rho(\gamma) ]\frac{\chi_{\rm orb}\(B_{[\gamma]}/\bbS^1\)}{m_{[\gamma]}}e^{-\sigma \ell_{[\gamma]}}\),
\end{align}
where $[\Gamma_{+}]= [\Gamma]-\{1\}$ is the set of the non trivial 
conjugacy classes of $\Gamma$. We will say that the formal dynamical zeta function is well defined if $R_\rho(\sigma)$ is  holomorphic for $\Re(\sigma)\gg1$ and extends meromorphically to $\sigma\in \bC$.

Observe  that if $(Z,g^{TZ})$ is of  negative sectional curvature,  then $B_{[\gamma]}\simeq \mathbb{S}^1$ and
\begin{align}
\chi_{\rm orb}\(B_{[\gamma]}/\bbS^1\)=1.
\end{align}
In this case, if $\rho$ is a trivial representation,
$R_{\rho}(\sigma)$ was recently  shown to be well defined\footnote{More 
precisely, the authors \cite{GLP2013,DyatlovZworski} showed the meromorphic 
extension when the flow is Anosov.}  
by Giulietti-Liverani-Pollicott \cite{GLP2013} and Dyatlov-Zworski \cite{DyatlovZworski}. 
Moreover, Dyatlov and Zworski \cite{zworski_zero} and  Borns-Weil and Shen \cite{BWShen20} showed that, if $(Z,g^{TZ})$ is a negatively curved  surface, the order of the zero of $R_\rho(\sigma)$ at $\sigma=0$ is related to  
the genus of $Z$. For general $\rho$, the proof of the meromorphic 
extension of $R_{\rho}$ is not particularly difficult. In 
particular, when $0$ is not a resonance of the weight dynamical system 
associated to $\rho$, then $R_{\rho}(0)$ is well defined. In this 
case,  Dang, Guillarmou, Rivi\`ere, and Shen \cite{ShVariationRuelle} 
showed that 
$R_{\rho}(0)$ remains  unchanged under a small perturbation of the geodesic flow.  

\subsection{Results of Fried, M\"uller, and  Spilioti on hyperbolic 
manifolds}
Assume that $Z$ is a connected closed odd dimensional  orientable  hyperbolic manifold.
%
%
%
Recall that $(F,\nabla^{F})$ is a flat vector bundle on $Z$ with holonomy  $\rho:\Gamma\to \GL_{r}(\bC)$.


If $F$ has a  flat metric or equivalently $\rho$ is unitary, 
using the Selberg trace formula, Fried  \cite[Theorem 
3]{FriedRealtorsion} showed that there exist explicit constants 
$C_\rho\in \bR^{*}$ and $r_\rho\in \mathbf{Z}$ such that as  $\sigma\to 0$,
\begin{align}\label{eqintroFried1}
R_\rho(\sigma)=C_\rho T_{\rm RS}(F)^2\sigma^{r_\rho}+\cO(\sigma^{r_{\rho}+1}).
\end{align}
Moreover, if $H^\cdot(Z,F)=0$, then
\begin{align}
&C_\rho=1,&r_\rho=0,
\end{align}
so that
\begin{align}\label{eq:Fintro}
R_\rho(0)=  T_{\rm RS}(F)^2.
\end{align}

When $\rho$ is not unitary, in \cite[Theorem 1.1]{Muller20}, M\"uller has shown a similar result. Namely, there exist explicit constants 
$C_\rho\in \bR^{*}$ and $r_\rho\in \mathbf{Z}$ such that as  $\sigma\to 0$,
\begin{align}
R_\rho(\sigma)=C_\rho \prod_{i=1}^{m}{\rm 
det}^{*}\(\Box^Z|_{\Omega^{i}(Z,F)}\)^{(-1)^ii}+\cO(\sigma^{r_{\rho}+1}).
\end{align}
Moreover, the constants  $C_{\rho}$ and $r_{\rho}$ depend only on 
the dimension of $\Omega_{0}^{\scriptscriptstyle\bullet }(Z,F)$.

If $\rho$ is close to a unitary and acyclic representation of 
$\Gamma$ in the sense of $C^{0}$ (\cite[Section 
6.7]{BravermanKappeler}, see also Section \ref{sAT}), by \cite[Proposition 
6.8]{BravermanKappeler} (see also \cite[Lemma 6.2]{Muller20}), we have $\Omega_{0}^{\scriptscriptstyle\bullet 
}(Z,F)=0$. In this case, in \cite[Proposition 1.3]{Muller20}, M\"uller  obtained $C_{\rho}=1$, $r_{\rho}=0$, and 
\begin{align}
	R_{\rho}(0)=T_{\rm CM}(F). 
\end{align} 
 Similar results were also obtained by Spilioti \cite[Theorem 2]{Spilioti20}.


%

\subsection{The main result of the article}
Let $G$ \index{G@$G$} be a linear connected real reductive group 
\cite[p.~3]{Knappsemi}, and let $\theta$\index{T@$\theta$} be the  Cartan involution. Let
$K$\index{K@$K$} be the maximal compact subgroup of the points of $G$ that are fixed by $\theta$.
Let $\fk$ \index{K@$\fk$}and  $\fg$\index{G@$\fg$} be the  Lie algebras of $K$ and $G$, and let $\fg=\fp\oplus \fk$ be the Cartan decomposition. Let $B$\index{B@$B$} be a nondegenerate bilinear symmetric form
on $\fg$ which is invariant under the adjoint action of $G$ and under $\theta$. Assume that $B$ is positive on $\fp$ and negative on
$\fk$. Set $X = G/K$.\index{X@$X$} Then $B|_{\fp}$ induces a 
Riemannian metric $g^{TX}$ on $X$, such that
$(X,g^{TX})$ has a parallel  nonpositive sectional curvature.

Let $\Gamma\subset G$\index{G@$\Gamma$} be a discrete  
cocompact subgroup of $G$. 
%
Assume for the moment that $\Gamma$ is torsion free. 
Set $Z=\Gamma\backslash X$.  Then $Z$ is a 
closed locally symmetric manifold,  equipped with the induced 
Riemannian metric $g^{TZ}$. Let $\rho:\Gamma\to \GL_{r}(\bC)$ 
be a representation of $\Gamma$. Let $F=\Gamma\backslash (X\times 
\bC^{r})$ be the associated flat vector bundle on $Z$.

When $\rho$ is unitary,  in \cite[p.~66, Conjecture]{Friedconj}, Fried  
conjectured similar results \eqref{eqintroFried1}-\eqref{eq:Fintro} 
still hold. This conjecture was recently proved by the author  
\cite{Shfried} 
and  generalised by Shen-Yu \cite{Shen_Yu} on locally symmetric orbifolds where 
$\Gamma$ is no longer torsion free, following previous contributions by Moscovici-Stanton  \cite{MStorsion}, using Bismut's orbital integral formula \cite{B09}. 


%

  The main result of this article concernes the 
  case where $\rho$ is not always unitary. It generalises  \cite[Theorem 
3]{FriedRealtorsion}, \cite[Theorem 1.1, Proposition 
1.3]{Muller20}, \cite[Theorem 2]{Spilioti20} to locally symmetric 
space, as well as \cite[Theorem 1.1]{Shfried} to non unitary twist. 
Recall that $\Box^Z$ is the flat Laplacian on 
$\Omega^{\scriptsize\bullet}(Z,F)$. 



\begin{thm}\label{Thm1}
	Assume that $\dim Z$ is odd. The following statements hold:
	\begin{enumerate}[i)]
	\item\label{eq1i}  The	dynamical zeta function $R_{\rho}(\sigma)$ is holomorphic for 	$\Re(\sigma)\gg1$ and extends 
	meromorphically to $\sigma \in \bC$.

	\item\label{eq1ii}  There exist explicit  constants 
 $C_{\rho}\in \bR^{*}$ 	and $r_{\rho}\in \mathbf{Z}$ (see 
 \eqref{eqCr111} and \eqref{eqCr1111}) such that, when $\sigma\to 0$,  
 we have 
	\begin{align}\label{eqdCt}
		R_{\rho}(\sigma)=C_{\rho} \left\{\prod_{i=1}^{m}{\rm 
det}^{*}\(\Box^{Z}|_{\Omega^{i}(Z,F)}\)^{(-1)^{i}i} \right\}\sigma^{r_{\rho}}+\cO(\sigma^{r_{\rho}+1}).
\end{align}
\item\label{eq1iii} If $\rho$ is close enough (see  \eqref{eqkeclose}) to an acyclic and unitary 
	representation of $\Gamma$, then $\Box^{Z}$ is invertible and 
\begin{align}\label{eqdCt11int}
&C_{\rho}=1,&r_{\rho}=0,
\end{align} 
 so that 
\begin{align}
	R_{\rho}(0)=T_{\rm CM}(F). 
\end{align} 
\end{enumerate} 
	 \end{thm}


%
%
%
 We will also extend the above result to the case where $\Gamma$ is not torsion free. Then $Z$ is an orbifold and 
 $F$ is a flat orbifold vector bundle.  In this case  the analytic 
 torsion of Ray-Singer is still well-defined 
 \cite{Ma_Orbifold_immersion,Daiyu,Shen_Yu}, and we can define the 
 Cappell-Miller torsion in a similar way. Also, the Ruelle zeta 
 function $R_{\rho}(\sigma)$ was  introduced in \cite{Shen_Yu}. 
 
\begin{thm}\label{Thm2}
 	The statement of Theorem \ref{Thm1} holds when 
 	$Z=\Gamma\backslash G/K$ is an orbifold.  
 \end{thm}
 
If the representation $\rho:\Gamma\to \GL_{r}(\bC)$ admits an 
invariant bilinear form, it induces a flat bilinear form $b_{F}$ on 
$F$. Since $b_{F}$ is flat, by \eqref{eqOXR}, the Laplacian of  
Burghelea-Haller \cite[(26)]{BurgheleaHaller07} \cite[(2.20)]{SuZhang08} coincides with the flat Laplacian $\Box^{Z}$ of 
Cappell-Miller. In particular, we get the following corollary. 

\begin{cor}\label{corin1}
	If the representation $\rho:\Gamma\to \GL_{r}(\bC)$ admits an 
invariant bilinear form,  the statements of Theorems \ref{Thm1} and 
\ref{Thm2} hold for the complexed valued analytic torsion of Burghelea-Haller. 
\end{cor}

\subsection{Moscovici-Stanton's vanishing theorem and Kazhdan's property (T)}
Let $\delta(G)\in \mathbf{N}$ be 
the fundamental rank of $G$, i.e., the 
difference between the complex ranks of $G$ and $K$. Note that $\delta(G)$ and $\dim Z$ 
have the same parity.

If $\delta(G)\g 3$, our theorem 
follows easily from an observation originally due to 
Moscovici-Stanton \cite[Corollary 2.2, Remark 3.7]{MStorsion} (see also Bismut 
\cite[Theorem 7.93]{B09}). More precisely, 
	\begin{align}
		&R_{\rho}(\sigma)\equiv 1,&\prod_{i=1}^{m}{\rm 
det}^{*}\(\Box^{Z}|_{\Omega^{i}(Z,F)}\)^{(-1)^{i}i} =1. 
	\end{align} 

When $\delta(G)=1$, Theorem \ref{Thm1} \ref{eq1i}) \ref{eq1ii}) is 
more significant (see
\cite{BMZ1,BMZ} and also \cite{MullerPfaffL2}  when $\rho$ is the restriction of a representation of $G$).  By the classification theory of real simple Lie 
algebras (\cite[Remark 7.9.2]{B09} and \cite[Theorem 6.15]{Shfried}), $\delta(G)=1$ is equivalent to   
	\begin{align}\label{eqMSdcl}
		\fg=\fg_{1}\oplus 
		\begin{cases}
		\bR\\
		{\mathfrak{sl}}_{3}(\bR)\\
		{\mathfrak {
so}}(p,q)\hbox{ with $p\g q$, $pq> 1$ odd}, 
		\end{cases}
	\end{align} 
	where $\fg_{1}$ is a real semisimple Lie algebra such that 
	$\delta(\fg_{1})=0$. Here, $\delta(\fg_{1})$ is defined in an 
	obvious way as $\delta(G)$. Note that $\mathfrak {so}(1,1)\simeq \bR$. 
	However, we single out this case since $\bR$ is not considered to 
	be simple. 
	
Theorem \ref{Thm1} \ref{eq1iii}) makes sense only if $\Gamma$ has an 
acyclic and unitary representation and if such a representation is not 
isolated. A construction of acyclic and unitary representations can be 
founded in \cite{FriedlNagel15}. Assume now that $\Gamma$ admits an 
acyclic and unitary representation. If $\Gamma$ has Kazhdan's property (T), by \cite{Rapinchuk99}, all unitary representations of $\Gamma$ 
are isolated.  Therefore, Theorem \ref{Thm1} \ref{eq1iii}) gives interesting results only if 
$\Gamma$ does not have Kazhdan's property (T). By \cite[Theorems 
1.7.1 and 3.5.4]{PropertyT08}, this is the case when $\fg$ contains a factor $\mathfrak {so}(p,1)$ or $\mathfrak {
su}(p,1)$. Combining it with \eqref{eqMSdcl}, we get 
	\begin{align}
		\fg=\fg_{1}\oplus 
		\begin{cases}
		\bR\\
		{\mathfrak {
so}}(p,1)\hbox{ with $p\g 3$ odd},
		\end{cases}
	\end{align} 
or  
	\begin{align}
		\fg=\fg_{1}\oplus \begin{cases}
		\mathfrak{
su}(s,1)\hbox{ with $s\g1$}\\
		{\mathfrak {
so}}(s,1)\hbox{ with $s\g 2$ even}
		\end{cases}
		\oplus 		\begin{cases}
		\mathfrak{sl}_{3}(\bR)\\
		{\mathfrak {
so}}(p,q)\hbox{ with $p\g q>1$, $pq$ odd},
		\end{cases}
	\end{align} 
where $\fg_{1}$ is a real semisimple Lie algebra with 
$\delta(\fg_{1})=0$ in both of the above two cases.  

\subsection{The proof of Theorem \ref{Thm1}}Our proof of  Theorem \ref{Thm1} is inspired by 
\cite{Shfried}. As we have seen that we can reduce the proof  to the case  where $\delta(G)=1$.  
In this case, 
our proof relies on   the Selberg zeta functions. 

Assume now $\delta(G)=1$.  Let $\ft\subset \fk$ be a Cartan subalgebra of $\fk$. Let 
$\fh=\fz(\ft)\subset \fg$ be the stabiliser of $\ft$ in $\fg$. By 
\cite[p.~129]{Knappsemi}, 
$\fh\subset \fg$ is a  $\theta$-invariant fundamental Cartan 
subalgebra of $\fg$. Let 
$\fh=\fb\oplus \ft$ be the Cartan decomposition of $\fh$. Note that  
$\dim \fb=\delta(G)=1$.  Let $H\subset G$ be the associated Cartan 
subgroup of $G$.  

Let $Z(\fb)\subset G$ be the 
stabiliser of $\fb$ in $G$ with Lie algebra $\fz(\fb)$. Let $Z^{0}(\fb)$ be the connected 
component of the identity in $Z(\fb)$.  Then $\fz(\fb),Z^{0}(\fb)$ split
\begin{align}
&\fz(\fb)=\fb\oplus\fm, &Z^{0}(\fb)=\exp(\fb)\times M,
\end{align} 
where $M$ is a connected reductive subgroup of $G$ with Lie algebra 
$\fm$. Let $\fm=\fp_{\fm}\oplus \fk_{\fm}$ be the Cartan 
decomposition of $\fm$.  Let 
$\fz^{\bot}(\fb)\subset \fg$ be the orthogonal space of 
$\fz^{\bot}(\fb)$ with respect to  $B$. 

We fix an orientation on 
$\fb$, whose choice  is irrelevant.  Let $\fn\subset \fg$ be the direct sum of the eigenspaces 
associated with the positive eigenvalues 
of a positive element in $\fb$. Then,  $\fn$ is a Lie subalgebra of $\fg$ so that 
\begin{align}
	\fz^{\bot}(\fb)=\fn\oplus \theta\fn. 
\end{align} 

Let $\eta=\eta^{+}-\eta^{-}$ be a virtual representation of $M$ 
acting on the finite dimensional complex vector spaces $E_{\eta}=E_{\eta}^{+}-E_{\eta}^{-}$ 
such that 
\begin{itemize}
	\item  the Casimir of $M$ acts on $\eta^{\pm}$ by the same 
	scalar;
	\item  the restriction of $\eta$ to $K_{M}=K\cap M$ lifts uniquely to a virtual 
representation of $K$.
\end{itemize} 
The Selberg zeta function for the pair $(\eta,\rho)$ is defined 	formally for $\sigma\in \bC$ by 
	\begin{align}\label{ieqdefsel}
		Z_{\eta,\rho}(\sigma)=\exp\Bigg(-\sum_{\tiny\substack{[\gamma]\in 
		[\Gamma_{+}]\\ \gamma\sim e^{a}k^{-1}\in H}}\frac{\chi_{\rm 
		orb}(B_{[\gamma]}/\mathbb{S}^{1})}{m_{[\gamma]}}\Tr\[\rho(\gamma)\]\frac{\Trs^{E_{\eta}}[k^{-1}]}{\left|\det(1-\Ad(e^{a}k^{-1}))|_{\fz^{\bot}(\fb)}\right|^{1/2}}e^{-\sigma \ell_{[\gamma]}}\Bigg),
\end{align} 
where the sum is taken  over the non elliptic conjugacy classes 
$[\gamma]$ of $\Gamma$ such that $\gamma$ can be conjugate by 
elements of $G$ into the Cartan subgroup $H$. 

In \cite[Section 6]{Shfried}, we have shown that the adjoint action of $K_{M}$ on 
$\fp_{\fm}$ lifts uniquely to 
a  virtual representation  of $K$. Let 
$\widehat{\eta}=\widehat{\eta}^{+}-\widehat{\eta}^{-}$ be the unique virtual  representation of $K$ such that 
\begin{align}
\widehat{\eta}|_{K_{M}}=	\Lambda^{\scriptscriptstyle\bullet 
}(\fp_{\fm}^{*})\widehat{\otimes}_{\bR} \eta. 
\end{align} 
The Casimir operator of $\fg$ acts as a generalised Laplacian
 $C^{\fg,Z,\widehat{\eta}^{\pm},\rho}$ 
on the smooth sections over $Z$ of the locally homogenous vector bundle, induced by 
$\widehat{\eta}^{\pm}$, twisted by the flat vector bundle  $F$(see 
\eqref{eqC316}).  By the general theory on elliptic differential 
operators, the regularised determinant
${\rm det} 
\big(C^{\fg,Z,\widehat{\eta}^{\pm},\rho}+\sigma\big)$ is holomorphic on $\sigma\in \bC$. 
Following \cite[Section 7]{Shfried}, using M\"uller's 
Selberg trace formula for arbitrary twists \cite{Muller11},  in 
Section \ref{Sselbergzeta}, we show that for $\Re(\sigma)\gg1$, up to a multiplication by a non zero entire function,  
$Z_{\eta,\rho}(\sigma)$ is just the graded regularised determinant 
\begin{align}\label{eqdetCi1}
\frac{	{\rm det} 
\big(C^{\fg,Z,\widehat{\eta}^{+},\rho}+\sigma_{\eta}+\sigma^{2}\big)}{{\rm det} 
\big(C^{\fg,Z,\widehat{\eta}^{-},\rho}+\sigma_{\eta}+\sigma^{2}\big)},
\end{align} 
where $\sigma_{\eta}\in \bR$ is some constant. In this way, we deduce   
the meromorphic extension of $Z_{\eta,\rho}$ and we get precise  information about the poles and zeros as well as their multiplicities. 

In \cite[Section 6]{Shfried}, we have  shown that 
the adjoint representations $\eta_{j}$ of $M$ on $\Lambda^{j}(\fn^{*})\otimes_{\bR}\bC$ 
satisfy our two previous assumptions.  Using the fact that $R_{\rho}$ is  an alternating product of $Z_{\eta_{j},\rho}$, we 
deduce the meromorphic  extension of $R_{\rho}$. In 
particular, up to a multiplication by a non zero entire function,  
$R_{\rho}$ is  a product of the graded  regularised determinants  
\eqref{eqdetCi1}. Comparing this with the weighted  regularised determinant of the flat Laplacian \eqref{eq:inn}, we get \eqref{eqdCt}. 

As in \cite[Section 8]{Shfried}, the proof of \eqref{eqdCt11int} 
requires a detailed  analysis on the right regular representation of $G$ on 
the left $\Gamma$-invariant space  
$C^{\infty}(G,\bC^{r})^{\Gamma}$.  In general,   
$C^{\infty}(G,\bC^{r})^{\Gamma}$ is not  unitarisable. Let $\mathscr Z(\fg_{\bC})$ be 
the centre of the enveloping algebra of the complexified Lie algebra 
$\fg_{\bC}$.  Given a character 
$\chi:\mathscr Z(\fg_{\bC})\to \bC$, let 
$V_{\chi}(\Gamma,\rho)\subset C^{\infty}(G,\bC^{r})^{\Gamma}$ be 
the subspace of $C^{\infty}(G,\bC^{r})^{\Gamma}$ consisting  of the 
$K$-finite elements such that for all $z\in \mathscr Z(\fg_{\bC})$, $z-\chi(z)$ 
acts  nilpotently. In Section \ref{sRightR}, we will show that $V_{\chi}(\Gamma,\rho)$ is a Harish-Chandra  
$(\fg_{\bC},K)$-module with a generalised infinitesimal character 
$\chi$ (see Definitions \ref{definfch} and \ref{defHC}). Proceeding 
similarly  as in \cite[Section 8]{Shfried}, we 
can deduce  that $C_{\rho}$ and $r_{\rho}$ are  determined  
by  the Euler  characteristics  of  certain cohomologies of $V_{\chi_{\mathbf 
1}}(\Gamma,\rho)$, where $\chi_{\mathbf 
1}$ is the trivial character.  The proof  of 
\eqref{eqdCt11int} reduces to showing  $V_{\chi_{\bf 1}}(\Gamma,\rho)=0$.

When $\rho$ is unitary, using  fundamental results of 
	Vogan-Zuckerman \cite{VoganZuckerman}, Vogan \cite{Vogan2}, and 
	Salamanca-Riba \cite{Salamanca}, in \cite[Proposition 
	8.12]{Shfried}, we have shown that $V_{\chi_{\bf 	1}}(\Gamma,\rho)=0$  if and 
only if $\rho$ is 
acyclic. When $\rho$ is not unitary, even if $\rho$ is acyclic, 
$V_{\chi_{\bf 1}}(\Gamma,\rho)$ does not always vanish (see 
\cite[(1.19)]{Muller20}).  

In Section \ref{S:rep}, using a perturbation argument, we will show that $V_{\chi_{\bf 1}}(\Gamma,\rho)$ 
vanishes if $\rho$ is close enough to an  acyclic and unitary 
representation $\rho_{0}$. Indeed, 
if $\{z_{i}\}_{i=1}^{\dim \fh}$ is a system of generators  $\mathscr Z(\fg_{\bC})$, 
we need to show that the operators $z_{i}-\chi_{\bf 1}(z_{i})$ act invertiblely  on  
sections of certain locally homogenous vector bundle (see Proposition 
\ref{propMBK}) twisted by the 
flat vector bundle associated to  $\rho$. Note that $z_{i}-\chi_{\bf 
1}(z_{i})$ are   differential operators. If $\rho$ is 
close enough to  $\rho_{0}$ in the sense 
of $C^{k}$ with $k$ big enough, all the above operators are lower 
order small 
perturbations of the corresponding  operators associated to 
$\rho_{0}$ which are invertible.  From the above considerations, we can deduce the desired invertibility. 

Remark that in the case of hyperbolic manifolds, only the flat Laplacian is involved. Since the flat Laplacian is the square of the 
flat Hodge-Dirac operator, and since the flat  Hodge-Dirac operator is a 
perturbation  of an invertible self-adjoint Hodge-Dirac operator 
associated to  $\rho_{0}$ by a small $0$-th order operator. Therefore, we can take $k=0$ (see \cite[Proposition 6.8]{BravermanKappeler} and \cite[Lemma 6.2]{Muller20}). 

%
%
%
%
%
%

In summary, the principle of the proofs is 
similar to  \cite{Shfried}. In effect, we refer to \cite{Shfried} as much as 
necessary, and we only develop the proofs when the arguments of  \cite{Shfried} do not extend in a straightforward way. Let us remark finally that 
\begin{itemize}
	\item as in the case where $\rho$ is unitary, Bismut's geometric 
	orbital integral formula plays an essential role in our proof of 
Theorem \ref{Thm1}.  Here, we use such  formulas inexplicitly, since all the involved orbital integrals have been 	evaluated in \cite{Shfried} and \cite{Shen_Yu}.

	\item  it is crucial that the adjoint representations of $K_{M}$ 
	on   
	$\fp_{\fm}$ and $\fn$ lift to virtual representations of 
	$K$. 		In \cite{Shfried}, we prove this fact by the 
	classification theory of simple real Lie algebras \eqref{eqMSdcl}. 
	One of the contributions of the current paper is to give a simple 	
	new conceptual proof of a complexified version of 
	\cite[Theorem 6.11]{Shfried}, which is enough for our applications. 
		\end{itemize}

%

%
%
%
%
%
%
%
%
%

\subsection{ Organisation of the article}
This article is organised as follows.
In Section \ref{SgLap}, we introduce the generalised Laplacian, its 
heat kernel,  its regularised determinant, as well as the flat Laplacian and the Cappell-Miller analytic torsion.


In Section \ref{Sselberg}, we show the  Selberg trace formula with 
arbitrary twists for the heat operator associated to the Casmir,  which 
is originally  due to M\"uller \cite{Muller11}. 

In Section  \ref{Sfondcs}, we  give a simple new conceptual proof for 
the lifting property   \cite[Theorem 6.11]{Shfried}.  

In Section \ref{SRuelle}, we introduce the Ruelle zeta function 
$R_{\rho}$, and 
we state  Theorem \ref{Thm1} as Theorem \ref{thmRmeo}. 

In Section \ref{Sselbergzeta}, we introduce the Selberg zeta 
function $Z_{\eta,\rho}$. We show the  meromorphic extensions of $Z_{\eta,\rho}$ and $R_{\rho}$, and  we establish Theorem \ref{Thm1} \ref{eq1i}) 
\ref{eq1ii}). 

In Section \ref{S:rep}, we show certain   cohomological  formulas for 
$C_{\rho}$ and $r_{\rho}$. We prove   Theorem \ref{Thm1} \ref{eq1iii}).

In Section \ref{Snontorsionfree}, we extend Theorem \ref{Thm1} to orbifolds. 
\subsection{Notation}
Throughout the paper, we use the superconnection formalism of 
\cite{Quillensuper} and \cite[Section 1.3]{BGV}. If $A$ is a 
$\mathbf{Z}_2$-graded algebra and if $a, b\in A$, the 
supercommutator $[a,b]$ is given by
\begin{align}
ab-(-1)^{\deg a \deg b}ba.
\end{align}
If $B$ is another $\mathbf{Z}_2$-graded algebra, we denote 
by $A  \widehat{\otimes}  B$ the super tensor product algebra of $A$ and $B$. 
If $E=E^{+}\oplus E^{-}$ is a $\mathbf{Z}_2$-graded vector 
space, the algebra $\End(E)$ is $\mathbf{Z}_2$-graded. If 
$\tau=\pm1$ on $E^{\pm}$ and if $a\in \End(E)$, the supertrace 
$\Trs[a]$ is defined by $\Tr[\tau a]$.

If $V$ is a real vector space, we will use the notation 
$V_{\bC}=V\otimes_{\bR}\bC$ for its complexification. 
If $M$ is a topological group, we will denote by $M^{0}$ the 
connected component of the identity in $M$.   We make the convention 
that $\mathbf{N}=\{0,1,2,\ldots\}$, 
$\mathbf{N}^{*}=\{1,2,\ldots\}$, $\bR_{+}^{*}=(0,\infty)$. 


\settocdepth{subsection}

\section{The generalised Laplacian} \label{SgLap}
The purpose of this section is to recall  some properties  of  a generalised Laplacian defined on a closed manifold. 
The generalised Laplacian $P$ is a second order (not necessarily 
self-adjoint) elliptic differential operator,  whose principal symbol is a positive scalar.  
We show that the regularised determinant 
$\det(\sigma+P)$ is a  holomorphic function on $\sigma\in \bC$.  Our proof of the meromorphic 
extension of the Ruelle zeta function associated to arbitrary twist given in Section \ref{Sselbergzeta} relies 
on this fact.  

This section is organised as follows.  In Sections  \ref{sspGL}-\ref{srDet}, we introduce the generalised Laplacian, the 
 associated semigroup, and the regularised determinant.

%
%
%
%
%

In Section \ref{sflatv}, we introduce a flat vector bundle and its 
holonomy representation.  We 
establish  an estimate on the growth of the holonomy representation. 

In Section \ref{sgLflat}, we construct a generalised Laplacian 
acting on the sections of  a flat vector bundle. We relates its heat  
kernel with the lifted heat kernel on the universal cover. 

Finally,  in Section \ref{sAT}, given a flat vector bundle, we recall the 
construction of the flat Laplacian of Cappell-Miller, which is a  
generalised Laplacian. We  introduce the complex valued analytic torsion of Cappell-Miller under the assumption that the flat Laplacian is invertible.

\subsection{The spectral theory of a generalised Laplacian}\label{sspGL}
Let $Z$ be a closed manifold of dimension $m$.  Let $E$ be a complex vector bundle on $Z$. Let $C^{\infty}(Z,E)$ be 
the space of $C^{\infty}$-sections of $E$. 

Let $g^{TZ}$ be a Riemannian metric on $Z$. Denote by $dv_{Z}\in 
\Omega^{m}(Z,o(TZ))$ the associated Riemannian volume form, where 
$o(TZ)$ is the orientation line bundle of $Z$. 
Let  $g^{E}$ be a 
Hermitian metric on $E$. Let $\<,\>_{g^{E}}$ be the 
induced Hermitian product on $E$. 
We equip $C^{\infty}(Z,E)$ with an $L^{2}$-product defined for 
$u_{1},u_{2}\in C^{\infty}(Z,E)$ by
\begin{align}\label{eqL2}
	\<u_{1},u_{2}\>_{L^{2}}=\int_{z\in Z}\<u_{1}(z),u_{2}(z)\>_{g^{E}}dv_{Z}. 
\end{align}
Let $L^{2}(Z,E)$ be the space of $L^{2}$-sections of $E$.  For $s\in 
\bR$, let  $\cH^{s}(Z,E)$ be the $s$-Sobolev space of sections of 
$E$. 

Let $P$ be a generalised Laplacian acting on $C^{\infty}(Z,E)$  in 
the sense of \cite[Definition 
2.2]{BGV}. Namely, if $|\cdot|^{2}_{g^{TZ,*}}$ is the dual metric on 
$T^{*}Z$ induced by $g^{TZ}$, then $P$ is a second order  elliptic differential operator with the principal symbol
\begin{align}
&	\sigma(P)(x,\xi)=|\xi|^{2}_{g^{TZ,*}_{x}}\cdot  {\rm id}_{E_{x}},& \hbox{for }(x,\xi)\in 
T^{*}Z. 
\end{align}
Equivalently, there is a metric connection $\nabla^{E}$ on $E$ and  
a first order differential operator $A$ such that 
\begin{align}\label{eqP+A}
	P=(\nabla^{E})^{*}\nabla^{E}+A,
\end{align}
where $(\nabla^{E})^{*}:C^{\infty}(Z,T^{*}Z\otimes_{\bR} E)\to 
C^{\infty}(Z,E)$ is the formal adjoint of $\nabla^{E}$ with  
respect to the $L^{2}$-product induced by $g^{E}$ and $g^{TZ}$. 

By \eqref{eqP+A},	there are $c>0$ and $C>0$ such that  for 
any $u\in C^{\infty}(Z,E)$, 
\begin{align}\label{eqdis}
	\Re \<Pu,u\>_{L^{2}}\g c\|u\|_{\cH^{1}}^{2} 
	-C\|u\|^{2}_{L^{2}}.
\end{align}

%
%
%

For $s_{1},s_{2}\in \bR$, if $L$ is a 
bounded  operator from $\cH^{s_{1}}(Z,E)$ to  $\cH^{s_{2}}(Z,E)$, denoted by 
$\|L\|_{\cH^{s_{1}},\cH^{s_{2}}}$ the corresponding operator norm. 
Consider $P$ as an operator  with domain  $\cH^{2}(Z,E)$. Since $P$ 
is elliptic, the unbounded operator $(P,\cH^{2}(Z,E))$ is closed. By \cite[Theorem 9.3]{ShubinPDO}, $(P,\cH^{2}(Z,E))$ has a compact resolvent, so that its spectrum $\Sp(P)$ is discrete. 
%
Moreover, for any 
$\e >0$, there is $r>0$ such that  
\begin{align}\label{eqSpre}
	\Sp(P)\subset \Big\{\lambda\in \bC: |\lambda|< r \hbox{ or } |\im(\lambda)|<\e 
	\Re(\lambda)\Big\}. 
\end{align}
Also, there are $C_{\e,r}>0$ and $C_{\e,r,n}>0$ for all $n\in \bZ$ such that for all $\lambda\in \bC$ with $ |\lambda |\g r$ and $ 
|\im(\lambda)|\g\e 
	\Re(\lambda)$, we 
have 
\begin{align}\label{eqres}
&	\left\|(\lambda-P)^{-1}\right\|_{L^{2},L^{2}}\l 
\frac{C_{\e,r}}{|\lambda|},&\left\|(\lambda-P)^{-1}\right\|_{\cH^{n},\cH^{n+2}}\l 
C_{\e,r,n}. 
\end{align}


Take $\lambda\in \Sp(P)$.  Let   $V_{P}(\lambda)\subset 
L^{2}(Z,E)$ be the 
characteristic space of $P$ associated to the eigenvalue $\lambda$. 
For $N\gg1$ large enough, we have 
\begin{align}\label{eqVpch}
V_{P}(\lambda)=\ker (\lambda-P)^{N}.
\end{align} 
Set
\begin{align}
m_{P}(\lambda)=\dim 
V_{P}(\lambda)\in \mathbf{N}.	
\end{align}

Since $P$ is elliptic, $V_{P}(\lambda)\subset  
C^{\infty}(Z,E)$.  Since $P$ is a relatively compact 
perturbation of the self-adjoint operator $(\nabla^{E})^{*}\nabla^{E}$, 
by the Keldysh theorem \cite[Theorem 10.1]{Gohberg69}, we have 
\begin{align}\label{eqvp}
	L^{2}(Z,E)=\ol{\oplus_{\lambda\in \Sp(P)}V_{P}(\lambda)}. 
\end{align}
In particular, $\Sp(P)$ is not empty and is an infinite set. 

\subsection{The semigroup of a generalised Laplacian}\label{sSG}
By \eqref{eqdis}, \eqref{eqres}, and by the Lumer-Phillips theorem 
\cite[p. 250]{YosidaFunctionalAnalysis}, $-P$ is the   generator of a 
$C^{0}$-semigroup $\exp(-tP)$  on $L^{2}(Z,E)$. 

Fix $r>0$ and $\e>0$ such that \eqref{eqSpre} and \eqref{eqres} hold.  Let 
$\Gamma=\Gamma_{1}\cup \Gamma_{2}\cup \Gamma_{3}$ be the contour in 
the Figure \ref{fig:1} with the indicated orientation. 

\begin{figure}[htp]
\centering 
\includegraphics[width = 7cm]{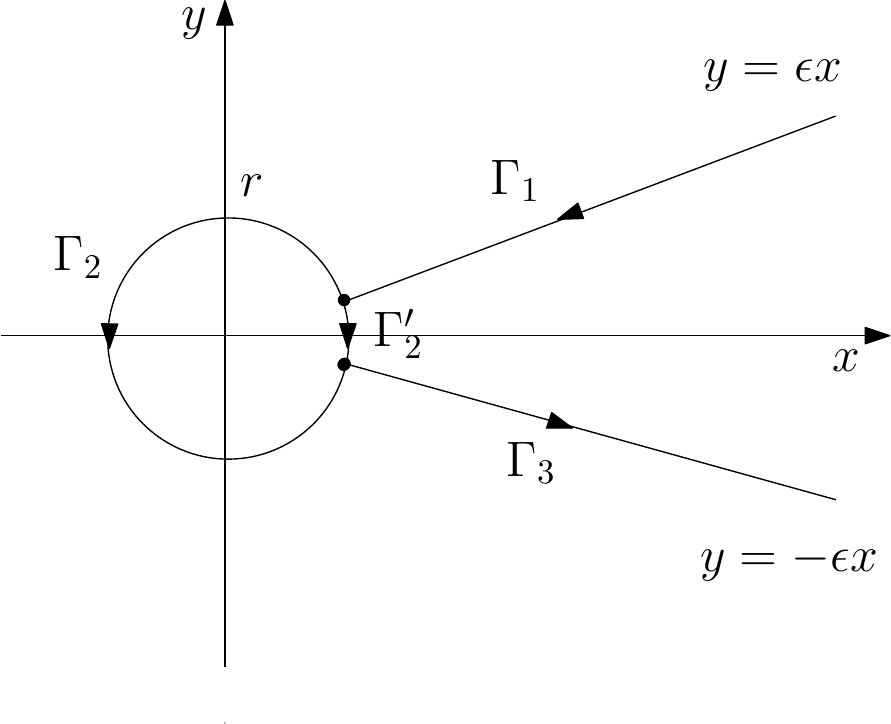}
\caption{The contour.}
\label{fig:1}
\end{figure}

\begin{prop}\label{propsemig}
For $t>0$, the following identity of bounded operators in $L^{2}(X,E)$ holds  
\begin{align}\label{eqetp}
	\exp(-tP)=\frac{1}{2i\pi}\int_{\Gamma}e^{-t\lambda}(\lambda-P)^{-1}d\lambda. 
\end{align} 
The heat operator $\exp(-tP)$ has a smooth integral kernel. In 
particular, it is in trace class such that 
\begin{align}\label{eqTrsg}
	\Tr\[\exp(-tP)\]=\sum_{\lambda\in 
	\Sp(P)}m_{P}(\lambda)e^{-t\lambda},
\end{align}
where the right hand side of \eqref{eqTrsg} converges absolutely. 
\end{prop}
\begin{proof}
By the first estimate of \eqref{eqres},  the integral  on 
the right-hand 
side of \eqref{eqetp} converges absolutely in the space of bounded 
operators. By the general theory of 
$C^{0}$-semigroup \cite[Theorem I.7.7]{Pazy83}, we get 
\eqref{eqetp}. 
	
Using integration by parts, we find that for any $k\in \mathbf{N}$, we 
have
\begin{align}\label{eqseipp}
	\exp(-tP)=\frac{(-1)^{k}k!}{2i\pi 
	t^{k}}\int_{\Gamma}e^{-t\lambda}(\lambda-P)^{-(k+1)}d\lambda. 
\end{align}
Take $k>m/2$. The imbedding $\cH^{2k}(Z,E)\to L^{2}(Z,E)$ is in trace 
class. Denote by $\|\cdot\|_{\rm tr}$ the trace norm.  By the second 
estimate of \eqref{eqres},  there is $C>0$ such that for all 
$\lambda\in \Gamma$, we have 
\begin{align}\label{eqRtr}
\left	\|(\lambda-P)^{-(k+1)}\right\|_{\rm tr}\l C. 
\end{align}
By \eqref{eqRtr}, the right-hand side of  \eqref{eqseipp} converges in the space 
of trace class operators. In particular, $\exp(-tP)$ is in trace class.  Similar methods show that for any $s,s'\in \mathbf{R}$,  $\exp(-tP)$ 
maps $\cH^{s}(Z,E)$ to $\cH^{s'}(Z,E)$ continuously. In particular, $\exp(-tP)$ 
has a smooth integral kernel. 

We claim that 
\begin{align}\label{eqSpsemi}
	\Sp(\exp(-tP))=\{0\}\cup \left\{e^{-t\lambda}: \lambda \in 
	\Sp(P)\right\},
\end{align}
and that  each $e^{-t\lambda}$ has 
the multiplicity 
$m_{P}(\lambda)$. Indeed, it is clear that  $e^{-t\lambda}$ 
with $\lambda\in \Sp(P)$ is the spectrum of $\exp(-tP)$. Since 
$\Sp(P)$ is infinite and since $\Sp(\exp(-tP))$ is a closed set, we see that $0\in \Sp(\exp(-tP))$.  Assume 
that  $\mu\in 
\bC\backslash \{0\}$ does not equal to any 
$e^{-t\lambda}$ with $\lambda\in \Sp(P)$. Note that  $\exp(-tP)$ is 
compact. For $\delta>0$ small enough, the spectral projection of $\exp(-tP)$ to the 
characteristic space of $\mu$ 	
is given by 
\begin{align}
\Pi_{\mu}=	\frac{1}{2i \pi 
	}\int_{|\lambda-\mu|=\delta}\frac{d\lambda}{\lambda-\exp(-tP)}. 
\end{align}
By our choice of $\mu$, $\Pi_{\mu}$ varnishes on $\oplus_{\lambda\in 
\Sp(P)}V_{P}(\lambda)$. By \eqref{eqvp}, $\Pi_{\mu}=0$, from which we 
get \eqref{eqSpsemi}.  Similarly,  
$\Pi_{e^{-t\lambda}}|_{\oplus_{\lambda\in \Sp(P)}V_{P}(\lambda)}$ is just the obvious projection  onto 
$V_{P}(\lambda)$. By \eqref{eqvp}, $\Pi_{e^{-t\lambda}}$ coincides with the spectral projection of $P$ onto 
$V_{P}(\lambda)$.  Thus, the algebraic multiplicity 
of $\mu=e^{-t\lambda}$ is equal to $m_{P}(\lambda)$.

By \eqref{eqSpsemi} and by Lidskii's 
Theorem \cite[Theorem 8.4]{Gohberg69}, we get \eqref{eqTrsg}. 
\end{proof}

Let us study the  long time exponential decay and the short time 
asymptotic of the heat trace $\Tr\[\exp(-tP)\]$. 
We choose $r>0$ and $\e>0$  such that \eqref{eqSpre} and \eqref{eqres} hold and that   
 \begin{align}\label{eqSpre2}
 \Sp(P)\cap \{z\in \bC: |z|=r\}=\varnothing. 
\end{align}
Let $\pi_{<}$ be the spectral projection of $P$ onto the space $\oplus_{\lambda\in 
\Sp(P): |\lambda|<r}V_{P}(\lambda)$.
Then 
\begin{align}\label{eqpi<}
	\pi_{<}=\frac{1}{2i\pi}\int_{|\lambda|=r}\frac{d\lambda}{\lambda-P}.
\end{align}
Set
\begin{align}\label{eqpi>}
	\pi_{>}=1-\pi_{<}. 
\end{align}

\begin{prop}\label{propexptd}
	There are $c>0$ and $C>0$ such that for $t\g 1$, we have 
	\begin{align}\label{eqtrd}
	\left\|\exp(-tP)\pi_{>}\right\|_{\rm tr}\l Ce^{-ct}. 
	\end{align}
\end{prop}
\begin{proof}
	Let $\Gamma'=\Gamma_{1}\cup \Gamma_{2}'\cup \Gamma_{3}$ be the 
 	oriented contour indication in the Figure \ref{fig:1}. Clearly, 
	there is $c>0$ such that if $\lambda\in \Gamma'$ and $t>0$,
	\begin{align}\label{eqreet}
	\left|e^{-t\lambda}\right|\l e^{-ct}. 	
\end{align} 
	

By \eqref{eqvp}, \eqref{eqetp}, \eqref{eqseipp}, 
\eqref{eqSpre2}-\eqref{eqpi>}, for $k\in \mathbf N$, we can deduce that  
	\begin{align}\label{eqetpd}
		\begin{aligned}
		\exp(-tP)\pi_{>}
=\frac{(-1)^{k}k!}{2i\pi 
	t^{k}}\int_{\Gamma'}e^{-t\lambda}(\lambda-P)^{-(k+1)}d\lambda. 
	\end{aligned}
	\end{align}
	Note that the integral converges for 	the operator norm  $\|\cdot \|_{L^{2},L^{2}}$. 

	If $k>m/2$, by the  resolvent identity,  the maps 	$\lambda\in 
	\(\Sp(P)\)^{c}\to 
	\left\|(\lambda-P)^{-(k+1)}\right\|_{\rm tr}$ is continuous.  
	Since $\Gamma_{2}'$ is compact, by \eqref{eqSpre2}, there is $C>0$ such 
	that for all $\lambda \in \Gamma_{2}'$, we have 
	\begin{align}\label{eqreR}
		\left\|(\lambda-P)^{-(k+1)}\right\|_{\rm tr}\l C.
	\end{align}
	By \eqref{eqRtr}, \eqref{eqreet},  and \eqref{eqreR}, if $k>m/2$, the integral 
	on the right-hand side of \eqref{eqetpd} 	converges for the 
	trace norm and we get \eqref{eqtrd}. 
\end{proof}

By \cite[Defintion 2.15, Lemma 2.34]{BGV}, the  integral kernel of $\exp(-tP)$ is exactly the one constructed in \cite[Section 2.4]{BGV}. As a 
consequence of  \cite[Theorem 2.30]{BGV}, we have the following 
proposition. 
\begin{prop}\label{propasy}
	There are  families of complex numbers $\{a_{k}\in \bC\}_{k\in 
	\mathbf{N}}$ and  of  positive real numbers $\{C_{k}>0\}_{k\in \mathbf{N}}$ such that for 
	$k\in \mathbf{N}$ and $t\in (0,1]$, we have
	\begin{align}
\left|		
\Tr\[\exp(-tP)\]-\frac{1}{t^{m/2}}\sum_{i=0}^{k}a_{i}t^{i}\right|\l 
C_{k}t^{k+1-m/2}. 
	\end{align}
\end{prop}

\subsection{The regularised determinant}\label{srDet} 
Fix $r>0$ and $\e>0$ such that \eqref{eqSpre}, 
\eqref{eqres},  and \eqref{eqSpre2} hold. 
Set
\begin{align}
	&{\rm det}_{<}(P)=\prod_{\lambda\in \Sp(P): 
	|\lambda|<r}\lambda^{m(\lambda)}\in 
	\bC,&{\rm det}_{<}^{*}(P)=\prod_{\lambda\in \Sp(P): 
	0<|\lambda|<r}\lambda^{m(\lambda)}\in \bC^{*}.
\end{align}

For $s>m/2$, set
\begin{align}
	\theta_{>}(s)=-\frac{1}{\Gamma(s)}\int_{0}^{\infty}\Tr\[\exp(-tP)\pi_{>}\]t^{s-1}ds.
\end{align}
By Propositions  \ref{propexptd} and \ref{propasy}, proceeding as 
\cite[Section 9.6]{BGV}, we see that
$\theta_{>}(s)$  has a meromorphic extension to $s\in \bC$, which is 
holomorphic at $s=0$. Set
\begin{align}
	{\rm det}_{>}(P)=\exp\(\frac{\p}{\p 
	s}\theta_{>}(0)\)\in \bC^{*}. 
\end{align}

\begin{defin}
Define
\begin{align}
		&\det(P)= {\rm det}_{<}(P)\cdot {\rm det}_{>}(P)\in 
		\bC,&{\rm det}^{*}(P)={\rm det}^{*}_{<}(P)\cdot 
		{\rm det}_{>}(P)\in \bC^{*}. 
	\end{align}
\end{defin}
By \eqref{eqTrsg}, it is easy to see that this definition does not 
depend on the choice of $r>0$ or $\e>0$.

\begin{thm}\label{thmGLH}
	The function $\det(\sigma+P)$	is holomorphic on $\sigma\in 
	\bC$. Its zeros are located  at $\sigma= -\lambda$ with multiplicity 
	$m_{P}(\lambda)$, where $\lambda\in \Sp(P)$. 
\end{thm}
\begin{proof}
Since $\sigma+P$ is a generalised Laplacian, it is enough to show our 
theorem near $\sigma=0$. Recall that $r>0$ and $\e>0$ are chosen such 
that \eqref{eqSpre}, \eqref{eqres}, and \eqref{eqSpre2} hold.  There is a small 
neighborhood $U\subset \bC$ of $0$ such that 
$U\cap \{-\Sp(P)\}\subset \{0\}$, and that 
	\eqref{eqSpre}, \eqref{eqres}, and \eqref{eqSpre2} still  hold for all $\sigma+P$ with 
	$\sigma\in U$. 
	
	We have the identity of functions on $U$, 
\begin{align}
	\det(\sigma+P)={\rm det}_{<}(\sigma+P){\rm det}_{>}(\sigma+P).
\end{align}
By our construction,  ${\rm det}_{>}(\sigma+P)$ is a holomorphic 
function without zero on  $U$, and that ${\rm det}_{<}(P+\sigma)$ is a 
polynomial function on $U$ where 
the only possibility of zeros is situated at $\sigma=0$ with multiplicity $m_{P}(0)$. The proof of our theorem is completed. 
\end{proof}

\begin{re}\label{re17}
	The statements of Propositions \ref{propsemig}-\ref{propasy}, and Theorem 	
	\ref{thmGLH} hold true when $Z$ is a Riemannian closed orbifold 
	and when $E$ is a  Hermitian  orbifold vector bundle on $Z$. 
	Indeed, the proofs of Propositions \ref{propsemig} and \ref{propexptd} are based 
on \eqref{eqSpre} and \eqref{eqres}, which are obtained using the 
peusdodifferential calculus. Using the peusdodifferential calculus  on 
orbifolds \cite[p.~2213]{Ma_Orbifold_immersion}, the generalisations of \eqref{eqSpre} and \eqref{eqres} to 
orbifolds are straightforward. The proof of Proposition \ref{propasy} 
in the orbifold setting is originally due 
to Ma \cite[Proposition 2.1]{Ma_Orbifold_immersion}. The generalisation of Theorem 		
\ref{thmGLH} is a consequence of the orbifold version of Propositions 
\ref{propsemig}-\ref{propasy}. 
\end{re}

\subsection{Flat vector bundles}\label{sflatv}
Let $Z$ be a connected closed manifold. 
Let 
$\Gamma=\pi_{1}(Z)$ be the fundamental group of $Z$.
Let $\widehat{\pi}:X\to Z$ be the universal cover of $Z$. The group  
$\Gamma$ acts on the left on $X$ as the deck transformation, so that 
\begin{align}
	Z=\Gamma\backslash X.
\end{align}

Let $g^{TZ}$ be a Riemannian metric	on $Z$. It lifts to a Riemannian 
metric $g^{TX}$ on $X$, so that  $\Gamma$ acts on $X$ isometrically.  Let $d_{X}(\cdot,\cdot)$ be the Riemannian distance on 
$X$. 
If $\gamma \in \Gamma$, let $d_{\gamma}:X\to \bR$ be the  displacement function on X, i.e.,
\begin{align}\label{eqdisp}
	d_{\gamma}(x)=d_{X}(x,\gamma x). 
\end{align}
%
 Set 
 \begin{align}\label{eqell=inf1}
 	\ell_{\gamma}=\inf_{x\in X} d_{\gamma}(x). 
 \end{align}
If $i_{Z}>0$ is the injectivity radius of $Z$, then for $\gamma\in 
\Gamma$ and $\gamma\neq1$, 
\begin{align}\label{eqdgc01}
	\ell_{\gamma}\g i_{Z}. 
\end{align} 	
 

 Let $F$ be a flat vector bundle of rank $r$ on $Z$ 
with flat connection $\nabla^F$. 
Let $\rho:\Gamma\to \GL_{r}(\bC)$ be the holonomy representation of 
$F$, so that 
\begin{align}\label{eqFhol1}
	F=\Gamma\backslash (X\times \bC^{r}).
\end{align}
Take $x\in X$ such that $\widehat{\pi}(x)=z$. Then $\rho(\gamma)$ can 
be considered as  an element in $\End(F_{z})$ defined by the parallel  transport
	with respect to the flat connection $\nabla^{ F}$ along the loop on 
	$Z$ based at $z$, which is obtained by the projection of a path on $X$	from $\gamma x$ to $x$.

Let $g^{F}$ be a 
Hermitian metric on $F$. 


\begin{prop}\label{proprhor}
	There is 
$C>0$ such that for any $\gamma\in 
\Gamma$, $x\in X$ with $z=\widehat{\pi}(x)$, \begin{align}\label{eqgno}
&\left|	\rho(\gamma)\right|_{g^{F}_{z}}\le 
Ce^{Cd_{\gamma}(x)}, &\big|	\Tr[\rho(\gamma)]\big|\le C e^{C\ell_{\gamma}}. 
\end{align}
\end{prop}
\begin{proof}
	Following \cite[Definitions 4.1, 4.2]{BZ92},  set
	\begin{align}
		\omega^{F}=\(g^{F}\)^{-1}\(\nabla^{F}g^{F}\)\in 
		\Omega^{1}(Z,\End(F)), 
	\end{align}
	and 
	\begin{align}
	\nabla^{F,u}=\nabla^{F}+\frac{1}{2}\omega^{F}.
    \end{align}
	Then, $\nabla^{F,u}$ is a connection on $F$ which 	 preserves the metric $g^{F}$.
	 
	 Let $(x_{t})_{0\l t\l t_{0}}$ be a smooth path  on $X$ of  speed $1$ from $\gamma x$ to $ x$. We 
	trivialise $F$ over $\widehat{\pi}(x_{\cdot})$ by the  parallel  transport 	
	with respect to  the metric connection $\nabla^{F,u}$. For $s,t\in \[0, t_{0}\]$, 
	let $\tau^{s}_{t}\in 
	\Hom(F_{\widehat{\pi}x_{s}},F_{\widehat{\pi} x_{t}})$ be the corresponding parallel  transport. 
	Let $U_{t}\in 
	C^{\infty}(\[0,t_{0}\],\End(F_{z}))$ be the unique 
	solution of the differential equation 
	\begin{align}
&		\dot{U}_{t}-\frac{1}{2}\tau^{t}_{0}\omega^{F}_{\widehat\pi x_{t}}(\dot x_{t})\tau^{0}_{t}U_{t}=0, &U_{0}={\rm Id}. 
	\end{align}
As elements in $\End(F_{z})$, we have 
\begin{align}\label{eqGro1}
	\rho(\gamma)= \tau^{0}_{t_{0}}U_{t_{0}}. 
\end{align}	

Since $Z$ is compact, the smooth section $\omega^{F}$ is uniformly 
bounded. Since $|\dot{x}_{s}|=1$ and since $\tau^{s}_{t}$ preserves $g^{F}$, by Gronwall's  
inequality, there is $C>0$ independent of $x$, of $\gamma$,  and  of 
the path $x_{\cdot}$,   such 
that for all $t>0$, 
\begin{align}\label{eqGro2}
	\left|U_{t}\right|_{g^{F}_{z}}\l Ce^{Ct}. 
\end{align}

By  \eqref{eqGro1} and \eqref{eqGro2}, using again that 
$\tau^{0}_{t_{0}}$ is unitary with respect to  
$g^{F}_{z}$,
we get
\begin{align}
	\left|\rho(\gamma)\right|_{g^{F}_{z}}\le Ce^{Ct_{0}}.
\end{align}
Since $t_{0}$ is the length of the path $(x_{t})_{t\in [0,t_{0}]}$, 
by taking the infimum of all such paths, we get the first inequality of \eqref{eqgno}. 

Using the trivial inequality  on matrices, by the first inequality of \eqref{eqgno}, we have 
\begin{align}\label{eqtrsmno}
\left|	\Tr\[\rho(\gamma)\]\right| \l r 
|\rho(\gamma)|_{g^{F}_{z}}\l rCe^{Cd_{\gamma}(x)}. 
\end{align}
Since the constants $C$ and $r$ do not depend on $x$, by taking the  
infimum on $x\in X$ in \eqref{eqtrsmno}, using \eqref{eqell=inf1}, we get the second inequality of \eqref{eqgno}.
\end{proof}

\subsection{A generalised Laplacian and flat vector 
bundles}\label{sgLflat}
We use the notation of Section \ref{sflatv}. 
Let $\widetilde{E}$ be a Hermitian vector bundle on $X$ with a metric connection $\nabla^{\wE}$. 
Assume that the $\Gamma$-action on $X$ lifts to $\widetilde{E}$ and 
preserves the Hermitian metric and the connection $\nabla^{\wE}$. 
Then $(\widetilde{E},\nabla^{\wE})$ descends  to a Hermitian vector 
bundle $E$ on $Z$ with a metric connection $\nabla^{E}$. 

Let $P^{X}$ be a $\Gamma$-invariant self-adjoint  generalised Laplacian 
acting on $C^{\infty}(X,\widetilde{E})$. Assume that 
\begin{align}\label{eqPX}
	P^{X}=\(\nabla^{\wE}\)^{*}\nabla^{\wE}+\widetilde{A},
\end{align}
where $\widetilde{A}\in C^{\infty}(X,\End(\wE))^{\Gamma}$. Since the 
Riemannian manifold $(X,g^{TX})$ is complete, for $t>0$, it is classical that the heat operator 
$\exp(-tP^{X})$ exists and has a smooth integral kernel. Let 
$p^{X}_{t}(x,y)$ be the   smooth integral kernel  with respect to the Riemannian volume $dv_{X}$.

Recall that $F$ is a flat vector bundle on $Z$ with holonomy  $\rho:\Gamma\to \GL_{r}(\bC)$. We have the identification 
\begin{align}\label{eqsectionEF}
	C^{\infty}(Z,E\otimes F)=\(C^{\infty}(X,\widetilde{E})\otimes 
	\bC^{r}\)^{\Gamma}. 
\end{align}
The operator $P^{X}\otimes {\rm id}$ preserves the above 
$\Gamma$-invariant space. It descends to a generalised  
Laplacian $P^{Z,\rho}$ acting on $C^{\infty}(Z,E\otimes F)$.
Since $\rho$ is not necessarily unitary, there is no obvious metric  on $F$ such that $P^{Z,\rho}$ is 
self-adjoint. Still, by Proposition \ref{propsemig}, 
the heat operator $\exp(-tP^{Z,\rho})$ exists and has a 
smooth integral kernel. For $t>0$, the  smooth integral kernel  with respect to the Riemannian volume $dv_{Z}$ can be identified with a $\Gamma\times 
\Gamma$-invariant section $\widetilde{q}_{t}^{Z,\rho}(x,y)$ in $C^{\infty}(X\times X, (\widetilde{E}\otimes 
\bC^{r})\boxtimes (\widetilde{E}\otimes 
\bC^{r})^{*})$. 

Note that $p_{t}^{X}(\gamma^{-1} x, y)\in \Hom(\widetilde{E}_{
y},\widetilde{E}_{\gamma^{-1} x})$. Let  $\gamma_{*}\in \Hom (\wE_{ 
\gamma^{-1}x},\wE_{x})$  be 
the obvious element. Then, $ \gamma_{*}p_t^{X}(\gamma^{-1}x,  
  y)  \in \Hom(\wE_{ 
y},\wE_{x})$. 
Recall that we have  fixed a metric $g^{F}$ on $F$. 

\begin{thm}\label{thmpsump}
	Given $t_{0}>0$, $T_{0}>0$ with $t_{0}<T_{0}$ and given a bounded open set 
	$X_{0}\subset X$, for any $s\in \mathbf{N}$, the sum on the right-hand side of \eqref{eqsumheat1}
converges absolutely in the space of $C^{s}([t_{0},T_{0}]\times 
X_{0}\times X_{0},\bC\boxtimes(\widetilde{E}\otimes 
\bC^{r})\boxtimes (\widetilde{E}\otimes 
\bC^{r})^{*})$, so that 
\begin{align}\label{eqsumheat1}
	\widetilde{q}_{t}^{Z}(x,y)=\sum_{\gamma\in \Gamma} \rho(\gamma)\otimes    \gamma_{*}p_t^{X}(\gamma^{-1}x, 
  y) .
\end{align}
\end{thm}
\begin{proof}
	We will first show that sum on the right hand side of  
	\eqref{eqsumheat1} converges  to a 
	$\Gamma\times \Gamma$-invariant smooth section $q_{t}(x,y)$ and 
	then show that $q_{t}(x,y)$ satisfies the 
	heat equation.	

	We fix $x_{0}\in X_{0}$. We claim that  for any $\delta>0$, we have 
\begin{align}\label{eq:N<er2}
	 \sum_{\gamma\in \Gamma}\exp\(-\delta d^{2}_{\gamma}(x_{0})\)<+\infty.
\end{align}	
Indeed, by \cite[Remark p.~1, Lemma 2]{Milnor_fundgroup} or  
\cite[(3.19)]{MaMar_cover}, there is $C>0$ such that for all $r\g 0$,
\begin{align}\label{eqestgamma}
	\left|\left\{\gamma\in \Gamma: d_{\gamma}(x_{0})\l r\right\}\right|\l Ce^{Cr}.
\end{align}
For $\delta>0$, by \eqref{eqestgamma}, we have
\begin{multline}
	\sum_{\gamma\in \Gamma}\exp\(-\delta 
	d^{2}_{\gamma}(x_{0})\)=\sum_{\gamma\in 
	\Gamma}\int^{\infty}_{\delta 
	d^{2}_{\gamma}(x_{0})}\exp(-t)dt\\
	=\int^{\infty}_{0}\exp(-t)\left|\left\{\gamma\in \Gamma: 
	d_{\gamma}(x_{0})\l \sqrt{t/\delta}\right\}\right|dt\l 
	C\int^{\infty}_{0}e^{-t+C\sqrt{t/\delta}}dt<+\infty.
\end{multline}

Since $X_{0}$ is bounded, by \eqref{eqgno} and by triangle inequality,  there is $C>0$ such that, 
for any $x\in X_{0}$, we have
\begin{align}\label{eqgno1}
	\left|\rho(\gamma)\right |_{g^{F}_{\widehat{\pi} x}}\l 
	Ce^{Cd_{\gamma}(x_{0})}.
\end{align}



	
	By finite propagation speed of solutions of hyperbolic 
	equations \cite[Theorem D.2.1]{MaMa},  as in \cite[Theorem 
	4]{MaMar_cover},  there are $c>0, C>0$ such that 
	 for any $t\in [t_{0},T_{0}]$ and $x,y\in X$, 
	\begin{align}\label{eqehet}
		\left|p_t^{X}(x,  y) \right|\l C\exp\( -c d_{X}^{2}(x,y)\). 
	\end{align}
		Moreover, similar estimates hold uniformly on $[t_{0},T_{0}]\times 
	X\times X$, for all derivations on $t,x,y$. By \eqref{eqehet} and by triangle inequality,  there are $c>0, C>0$ such that 
	for $t\in [t_{0},T_{0}]$, $x,y\in X_{0}$, $\gamma\in \Gamma$, we have
		\begin{align}\label{eqehetr}
		\left|p_t^{X}(\gamma^{-1}x,  y) \right|\l C\exp\( -c
		d_{\gamma}^{2}(x_{0})\),
	\end{align}	
	and similar estimates for all derivations on $t,x,y$. 

By \eqref{eq:N<er2}, \eqref{eqgno1},  and \eqref{eqehetr}, for all 
$s\in \mathbf{N}$, the sum on the right-hand side of  
	\eqref{eqsumheat1} converges absolutely   in $C^{s}([t_{0},T_{0}]\times X_{0} 
	\times X_{0}, \bC\boxtimes(\widetilde{E}\otimes 
\bC^{r})\boxtimes (\widetilde{E}\otimes 
\bC^{r})^{*})$ to a $\Gamma\times 
	\Gamma$-invariant smooth section $q_{t}(x,y)$, 	
	so that 
	\begin{align}\label{eqehet1}
		\begin{aligned}
		\frac{\p}{\p 		t}q_{t}(x,y)&=\sum_{\gamma\in \Gamma} 
		\rho(\gamma) \otimes \frac{\p}{\p 		
		t}\gamma_{*}p_t^{X}(\gamma^{-1} x,  
  y), \\
 P^{X}_{x}q_{t}(x,y)&=\sum_{\gamma\in \Gamma} \rho(\gamma) \otimes  P^{X}_{x} \gamma_{*}p_t^{X}(\gamma^{-1}x,  
  y),
  \end{aligned}
	\end{align}
where $P^{X}_{x}$ denotes the differential operator 
$P^{X}$ acting on the  variable $x$. 	
By \eqref{eqehet1}, using
$	\frac{\p}{\p 		t} p_t^{X}(x,  
  y)=-P^{X}_{x} p_t^{X}(x,y),
$ and using the fact that 
$p_{t}^{X}$ is $\Gamma\times \Gamma$-invariant,  we get	
	\begin{align}\label{eqehet3}
	\frac{\p}{\p 		t}q_{t}(x,y)=-P^{X}_{x}q_{t}(x,y).
	\end{align}
	
Take $u\in C^{\infty}(Z,E\otimes F)$. We identify $u$ with its $\Gamma$-invariant lifting
$\widetilde{u}\in \(C^{\infty}(X,\widetilde{E})\otimes 
\bC^{r}\)^{\Gamma}$.  Let $F_{Z}\subset X$ be a fundamental 
domain\footnote{For example, we can take $F_{Z}=\{x\in 
X:\text{ for all }\gamma\in \Gamma, d_{X}(x,x_{0})<d_{X}(x,\gamma x_{0})\}.$} of $Z$ in $X$. 
By \eqref{eqgno}, there is $C>0$, for $y\in 
F_{Z}$ and $\gamma\in \Gamma$, we have  
\begin{align}\label{equexp}
	\left|\widetilde{u}(\gamma y)\right|\l Ce^{Cd_{\gamma}(y)}. 
\end{align}
By \eqref{eq:N<er2}, \eqref{eqehetr},  and \eqref{equexp},  using Fubini's theorem, we have
\begin{align}
	\int_{y\in F_{Z}}q_{t}(x,y)\widetilde{u}(y)dv_{X}=\int_{y\in 
	X}p_{t}^{X}(x,y)\widetilde{u}(y)dv_{X}.  
\end{align}
By \cite[Definition 2.15, Lemma 2.34]{BGV}, it remains to show 
that for any $u\in  C^{\infty}(Z,E\otimes F)$, 
as $t\to0$, we have 
\begin{align}\label{eqehet4}
	\sup_{x\in 
	F_{Z}}\left|\int_{y\in X}p^{X}_{t}(x,y)\widetilde{u}(y)dv_{X}- 
	\widetilde{u}(x)\right|\to 0. 
\end{align}

Let $F^{0}_{Z}, F^{1}_{Z}$ be two bounded  open 
neighbourhoods of $\ol{F}_{Z}$ such that $\ol{F}^{0}_{Z}\subset F^{1}_{Z}$.  Write 
\begin{align}
	\widetilde{u}=\widetilde{u}_{1}+\widetilde{u}_{2}, 
\end{align}
where $\widetilde{u}_{1}$ and  $\widetilde{u}_{2}$ are  
$C^{\infty}$-sections on $X$ such that   
$	\Supp(\widetilde{u}_{1})\subset 
F^{1}_{Z}$ and $\Supp(\widetilde{u}_{2})\subset X\backslash 
F^{0}_{Z}.$
Since $\widetilde{u}_{1}$ has a compact support in  $X$,  by the 
Sobolev imbedding theorem, for $k>m/4$, there is $C_{k}>0$ such that  
\begin{multline}\label{eqexpC0}
	\left\|\exp\(-tP^{X}\)\widetilde{u}_{1}-\widetilde{u}_{1}\right\|_{C^{0}(F^{1}_{Z})}
	\l C_{k}
	\left \|\(1+P^{X}\)^{k}
	\(\exp\(-tP^{X}\)\widetilde{u}_{1}-\widetilde{u}_{1}\)\right\|_{L^{2}(X)}\\
	= C_{k}
	\left\|\exp\(-tP^{X}\)\(1+P^{X}\)^{k}\widetilde{u}_{1}-\(1+P^{X}\)^{k}\widetilde{u}_{1}\right\|_{L^{2}(X)}. 
\end{multline}
Since $\(1+P^{X}\)^{k}\widetilde{u}_{1}$ is $L^{2}$ on $X$, 
by  a  property  of the semigroup $\exp(-tP^{X})$ and by \eqref{eqexpC0}, as $t\to 
0$, we have
\begin{align}\label{eqehet5}
	\sup_{x\in F_{Z}}&\left|	
\int_{y\in X}p^{X}_{t}(x,y)\widetilde{u}_{1}(y)dv_{X}- 
	\widetilde{u}_{1}(x)\right|
	\to 0. 
\end{align}
By finite propagation speed of solutions of hyperbolic equations  
\cite[Theorem D.2.1]{MaMa}, there are $c>0,C>0$, for $t\in 
(0,1], x,y\in X$, 
\begin{align}\label{eqestheatd}
	\left|p_{t}^{X}(x,y)\right|\l 
	\frac{C}{t^{m/2}}\exp\(-c\frac{d_{X}^{2}(x,y)}{t}\). 
\end{align}
Since $\Supp (\widetilde{u}_{2})\subset X\backslash F_{Z}^{0}$ and since $d_{X}(X\backslash F^{0}_{Z},F_{Z} )>c'>0$, 
using \eqref{eq:N<er2}, \eqref{equexp}, and \eqref{eqestheatd}, we 
can deduce that there are $c>0$ and $C>0$, for $t\in (0,1]$, 
\begin{align}\label{eqehet6}
\sup_{x\in F_{Z}}\left|	\int_{y\in X}p^{X}_{t}(x,y)\widetilde{u}_{2}(y)dv_{X}\right|\l C 
e^{-c/t}. 
\end{align}
By \eqref{eqehet5} and \eqref{eqehet6}, we get \eqref{eqehet4}. The 
proof of our theorem is completed. 
\end{proof}

\begin{cor}\label{cordpresel}
For $t>0$,	 the right-hand side of the following identity 
convergences absolutely, so that 
	\begin{align}\label{eqselb}
	\Tr\[\exp\(-tP^{Z,\rho}\)\]=\sum_{\gamma\in 
	\Gamma}\Tr[\rho(\gamma)]\int_{x\in F_{Z}}\Tr^{\wE}\[\gamma_{*}p^{X}_{t}\(\gamma^{-1}x,x\)\]dv_{X}.
\end{align}
\end{cor}
\begin{proof}
	This is a consequence of \cite[Theorem D.1.5]{MaMa}, Proposition 
	\ref{proprhor}, Theorem \ref{thmpsump}, and  the estimates 
	\eqref{eqgno}, 	\eqref{eq:N<er2},  \eqref{eqehetr}. 
\end{proof}

Set $\Gamma_{+}=\Gamma-\{{\rm id}\}$. We have a generalisation of \cite[Proposition 4.8]{Shfried}.
\begin{prop}\label{propdgc01}
	There are $c>0,C>0$ such that for any $x\in X$ and 
$t>0$,  
	\begin{align}\label{eqsumheat}
	\sum_{\gamma\in \Gamma_{+}}\big|\Tr[\rho(\gamma)]\big| 
	\left|p_t^{X}\(\gamma^{-1} 
  x,x\)\right|\l C\exp\(-\frac{c}{t}+Ct\). 
\end{align}
\end{prop}
\begin{proof}
%
	Since $X$ is a cover of a compact manifold, by 
	\cite[(4-27)]{Shfried} or \cite[Theorem 4]{MaMar_cover}),  there exist $c>0$, $C>0$ such that for $t>0$, $x,x'\in X$, 
we have 
\begin{align}\label{eqdgc02}
	\left|p_{t}^{X}\(x, x'\)\right|\l 
	\frac{C}{t^{m/2}}\exp\(-c\frac{d_{\gamma}^{2}(x)}{t}+Ct\). 
\end{align}

%
%
By \eqref{eqdgc01} and \eqref{eqdgc02}, there are $c_{1}>0$, $c_{2}>0$, and $C>0$ 
such that for $t>0$, $x\in X$, and $\gamma\in \Gamma_{+}$, we have
\begin{align}\label{eqestheat}
	\left|p_{t}^{X}\(\gamma^{-1}x, x\)\right|\l 
	C\exp\(-\frac{c_{1}}{t}-c_{2}\frac{d_{\gamma}^{2}(x)}{t}+Ct\). 
\end{align}
By \eqref{eqgno} and \eqref{eqestheat}, using 
%
$Cd_{\gamma}(x)\l c_2\frac{d_\gamma^2(x)}{2t}+
\frac{C^{2}}{2c_{2}}t$, we get
\begin{align}
	\big|\Tr[\rho(\gamma)]\big| \left|p_t^{X,\tau}\(\gamma^{-1}x, 
  x\)\right|\l   C
  \exp\(-\frac{c_1}{t}-c_2\frac{d_\gamma^2(x)}{2t}+C't\).
\end{align}
Now proceeding as in \cite[(4-31)]{Shfried}, we get \eqref{eqsumheat}. 
\end{proof}

%
%
%

\subsection{Flat Laplacian and the Cappell-Miller analytic torsion}\label{sAT}
Let  $Z$ be a compact manifold of dimension $m$. Let $(F,\nabla^{F})$ 
be a  complex flat vector bundle on 
$Z$ with holonomy $\rho:\Gamma\to \GL_{r}(\bC)$. Let $\Omega^{\scriptscriptstyle\bullet}(Z,F)$ be the space of smooth differential forms 
with coefficients in $F$. Let $(\Omega^{\scriptscriptstyle\bullet}(Z,F),d^{Z})$ be the de Rham complex, and let $H^{\scriptscriptstyle\bullet}(Z,F)$ be the corresponding  de Rham cohomology. 

Let $g^{TZ}$ be a Riemannian metric on $Z$. It induces a Riemannian 
metric $g^{TX}$ on $X$. Let $\<,\>_{L^{2}}$ be the $L^{2}$-product on 
$\Omega^{\scriptscriptstyle\bullet}(X)$ induced by $g^{TX}$. Let $d^{X,*}$ be the formal 
adjoint of $d^{X}$. Set 
\begin{align}
 \Box^X=\[d^X,d^{X,*}\].
\end{align}
Then, $\Box^{X}$ is a self-adjoint generalised Laplacian  acting on $\Omega^{\scriptscriptstyle\bullet}(X)$. 

We use the convention in Section \ref{sgLflat} with 
$E=\Lambda^{\scriptscriptstyle\bullet } (T^{*}Z)$ and 
$\widetilde{E}=\Lambda^{\scriptscriptstyle\bullet }(T^{*}X)$.
By \eqref{eqsectionEF}, $d^{X}\otimes {\rm id}$ descends to the de 
Rham operator $d^{Z}$. Similarly, $d^{X,*}\otimes {\rm id},\Box^{X}\otimes {\rm id}$ descend to  operators 
$d^{Z,*},\Box^{Z}$, so that 
\begin{align}
	\Box^{Z}=\[d^{Z},d^{Z,*}\].
\end{align} 
The operator $\Box^{Z}$ is a generalised Laplacian which will be called 
the flat Laplacian. 

Let $\Omega_{0}^{\scriptscriptstyle\bullet }(Z,F)$ be the 
characteristic space of $\Box^{Z}$ associated with the eigenvalue 
$0$. Since $d^{Z}$ commutes with $\Box^{Z}$, we see that  
$\(\Omega_{0}^{\scriptscriptstyle\bullet }(Z,F),d^{Z}\)$ is a 
complex. By \cite[p. 
181]{CappellMiller},  
\begin{align}\label{eqHodgeCM}
	H^{\scriptscriptstyle\bullet 
	}\(\Omega_{0}^{\scriptscriptstyle\bullet }(Z,F),d^{Z}\)\simeq 
	H^{\scriptscriptstyle\bullet }(Z,F). 
\end{align} 

Set
\begin{align}\label{eqchiCM}
\chi'_{\rm CM}(F)=\sum_{i=1}^{m}(-1)^{i}i\dim	\Omega_{0}^{i}(Z,F).
\end{align} 
Let $T_{\rho}(\sigma)$ be a meromorphic function on $\bC$ defined by 
\begin{align}\label{eqTs}
	T_{\rho}(\sigma)=\prod_{i=1}^{ m} 
	\det\(\sigma+\Box^{Z}|_{\Omega^{i}(Z,F)}\)^{(-1)^{i}i}. 
\end{align}
As $\sigma\to 0$, we have
\begin{align}\label{eqTs1}
	T_{\rho}(\sigma)=\left\{\prod_{i=1}^{m}{\rm det}^{*}\(\Box^{Z}|_{\Omega^{i}(Z,F)}\)^{(-1)^{i}i}\right\}\sigma^{\chi'_{\rm 
	CM}(Z,F)}+\cO\(\sigma^{\chi'_{\rm CM}(Z,F)+1}\). 
\end{align}

\begin{defin}
	If $\Box^{Z}$ is invertible,  the  complex valued analytic 
	torsion of Cappell-Miller is defined by 
\begin{align}
T_{\rm CM}(F)=	
\prod_{i=1}^{m}{\rm 
det}^{*}\(\Box^{Z}|_{\Omega^{i}(Z,F)}\)^{(-1)^{i}i}\in \bC^{*}.
\end{align}
We refer the reader to \cite{CappellMiller} 
for the definition of the Cappell-Miller torsion in the case where $\Box^{Z}$ is not invertible. 
\end{defin}

Assume that $\Box^{Z}$ is invertible. By \cite[p.~179]{CappellMiller}, if $Z$ is orientable and has even dimension, 
then $T_{\rm CM}(F)=1$. By \cite[Theorem 8.3]{CappellMiller}, if $\dim Z$ is odd, $T_{\rm CM}(F)$ does not depend on the metric $g^{TZ}$. It becomes a 
topological invariant.  

Let ${\rm Rep}(\Gamma,\bC^{r})$ be the set of all $r$-dimensional 
complex representations of $\Gamma$. It is well known that ${\rm 
Rep}(\Gamma,\bC^{r})$ has a natural structure of a complex algebraic 
variety (see \cite[Section 13.6]{BravermanKappeler}). Let us follow \cite[Proposition 4.5]{GoldmanMillson88} and \cite[Section 6.2]{Muller20}. Let $U\subset {\rm 
Rep}(\Gamma,\bC^{r})$ be a contractible neighbourhood 
of $\rho_{0}$. Consider the $\Gamma$-action on  $U\times X\times \bC^{r}$ defined by 
\begin{align}
\gamma \cdot (\rho, x, v)=(\rho,\gamma x, \rho(\gamma)v). 
\end{align} 
The projection on the quotient space 
\begin{align}
\Gamma\backslash (	U\times X\times \bC^{r} )\to U\times Z
\end{align} 
define a vector bundle on $U\times Z$, whose  restriction to 
$\{\rho\}\times Z$ is just the flat vector bundle with holonomy  
$\rho$. We write $F_{\rho}$ to emphasise the dependence on $\rho$. 

Since $U$ is contractible, we have an identification of vector bundles 
over $U\times Z$, 
\begin{align}\label{eq181}
	\Gamma\backslash (	U\times X\times \bC^{r} )\simeq U\times 	
	F_{\rho_{0}}. 
\end{align} 
Note that the  identification is non canonical and is only continuous 
in the variables in $U$. 

For $\rho\in U$, by \eqref{eq181}, we have a bundle isomorphism 
	\begin{align}\label{eqFNrho1}
		F_\rho\simeq F_{\rho_{0}}. 
	\end{align} 
Under this identification, the flat connection on $F_{\rho}$ can be 
written as 
\begin{align}\label{eqFNrho}
		\nabla^{F_{\rho}}=\nabla^{F_{\rho_{0}}}+A_{\rho},
	\end{align} 
	with  $A_\rho\in \Omega^{1}(Z,\End(F_{\rho_{0}}))$.  For $k\in \mathbf N$ 
and $\e>0$, we 
call $\rho$ is $C^{k}$ $\e$-close to $\rho_{0}$, if
\begin{align}\label{eqkeclose}
	\left\|A_{\rho}\right\|_{C^{k}}\l \e.
\end{align} 

Recall that we have fixed a Riemannian metric on $TZ$.

\begin{prop}\label{propclosesign}
	If $\rho_{0}\in {\rm Rep}(\Gamma,\bC^{r})$ is unitary and acyclic,  then there is 
	$\e>0$ such that if $\rho$ is $C^{0}$ $\e$-close to $\rho_{0}$, 
	then the flat 	Laplacian on $F_{\rho}$ is invertible, so that 
	\begin{align}
	T_{\rm CM}(F_{\rho})=\prod_{i=1}^{m}{\rm 
det}^{*}\(\Box^{Z}|_{\Omega^{i}(Z,F_{\rho})}\)^{(-1)^{i}i}\in \bC^{*}.
\end{align} 
\end{prop}
\begin{proof}
This is \cite[Lemma 6.2]{Muller20} (\cite[Proposition 6.8]{BravermanKappeler}), whose proof is 
based on the invertibility of the first order elliptic differential operator 
$d^{Z}+d^{Z,*}$ (c.f. Proposition \ref{propMBK}). 
\end{proof}




\begin{re}\label{re12}
	Assume that $Z$ is a closed orbifold and that $F$ is a flat 	
	orbifold vector bundle on $Z$. Given a Riemannian metric on $Z$, 
	we can define the flat Laplacian, the regularised determinant,  
	Cappell-Miller analytic torsion in the same way, so that 
	\eqref{eqTs1} and Proposition \ref{propclosesign} still hold true. 
\end{re}

\section{M\"uller's Selberg trace formula}\label{Sselberg}


The purpose of this section is to establish the 
Selberg trace formula for the heat operator  of  the Casimir on the locally symmetric space twisted by an arbitrary  flat vector 
bundle.  Such a formula is obtained by M\"uller \cite{Muller11} for the 
functional of the Casimir  with respect to an even Paley-Wiener 
function. The extension to the heat operator does not contain particular difficulties. We include some detail 
for completeness.

This section is organised as follows. 
In Sections \ref{sReductive}-\ref{sec:sym}, we introduce the real 
reductive  group $G$,  its maximal compact subgroup $K$,  and  the 
Casimir $C^{\fg}$,  the associated  symmetric space $X=G/K$, and the 
$K$-principal bundle $p:G\to X$. Given a finite dimensional unitary representation  $\tau:K\to 
U(E_{\tau})$ of $K$, we construct  the associated Hermitian vector bundle 
$\cE_{\tau}$ on $X$. The Casimir operator on  
$C^{\infty}(X,\cE_{\tau})$ is a self-adjoint generalised Laplacian $C^{\fg,X,\tau}$. 

In Sections \ref{sSemi} and \ref{sec:orb}, we introduce the 
semisimple elements in $G$ and the associated semisimple orbital  
integrals with respect to the heat 
 operator of $C^{\fg,X,\tau}$.


Finally, in Section \ref{sec:Gamma}, we introduce a discrete cocompact  subgroup 
$\Gamma\subset G$ of $G$ and the corresponding 
locally symmetric space $\Gamma\backslash X$. 
Given a finite dimensional representation of $\Gamma$, the Casimir operator descends to 
a generalised Laplacian on $Z$. We establish the Selberg trace 
formula for the associated heat operator.

%

\subsection{Real reductive groups}\label{sReductive}
Let $G$ be a linear connected real reductive group 
\cite[p.~3]{Knappsemi}, and let $\theta \in {\rm Aut}(G)$ be the Cartan involution. 
That means $G$ is a closed connected group of real matrices that is 
stable under transpose, and $\theta$ is the composition of transpose 
and inverse of matrices. Let $K \subset G$ be the subgroup of $G$ fixed 
by $\theta$, so that $K$ is a maximal compact subgroup of $G$.

Let $\fg,\fk$ be the Lie algebras of $G,K$.  The Cartan involution acts by differential as  Lie algebra automorphism on $\fg$, which will still be denoted by $\theta$. Then $\fk$ is the eigenspace of $\theta$ associated with the eigenvalue $1$. Let $\fp$ be
the eigenspace of $\theta$ associated with the eigenvalue $-1$, so that
\begin{align}\label{eq:Cartande}
\fg = \fp \oplus \fk.
\end{align}
Set
\begin{align}
&	m=\dim \fp,&n=\dim \fk.
\end{align}
By \cite[Proposition 1.2]{Knappsemi}, we have the diffeomorphism
\begin{align}\label{eq:cartan2}
 (Y,k)\in \fp\times K\to  e^Y k\in G.
\end{align}

Let $B$ \index{B@$B$} be a real-valued  nondegenerate bilinear 
symmetric form on $\fg$ which is invariant under the adjoint action 
$\Ad$ of $G$, and also under $\theta$. Then 
\eqref{eq:Cartande} is an orthogonal splitting of $\fg$ with respect 
to $B$. We assume $B$ to be positive-definite on $\fp$, and negative-definite on $\fk$. 
The form $\<\cdot,\cdot\>=-B(\cdot,\theta\cdot)$ defines an 
$\Ad(K)$-invariant scalar product on $\fg$ such that the splitting 
\eqref{eq:Cartande} is still orthogonal. We denote by $|\cdot|$ \index{1@$\lvert\cdot\rvert$}the corresponding norm.

Let $Z_G\subset G$\index{Z@$Z_G$} be the centre of $G$ with Lie algebra $\fz_\fg\subset \fg$.\index{Z@$\fz_\fg$}
By \cite[Corollary 1.3]{Knappsemi}, $Z_G$ is a (possibly non 
connected) reductive group  with 
maximal compact subgroup $Z_{G}\cap K$ with the Cartan decomposition
\begin{align}\label{eq:ZG}
&\fz_\fg=\fz_{\fp}\oplus \fz_\fk.
\end{align}
Since $\fz_{\fp}$ commutes with $Z_{G}\cap K$, by \eqref{eq:cartan2}, 
we have an identification  of the groups
\begin{align}
	 Z_G=\exp (\fz_\fp) \times (Z_G\cap K).
\end{align}	

Let $G_{\rm ss}\subset G$ be the connected subgroup of $G$ associated 
with  the Lie algebra $[\fg,\fg]$. By \cite[Corollary 7.11]{KnappLie}, 
$G_{\rm ss}$ is a closed subgroup of $G$. Moreover, $G_{\rm ss}$ is 
semisimple and   
\begin{align}\label{eqGOSS}
	G=G_{\rm ss}\cdot Z^{0}_{G}.
\end{align}
\begin{prop}\label{propreal1t}
	If $G$ has a compact center, any one dimensional real 
	representation of 	$G$   is trivial.
\end{prop}
\begin{proof}
This is a 
	consequence of \eqref{eqGOSS} and of the fact that any morphism 
	of groups from a connected semisimple Lie group or a connected compact Lie group to  	$\bR^{*}$ is trivial. 
\end{proof}

Let $\fg_\bC=\fg\otimes_\bR\bC$ \index{G@$\fg_\bC$}be the complexification of $\fg$ and let $\fu=\sqrt{-1}\fp\oplus \fk$\index{U@$\fu$}
be the compact form of $\fg$. By $\bC$-linearity, the bilinear form $B$ extends 
to a complex symmetric bilinear form on $\fg_{\bC}$. Its restriction 
$B|_{\fu}$ to $\fu$ is a real negative-definite symmetric bilinear 
form.   


Let $G_{\bC}$ be the connected group of complex matrices associated 
with the Lie algebra $\fg_{\bC}$. Let $U,U_{\rm ss}\subset G_{\bC}$ be  
the connected subgroup of $G_{\bC}$ associated with the Lie algebras 
$\fu,[\fu,\fu]$.  If $G$ has a compact centre, by \cite[Propositions 
5.3, 5.6]{Knappsemi}, 
$G_{\bC}$ is a reductive group with maximal compact subgroup 
$U$. By \cite[Theorem 4.32]{Knappsemi}, $U_{\rm ss}\subset U$ is a semisimple compact subgroup, so that 
	\begin{align}\label{eqUtss0}
		U=Z_{G}^{0}U_{\rm ss}. 
	\end{align} 
By Weyl's theorem \cite[Theorem 4.26]{Knappsemi}, the 
universal cover $\widetilde{U}_{\rm ss}$ of $U_{\rm ss}$ is 
compact. Set
	\begin{align}\label{eqUtss}
		\widetilde{U}=Z^{0}_{G}\times \widetilde{U}_{ss}.
	\end{align} 
	By \eqref{eqUtss0} and \eqref{eqUtss}, 
the obvious projection $\widetilde{U}\to U$ is a finite cover of  $U$. 

\begin{re}\label{reweyltrick}
	By Weyl's unitary trick \cite[Proposition 5.7]{Knappsemi}, if $G$ is semisimple and if $U$ is simply connected, it is 
	equivalent to consider the finite dimensional complex 
	representations of the Lie groups $G$, $U$ or of the Lie algebras $\fg$, $\fu$. 
	
	For general $G$,  any finite dimensional complex representation of $G$ 
	induces a representation of $\fg$. It extends uniquely to a  representation of $\fu$. In generally, the 
$\fu$-representation does not alway lift to $U$. However, if 
$G$ has a compact centre, by \eqref{eqUtss}, the $\fu$-representation lifts to $\widetilde{U}$. 
\end{re}

Denote by $\rk_{\bC}(G)$ (resp. $\rk_{\bC}(K)$) the complex rank of $G$ 
(resp. $K$), i.e., the dimension of  a Cartan subalgebra of $\fg_{\bC}$ (resp. $\fk_{\bC}$). 

\begin{defin}
  The fundamental rank of $G$ is defined by 
\begin{align}\label{eqd=g-k}
	\delta(G)=\rk_{\bC}(G)-\rk_{\bC}(K)\in \mathbf{N}. 
\end{align}
Note  that $m$ and $\delta(G)$ have the same parity. 
\end{defin}

In the sequel, if $\gamma\in G$, we denote by $Z(\gamma) \subset G$ the centraliser 
of $\gamma$ in $G$, and by $\fz(\gamma)\subset \fg$ its Lie algebra. 
If $a\in \fg$, let $Z(a)\subset G$ 
 be the stabiliser of $a$ in $G$, 
and let $\fz(a)\subset \fg$ be its Lie algebra. If $\fa\subset \fg$ 
is a subset, we define $Z(\fa)$ and $\fz(\fa)$ similarly.


\subsection{The Casimir operator}\label{sCasimir}
Let $\mathscr{U}(\fg)$ \index{U@$\mathscr{U}(\fg)$} be the enveloping algebra 
of $\fg$,  and let $\mathscr{Z}(\fg)\subset \mathscr{U}(\fg)$ \index{Z@$\mathcal{Z}(\fg)$} be the centre of $\mathscr{U}(\fg)$.

Let $C^\fg\in \mathscr{Z}(\fg)$ be the Casimir element associated to $B$.
If $e_1,\cdots,e_m$ is an orthonormal basis of $(\fp,B|_{\fp})$, and 
if $e_{m+1},\cdots,e_{m+n}$ is an  orthonormal basis of 
$(\fk,-B|_{\fk})$.
Then,
\begin{align}\label{eq:Cg}
  C^\fg=-\sum_{i=1}^{m}e^2_i+\sum_{i=m+1}^{n+m}e_i^{2}.\index{C@$C^{\fg}$}
\end{align}

If $V$ is a finite dimensional complex vector space, and if $\rho : \fg 
\to \End(V)$ is a morphism of Lie algebras, the map $\rho$ extends to a 
morphism $\mathscr{U} (\fg) \to \End (V)$ of algebras. 
%
We 
denote by $C^{\fg,V}$ or $C^{\fg,\rho}\in \End(V)$ the corresponding 
Casimir operator acting on $V$, i.e.,
\begin{align}\label{eq:ckkckp}
	C^{\fg,V}=C^{\fg,\rho}=\rho(C^{\fg}).
\end{align}
Similarly,  the Casimir of $\fu$ (with respect to $B$)  acts on 
$V$, so that 
\begin{align}\label{eqCu=Cg}
	C^{\fu,V}=C^{\fg,V}. 
\end{align}

\subsection{The symmetric space}\label{sec:sym}We use the notation in 
Section \ref{sReductive}. 
Let $\omega^{\fg}$ be the canonical left-invariant $1$-form on $G$ 
with values in $\fg$. By \eqref{eq:Cartande}, $\omega^{\fg}$ splits as 
\begin{align}
	\omega^{\fg}=\omega^{\fp}+\omega^{\fk}. 
\end{align}

Let $X=G/K$ be the associated symmetric space. Let $p:G\to X$ be the 
natural  projection. Then $p: G\to X$ is a 
$K$-principal bundle with connection form $\omega^{\fk}$. The 
group $K$ acts isometrically on $\fp$. The tangent bundle is given by
\begin{align}\label{eqTX}
TX=G\times_K \fp.
\end{align}
By \eqref{eqTX}, the scalar product $B|_{\fp}$ on $\fp$ induces a Riemannian metric $g^{TX}$ on $X$. 
The connection $\nabla^{TX}$ on $TX$ which is induced by 
$\omega^{\fk}$ is the Levi-Civita connection of $TX$. Its 
curvature is parallel and nonpositive.

Let 
$e(TX,\nabla^{TX})\in \Omega^{ m}(X,o(TX))$ be the Euler characteristic 
form on $X$. Let  $dv_{X}\in \Omega^{m}(X,o(TX))$ be  the 
Riemannian volume form on the Riemannian manifold $(X,g^{TX})$. Define 
$\[e(TX,\nabla^{TX})\]^{\max}\in \mathbf{R}$ by 
\begin{align}
	e\(TX,\nabla^{TX}\)=\[e\(TX,\nabla^{TX}\)\]^{\max}dv_{X}.  
\end{align}
An explicit formula for $\[e(TX,\nabla^{TX})\]^{\max}$ can be found in \cite[(4-5)]{Shfried}. 

More generally, if $\tau$  is an orthogonal (resp. unitary) 
representation of $K$ on a finite dimensional Euclidean (resp. 
Hermitian) space $E_\tau$, set
\begin{align}
  \label{eq:Etau}
  \cE_\tau=G\times_K E_\tau.
\end{align}
Then $\cE_{\tau}$ is a Euclidean (resp. Hermitian) vector bundle on 
$X$, which is equipped with a connection induced by $\omega^{\fk}$. 

We 
identify the space $C^{\infty}(X,\cE_\tau)$ of smooth sections of 
$\cE_{\tau}$ to the space   $C^\infty(G,E_\tau)^K$ of smooth 
$E_{\tau}$-valued $K$-invariant functions on $G$. 
The group $G$ acts on the left on $C^{\infty}(X,\cE_\tau)$. 
Denote by $C^{\fg,X,\tau}$ the Casimir element of $G$ on 
$C^{\infty}(X,\cE_\tau)$. By \eqref{eq:Cg}, $C^{\fg,X,\tau}$ is a 
self-adjoint  generalised Laplacian  on $X$ satisfying  \eqref{eqPX}. 
When $E_{\tau}=\Lambda^{\cdot}(\fp^{*})$, we use the notation 
$C^{\fg,X}$. It is 
classical (see \cite[Proposition 7.8.1]{B09}) that  
$C^{\fg,X}$ is just the Hodge Laplacian on $X$ associated to the 
trivial line bundle.




\subsection{Semisimple elements}\label{sSemi}
The group $G$ acts isometrically on $X$. If $\gamma \in G$, let $d_{\gamma}$ be the corresponding displacement 
 function on $X$ defined in \eqref{eqdisp}. 
Also, $\ell_{\gamma}$ depends only on the conjugacy class of  $\gamma$ in $G$, and will be denoted by $\ell_{[\gamma]}$. 
Let $X(\gamma)\subset X$ be the closed subset where $d_{\gamma}$ reaches its 
minimum. Clearly, the group $Z(\gamma)$ acts on $X(\gamma)$. 

An element $\gamma\in G$ is called semisimple \cite[Definition 
2.19.21]{Eberlin96}, if $X(\gamma)$ is nonempty. 
If $\gamma$ is semisimple, by \cite[Theorem 3.1.2]{B09}, there is 
$g_{\gamma}\in G$ 
such that  $\gamma=g_{\gamma}e	
^{a}k^{-1}g_{\gamma}^{-1}$ and
\begin{align}\label{eqasr}
	&a\in \fp,&k\in K,&& \Ad(k)a=a.  
\end{align}
Moreover, the norm $|a|$ depends only  on the conjugacy class of $\gamma$ 
in $G$, and 
\begin{align}
	|a|=\ell_{[\gamma]}. 
\end{align}
A semisimple element $\gamma$ is 
called elliptic, if $\ell_{[\gamma]}=0$. 

If $\gamma$ is semisimple, by \cite[Proposition 
7.25]{KnappLie}, $Z(\gamma)$ is a 
reductive group (with Cartan involution $g_{\gamma}\theta 
g_{\gamma}^{-1}$) with maximal compact subgroup $K(\gamma)$ with 
Cartan decomposition 
\begin{align}\label{eqzpkgamma}
\fz(\gamma)=\fp(\gamma)\oplus \fk(\gamma).	
\end{align} 
By 
\cite[Theorem 3.1.1]{B09}, the map $g\in Z(\gamma)\to pg\in X$ 
induces the identification of $Z (\gamma)$-manifolds,
\begin{align}\label{eqXr=ZrKr}
	Z(\gamma)/K(\gamma)\simeq X(\gamma). 
\end{align}

\subsection{Semisimple orbital integrals}\label{sec:orb}

For $t>0$, let $\exp(-tC^{\fg,X,\tau})$ be the heat operator of $C^{\fg,X,\tau}$. 
Let $p_t^{X,\tau}(x,x')$ \index{P@$p_t^{X,\tau}(x,x')$} be the smooth integral  kernel of  $\exp(-tC^{\fg,X,\tau})$
 with respect to the Riemannian volume $dv_X$. 
 
 Let $\gamma\in G$ be a semisimple element.  Since  
 $C^{\fg,X,\tau}$ commutes with the $G$-action on 
 $C^{\infty}(X,\cE_{\tau})$, the function 
 \begin{align}
	g\in G\to \Tr^{E_\tau}\[p_t^{X,\tau}(p g,p\gamma g)\]\in \bC
\end{align} 
 descends to $Z(\gamma)\backslash G$.

The form $-B(\cdot,\theta\cdot)$ induces a volume form $dv_{G}$ on 
$G$. We define $dv_{Z(\gamma)}$ similarly. Let 
$dv_{Z(\gamma)\backslash G}$ be the induced volume form on 
$Z(\gamma)\backslash G$, so that $dv_{G}=dv_{Z(\gamma)\backslash 
G}dv_{Z(\gamma)}$. In the same way, we can also  define 
$dv_{K(\gamma)\backslash K}$ and its volume $\vol(K(\gamma)\backslash 
K)$.

%

\begin{defin}\label{def:orbital}Let $\gamma\in G$ be semisimple.   The orbital integral of $\exp(-tC^{\fg,X,\tau})$ is defined by
\begin{align}\label{eq:TRrz}
	\Tr^{[\gamma]}\[\exp\(-tC^{\fg,X,\tau}\)\]=\frac{1}{\vol(K(\gamma)\backslash  K)}\int_{g\in Z(\gamma)\backslash G}\Tr^{E_\tau}\[p_t^{X,\tau}(p g,p\gamma g)\]dv_{Z(\gamma)\backslash G}.\index{T@$\Tr^{[\gamma]}[\cdot]$}
\end{align}
Clearly, $ \Tr^{[\gamma]}\[\exp\(-tC^{\fg,X,\tau}\)\]$  depends 
only on the conjugacy class of $\gamma$ in $G$. 
\end{defin}


\begin{re}
If $E_\tau$ is a $\bZ_2$-graded or 
virtual  representation of $K$, we use the notation 
$\Trs^{[\gamma]}[\cdot]$ \index{T@$\Trs^{[\gamma]}[\cdot]$}
 when the trace on the right-hand side of \eqref{eq:TRrz} is replaced by the supertrace  on $E_\tau$.
\end{re}

\begin{re}
	An explicit geometric formula for 
	$\Tr^{[\gamma]}\[\exp\(-tC^{\fg,X,\tau}\)\]$ is obtained by 
	Bismut \cite[Theorem 6.1.1]{B09}. This formula involves  an 	
	explicit  function $J_{\gamma}$ defined on the Lie algebra 	
	$\fk(\gamma)$. It  can be written down with 	the 
	help of  a root system \cite{BS19_CRAS}, \cite[Theorem 4.7]{BS19}. 
\end{re}

\begin{re}
Most of the results obtained in 
	\cite{Shfried,Shen_Yu} and as well as in this paper rely on Bismut's formula  \cite[Theorem 6.1.1]{B09}. In this paper, the 
	explicit formula for $J_{\gamma}$ is not needed, since all the 
	involved 	orbital integrals have already been calculated in 	
	\cite{Shfried,Shen_Yu} except for a trivial one \eqref{eq:dgn1}. 
\end{re}

\subsection{A discrete subgroup of $G$}\label{sec:Gamma}
Let $\Gamma\subset G$  be a discrete  
cocompact subgroup of $G$. By \cite[Lemma 1]{Selberg60} (see also 
\cite[Proposition 3.9]{Ma_bourbaki}), $\Gamma$ 
contains only semisimple elements. 
Let $\Gamma_{e}\subset \Gamma$ be the subset of 
elliptic  elements. Then,   $\Gamma_{+}=\Gamma-\Gamma_{e}$ consists of nonelliptic elements. 

The group $\Gamma$ acts isometrically on the left on $X$. 
Take 
\begin{align}
	Z=\Gamma\backslash X=\Gamma \backslash 
G/K. 
\end{align}
Then $Z$ is a compact orbifold. We denote by $\widehat{p}:\Gamma\backslash G\to Z$ and $\widehat{\pi}:X\to Z$ \index{P@$\widehat{p},\widehat{\pi}$}the natural projections, so that the diagram
\begin{align}
\begin{aligned}
\xymatrix{
G \ar[d]^p \ar[r] &\Gamma \backslash G\ar[d]^{\widehat{p}}\\
X \ar[r]^{\widehat{\pi}} &Z}
\end{aligned}
\end{align}
commutes.

From now on until Section \ref{S:rep}, we assume that 
$\Gamma$ is torsion free, i.e., $\Gamma_{e}=\{\rm id\}$. 
Then $Z$ is a 
connected closed orientable Riemannian locally symmetric manifold with nonpositive sectional curvature. Since $X$ is contractible,
  $\pi_1(Z)=\Gamma$ and $X$ is the universal cover of $Z$. In Section 
  \ref{Snontorsionfree}, we will treat the case where $\Gamma$ is no 
  longer torsion free.

The $\Gamma$-action on $X$ lifts to all the homogeneous Euclidean or 
Hermitian vector bundles $\cE_\tau$ on $X$ constructed in 
\eqref{eq:Etau}, and preserves the metric connections. 
Then $\cE_\tau$ descends to a 
Euclidean or Hermitian  vector 
bundle 
\begin{align}\label{eqFtau0}
	\cF_{\tau}=\Gamma\backslash \cE_{\tau}=\Gamma\backslash 
G\times_{K}E_{\tau} 
\end{align}
on $Z$, which is equipped with a canonical metric connection. 

Let $F$ be a flat vector bundle on $Z$ with holonomy representation  
$\rho:\Gamma\to \GL_{r}(\bC)$, so that \eqref{eqFhol1} holds. 
%
As in \eqref{eqsectionEF}, we have the 
identification 
\begin{align}\label{eqC316}
	C^{\infty}(Z,\cF_{\tau}\otimes 
	F)=\(C^{\infty}(X,\cE_{\tau})\otimes \bC^{r}\)^{\Gamma}. 
\end{align}

We use the notation in  Section \ref{sgLflat}. In particular, the 
self-adjoint generalised Laplacian  $C^{\fg,X,\tau}\otimes {\rm id}$
descends to a generalised Laplacian operator $C^{\fg,Z,\tau,\rho}$ 
acting on $C^{\infty}(Z,\cF_{\tau}\otimes F)$. As 
before, if $E_\tau=\Lambda^{\cdot}(\fp^{*})$, we denote 
$C^{\fg,Z,\rho}$ for simplification.  Clearly, $C^{\fg,Z,\rho}$ is 
just the flat Laplacian introduced in Section \ref{sAT}.

For $\gamma\in \Gamma$, set 
\begin{align}
	\Gamma(\gamma)=Z(\gamma)\cap \Gamma. 
\end{align}
By \cite[Lemma2]{Selberg60} (see also \cite[Proposition 
4.9]{Shfried}, \cite[Proposition 3.9]{Ma_bourbaki}), $\Gamma(\gamma)$ is cocompact in $Z(\gamma)$.

Let $[\Gamma_{+}]$ and $[\Gamma]$ be the sets of conjugacy classes in $\Gamma_{+}$ and $\Gamma$.  
If $\gamma\in \Gamma$, the associated conjugacy class in $\Gamma$ is 
denoted by 
$[\gamma]\in [\Gamma]$.\footnote{
The quantities  $\ell_{[\gamma]}$ and 
$\Tr^{[\gamma]}[\cdot]$ depend only on  the conjugacy class of 
$\gamma$ in $G$. So they are well defined on the conjugacy classes of $\Gamma$.  
}
If $[\gamma]\in [\Gamma]$, for all $\gamma'\in [\gamma]$,  the 
locally symmetric spaces 
\begin{align}\label{eqBgamma}
	\Gamma(\gamma')\backslash X(\gamma')
\end{align}
are canonically diffeomorphic, which will be denoted by 
$B_{[\gamma]}$. Let $\vol(B_{[\gamma]})$ be the Riemannian volume 
of $B_{[\gamma]}$ induced by the bilinear form $B$. 

We have a generalisation of \cite[Theorem 4.10]{Shfried}.

\begin{thm}\label{thmselnon}
There exist $c>0$, $C>0$ such that
  for $t>0$, we have
  \begin{align}\label{eq:h11exp}
    \sum_{[\gamma]\in [\Gamma_{+}]} 
	\vol\big(B_{[\gamma]}\big) 
	\big|\Tr[\rho(\gamma)]\big| 
	\left|\Tr^{[\gamma]}\[\exp\(-tC^{\fg,X,\tau}\)\]\right|
	\l C\exp\(-\frac{c}{t}+Ct\).
  \end{align}
  For $t>0$, the following identity holds, 
\begin{align}\label{eq:sel}
  \Tr\[\exp\(-tC^{\fg,Z,\tau,\rho}\)\]=\sum_{[\gamma]\in [\Gamma]} 
  \vol\big(B_{[\gamma]}\big) \Tr[\rho(\gamma)] 
  \Tr^{[\gamma]}\[\exp\(-tC^{\fg,X,\tau}\)\].
\end{align}
\end{thm}
\begin{proof}
Proceeding as the proof of the Selberg trace formula \cite{Se56}, by  \cite[(4.8.8), (4.8.11), (4.8.16)]{B09}, for $[\gamma]\in [\Gamma]$, we have
\begin{align}\label{eq:sel2}
 \sum_{\gamma'\in [\gamma]}\int_{x\in 
 F_{Z}}\Tr^{E_\tau}\[\gamma'_{*}p_t^{X,\tau}\((\gamma')^{-1} x, 
 x\)\]dv_{X}
 =\vol\big(B_{[\gamma]}\big) \Tr^{[\gamma]}\[\exp\(-tC^{\fg,X,\tau}\)\]. 
\end{align}
By Corollary \ref{cordpresel}, Proposition \ref{propdgc01},  and \eqref{eq:sel2}, 
we get our proposition. 
\end{proof}

\begin{re}
In \cite[Theorem 1.1]{Muller11}, instead of heat operators, M\"uller 
obtain  a similar formula for $\varphi(C^{\fg,Z,\tau,\rho})$ where 
$\varphi$ is an  even Paley-Wiener function on $\bR$.  
\end{re}

Let us give a direct application of the  Selberg trace formula. 
Recall the following theorem due to \cite[p.~194]{MStorsion} and 
\cite[Theorem 7.9.1]{B09}.  Let 
$N^{\Lambda^{\scriptscriptstyle\bullet  }(T^{*}X)}$ be the number 
operator on $\Lambda^{\scriptscriptstyle\bullet  }(T^{*}X)$, which is 
multiplication by $p$ on $\Lambda^{p}(T^{*}X)$. 

\begin{thm}\label{thmMS}
If $\delta(G)\g 2$, for any  semisimple element $\gamma\in G$, 
\begin{align}
	\Trs^{[\gamma]}\[N^{\Lambda^{\scriptscriptstyle\bullet 
	}(T^{*}X)}\exp\(-t 
	C^{\fg,X}\)\]=0. 
\end{align} 
\end{thm} 
We have a generalisation of \cite[Corollary 2.2]{MStorsion} and \cite[Theorem 
7.9.3]{B09}. 

\begin{cor}\label{corMS0}
	Let $F$ be a 
flat vector bundle on $Z$ with holonomy $\rho$.  Assume that $\dim Z$ 
is odd and that $\delta(G)\neq 1$. Then, for $t>0$, we have 
\begin{align}\label{eqMS0}
	\Trs\[N^{\Lambda^{\scriptscriptstyle\bullet 
	}(T^{*}Z)}\exp\(-t \Box^{Z}\)\]=0.
\end{align} 
In particular, 
\begin{align}\label{eqMS01}
\prod_{i=1}^{m}	{\rm det}^{*}\(\Box^{Z}|_{\Omega^{i}(Z,F)} 
\)^{(-1)^ii}=1. 
\end{align} 
For any acyclic and unitary representation $\rho_{0}$ of $\Gamma$, 
there is $\e>0$ such that if $\rho$ is $C^{0}$ $\e$-close  to $\rho_{0}$, then 
\begin{align}\label{eqMS02}
	T_{\rm CM}(F)=1.
\end{align} 
\end{cor}
\begin{proof}
	Since $m$ is odd,  $\delta(G)$ is odd. Since $\delta(G)\neq 1$, 
	we have $\delta(G)\g 3$. By Theorems \ref{thmselnon} and 
	\ref{thmMS}, we get \eqref{eqMS0} and \eqref{eqMS01}. By Proposition 	\ref{propclosesign}, we get \eqref{eqMS02}. 
\end{proof} 


\section{Fundamental Cartan subalgebra and related 
constructions}\label{Sfondcs}
The purpose of this section is to give a simple new conceptual proof 
for a complexified version of \cite[Theorem 6.11]{Shfried}.  There, 
the corresponding  proof is based on the classification theory of real simple Lie algebras. 

This section is organised as follows. 
In Section \ref{sCartan}, we introduce a $\theta$-invariant fundamental Cartan 
subalgebra $\fh=\fb\oplus \ft$.

In Section \ref{sslpg}, we introduce a splitting of $\fg$ according to the 
action of $\fb$. 

In Section \ref{sRoot}, we introduce a root system of $\fg$ with respect 
to the fundamental Cartan subalgebra $\fh$. We reinterpret some 
objets constructed in Section \ref{sslpg}.

Finally, in Section \ref{skey}, we give a simple new conceptual proof  for a 
complexified version of \cite[Theorem 6.11]{Shfried}.


\subsection{A fundamental Cartan subalgebra}\label{sCartan}
Let $T\subset K$ be a maximal torus of $K$. Let $\ft\subset \fk$ be the Lie 
algebra of $T$. Set
\begin{align}\label{eqdefb}
	\fb=\{a\in \fp: [a,\ft]=0\}. 
\end{align}
Put
\begin{align}\label{eqh=b+t}
	\fh=\fb\oplus \ft.
\end{align}
By \cite[p.~129]{Knappsemi},  $\fh$ is a  Cartan subalgebra of $\fg$. Let $H\subset G$ be the associated 
Cartan subgroup, that is the centraliser of $\fh$ in $G$. By 
\cite[Theorem 5.22]{Knappsemi}, $H$ is a connected abelian reductive 
subgroup of $G$, so that 
\begin{align}
	H= \exp(\fb) \times T. 
\end{align}
We will call $\fh$ and $H$ the fundamental Cartan subalgebra of $\fg$ 
and the  fundamental Cartan subgroup of $G$. 


Since $\fh_{\bC}=\fh\otimes_{\bR}\bC$ is a Cartan subalgebra of 
$\fg_{\bC}$, by  \eqref{eqd=g-k} and \eqref{eqh=b+t}, we have 
\begin{align}
	\delta(G)=\dim \fb. 
\end{align}

\subsection{A splitting of $\fg$}\label{sslpg}
Recall that $Z(\fb)\subset G$ is  the stabiliser of $\fb$ in $G$ with 
Lie algebra $\fz(\fb)\subset \fg$. By \cite[Proposition 
7.25]{KnappLie},  $Z(\fb)$ is a  possibly non connected reductive subgroup of 
$G$. Also, $\theta$ acts on $\fz(\fb)$ so that we have the Cartan 
 decomposition
 \begin{align}\label{eqzb0}
 	\fz(\fb)=\fp(\fb)\oplus \fk(\fb). 
 \end{align}

Let $\fm\subset 
\fz(\fb)$ be the orthogonal space (with respect to B) of $\fb$ in $\fz(\fb)$. Then $\fm$ 
is a Lie subalgebra of $\fg$, and  $\theta$ acts on $\fm$ so that 
 \begin{align}\label{eqmb0}
 	\fm=\fp_{\fm}\oplus \fk_{\fm}. 
 \end{align}
Let 
$M\subset G$ be the connected Lie group associated to the  Lie 
algebra $\fm$. By \cite[(3.3.11) and Theorem 3.3.1]{B09}, $M$ is  
closed in $G$ and is  a connected reductive subgroup of $G$ 
with maximal compact subgroup 
\begin{align}
	K_{M}=M\cap K.
\end{align}
Moreover, we have
\begin{align}\label{eqZ(b)}
&Z^{0}(\fb)=\exp(\fb)\times M,&\fz(\fb)=\fb\oplus \fm, &&	
\fp(\fb)=\fb\oplus\fp_{\fm},&&& \fk(\fb)=\fk_{\fm}. 
\end{align}
Since $\fh\subset \fz(\fb)$ is also a fundamental Cartan subalgebra 
of $\fz(\fb)$, we have
\begin{align}\label{eqdm=0}
	\delta(M)=0. 
\end{align}
Let 
\begin{align}
	X_{M}=M/K_{M}
\end{align}
be the associated  symmetric space. By \eqref{eqdm=0}, we see that  $\dim X_{M}$ 
is even.

Let $\fp^{\bot}(\fb),\fk^{\bot}(\fb),\fz^{\bot}(\fb)$ be respectively the 
orthogonal spaces (with respect to $B$) of $\fp(\fb),\fk(\fb),\fz(\fb)$ 
in $\fp,\fk,\fg$. Clearly,
\begin{align}
	\fz^{\bot}(\fb)=\fp^{\bot}(\fb)\oplus \fk^{\bot}(\fb). 
\end{align}
And also 
\begin{align}\label{eq:mpk1}
 & \fp=\fb\oplus\fp_\fm\oplus\fp^\bot(\fb),&\fk=\fk_\fm\oplus\fk^\bot(\fb),&& \fg=\fb\oplus \fm\oplus \fz^{\bot}(\fb).
\end{align}
The group $K_M$ acts trivially on $\fb$. It also acts on $\fp_\fm$, 
$\fp^\bot(\fb)$, $\fk_\fm$ and $\fk^{\bot}(\fb)$, and  preserves the 
splittings \eqref{eq:mpk1}. Similarly, the groups $M$ and $Z^{0}(\fb)$ act trivially on $\fb$, act  on $\fm,\fz^{\bot}(\fb)$, and 
preserves the third splitting in \eqref{eq:mpk1}.

\subsection{A root system of $(\fh,\fg)$}\label{sRoot}
Let $R\subset \fh_{\bC}^{*}$ be a root system of $(\fh,\fg)$ \cite[Section II.4]{KnappLie}. If 
$\alpha\in R$, let $\fg_{\alpha}\subset \fg_{\bC}$ be the weight space associated with  $\alpha$, which is of dimension $1$. Then we have the splitting
\begin{align}\label{eqgc}
	\fg_{\bC}=\fh_{\bC}\bigoplus \oplus_{\alpha\in R}\fg_{\alpha}. 
\end{align}

If $\alpha\in R$, then $\ol{\alpha}\in R$.
A root is called real if $\ol{\alpha}=\alpha$, imaginary if 
$\ol{\alpha}=-\alpha$, and complex otherwise. 
By \cite[Proposition 11.16]{KnappLie} (see also \cite[Proposition 
3.7]{BS19}), since $\fh$ is fundamental, there are no real roots. Let 
$R^{\rm im}\subset R$ and $R^{\rm c}\subset R$ be the subsets of imaginary and complex 
roots, so that 
\begin{align}
	R=R^{\rm im}\sqcup R^{\rm c}. 
\end{align}

Let $\fh^{\bot}$ be the orthogonal to $\fh$ in $\fg$ with respect to 
$B$.
Set
\begin{align}\label{eqihb0}
	\mathfrak i=\fz(\fb)\cap \fh^{\bot}. 
\end{align}
Let $\fc$ be the orthogonal to $\mathfrak i$ in $\fh^{\bot}$.
Then
\begin{align}\label{eqg=hic}
	\fg=\fh\oplus \mathfrak i\oplus \fc. 
\end{align}
By \eqref{eqgc}, we have
\begin{align}\label{eqiccc}
&	\fii_{\bC}=\oplus_{\alpha\in R^{\rm 
im}}\fg_{\alpha},&\fc_{\bC}=\oplus_{\alpha\in R^{\rm c}}\fg_{\alpha}.
\end{align}

Also, $\theta$ acts on $\mathfrak i, \fc$, so that we have the splittings
\begin{align}\label{eqKMic}
&	\fii=\fii_{\fp}\oplus \fii_{\fk},&\fc=\fc_{\fp}\oplus \fc_{\fk}. 
\end{align}
By \eqref{eqzb0}, \eqref{eqmb0}, \eqref{eq:mpk1}, \eqref{eqihb0}, and 
\eqref{eqg=hic}, we have 
\begin{align}\label{eqKMic2}
	&\fm=\ft\oplus \fii,&\fp_{\fm}=\fii_{\fp},&&\fk_{\fm}=\ft\oplus 
	\fii_{\fk},\\
&\fz^{\bot}(\fb)=\fc,&\fp^{\bot}(\fb)=\fc_{\fp},&&\fk^{\bot}(\fb)=\fc_{\fk}.\notag
\end{align}
In particular, we can rewrite the third identity of \eqref{eq:mpk1},
\begin{align}
	\fg=\fb\oplus \fm\oplus \fc.
\end{align}

\begin{prop}\label{prop38}
	The vector spaces $\fii_{\fp},\fii_\fk$ have even 
	dimensions, and $\fc_\fp,\fc_\fk$ have the same even dimension. 
	The  $K_{M}$-action  preserves the second splitting in  
	\eqref{eqKMic}, so that  the actions  on $\fc_\fp,\fc_\fk$ are equivalent.
\end{prop}
\begin{proof}
	This is \cite[Proposition 3.8]{BS19}, 	except in the last 
	statement the group 
 	$K_{M}$ is replaced by $T$.  
	By the consideration after 
	\eqref{eq:mpk1}, by the last two identities in \eqref{eqKMic2}, 
	$K_{M}$ acts on $\fc_{\fp},\fc_{\fk}$.  Since $K_{M}$ and $ 	
	\fb$ commutes, if $b\in \fb$ is such 
	that $\<\alpha,b\>\neq0$ for all $\alpha\in R^{\rm c}$, then 
	$\ad(b):\fc_{\fp}\to \fc_{\fk}$ defines  a  $K_{M}$-equivalent.  
%
\end{proof}

Let $R_{+}\subset R$ be a positive root system. Set
\begin{align}
&	R_{+}^{\rm im}=R_{+}\cap R^{\rm im},&R_{+}^{\rm c}=R_{+}\cap 
R^{\rm c}. 
\end{align}
As explained in \cite[Section 3.5]{BS19}, we can choose $R_{+}$ so 
that $R_{+}^{\rm c}$ is preserved by the complex conjugation.  

Set
\begin{align}
&	\fc_{+,\bC}=\oplus_{\alpha\in 
R^{c}_{+}}\fg_{\alpha},&\fc_{-,\bC}=\oplus_{\alpha\in 
R^{c}_{-}}\fg_{\alpha}. 
\end{align}

\begin{prop}\label{prop310}
	 	The following statements hold.
 	\begin{enumerate}[i)]
\item\label{eqp1}
The vector spaces $\fc_{+,\bC},\fc_{-,\bC}$ are the 
complexifications 
of real Lie subalgebras $\fc_{+},\fc_{-}$ of $\fg$, which have the same even dimension, and are such that
\begin{align}\label{eqtcpm}
&	\fc=\fc_{+}\oplus \fc_{-},&\fc_{+}=\theta\fc_{-}. 
\end{align}
\item\label{eqp2} The bilinear form  $B$ vanishes on $\fc_+,\fc_-$ and induces the 
identification,
\begin{align}\label{eqBcpm}
\fc_{-}^{*} \simeq  \fc_+.
\end{align}
\item\label{eqp3} The group $Z^{0}(\fb)$ acts on $\fc_{\pm}$, so that 
\eqref{eqBcpm} is $Z^{0}(\fb)$-equivalent. 
\item\label{eqp4} The actions of $M$ on $\fc_{\pm}$ are equivalent.  
\item\label{eqp5} The projections on $\fp,\fk$ map $\fc_{\pm}$ into $\fc_\fp,\fc_\fk$ 
isomorphically. 
\item\label{eqp6} 
Finally, the actions of $K_{M}$ on $\fc_+, \fc_-, 
\fc_\fp, \fc_\fk$ are equivalent.
\end{enumerate}
\end{prop}
\begin{proof}
The statements \ref{eqp1}), \ref{eqp2}), \ref{eqp5}) are just \cite[Proposition 
3.10]{BS19}, and  \ref{eqp6}) has been established  for  the 
$T$-action  instead of the $K_{M}$-action . 



By our choice of the positive root system, we have
\begin{align}
[\fii,\fc_{\pm,\bC}]\subset \fc_{\pm,\bC}.
\end{align}
Therefore, $\fz(\fb)=\fh\oplus \fii$ preserves $\fc_{\pm,\bC}$, and $Z^{0}(\fb)$ 
acts on  $\fc_{\pm,\bC}$. Since the $Z^{0}(\fb)$-action on $\fc_{\pm,\bC}$ commutes with the complex conjugation,  $Z^{0}(\fb)$ acts on $\fc_{\pm}$.  
Since $B$ is $Z^{0}(\fb)$-invariant, we see that \eqref{eqBcpm} is 
$Z^{0}(\fb)$-equivalent, from which we get  \ref{eqp3}).

Since $\delta(M)=0$, 
by \cite[Probleme XII.14]{Knappsemi}, there is $k_{0}\in K_{M}$ such that 
$\Ad(k_{0})|_{\fm}=\theta|_{\fm}$. 
By the second identity 
of \eqref{eqtcpm}, $\Ad(k_{0})\theta:\fc_{+}\to \fc_{-}$ is an 
equivalence of $M$-representations, from which  we get \ref{eqp4}).

Since the  projection $\fc_{\pm}\to \fc_{\fp}$ is $(1+\theta)/2$, 
since $K_{M}$ is fixed by $\theta$,  we get \ref{eqp6}). 
\end{proof}

\begin{cor}\label{propVAM}
	For $0\l j\l \dim \fc_{\pm}$,  we have isomorphisms of 
	representations of $M$,
\begin{align}\label{eqVAMs1}
	\Lambda^{j}(\fc_{\pm}^{*})\simeq \Lambda^{\dim 
	\fc_{\pm}-j}(\fc_{\pm}^{*}). 
\end{align}
\end{cor}
\begin{proof}
	Since $\delta(M)=0$, $M$ has a compact center. By Proposition  
	\ref{propreal1t}, $M$ acts trivially on $\Lambda^{\dim 
	\fc_{\pm}}(\fc^{*}_{\pm})$. 
	We have an isomorphism of real representations of $M$,  
	\begin{align}\label{eqetaljlj}
	\Lambda^{j}(\fc^{*}_{\pm})\simeq 
	\Lambda^{\dim\fc_{\pm}-j}(\fc_{\pm}). 
\end{align}
By  Proposition \ref{prop310} \ref{eqp3}) \ref{eqp4}) and \eqref{eqetaljlj}, we get  \eqref{eqVAMs1}. 
\end{proof}

\subsection{A lifting property}\label{skey}
Let $R(K)$ be the representation ring of $K$. We can identify $R(K)$ 
with the subring of the $\Ad(K)$-invariant smooth functions on $K$ which is generated by the 
characters of finite dimensional complex representations of $K$. 

The restriction 
induces  an injective  morphism of rings
\begin{align}\label{eqiKTunique}
i:	R(K)\to R(T). 
\end{align} 
Let $W(T:K)=N_{K}(T)/T$ be the  Weyl group of $K$, where $N_{K}(T)$ 
is the normaliser of  $T$ in $K$. Then 
$W(T:K)$ acts on $R(T)$. By \cite[Proposition VI.2.1]{BrockerDieck}, 
$i$ induces an isomorphism of rings
\begin{align}\label{eqRKRT}
	i: R(K)\simeq  R(T)^{W(T:K)}. 
\end{align}

\begin{prop}\label{propNp}
	The adjoint action of $N_{K}(T)$ preserves the decomposition 
	\begin{align}
	\fg=\ft\oplus \fb\oplus \mathfrak i_{\fk}\oplus \mathfrak i_{\fp}\oplus \fc_{\fk}\oplus 
	\fc_{\fp}. 
\end{align} 
\end{prop}
\begin{proof}
	Since $N_{K}(T)$ preserves $\ft$, by \eqref{eqdefb}, $N_{K}(T)$ 
	preserves $\fb$. Also, $N_{K}(T)$ preserves $\fz(\fb)$. Since $N_{K}(T)\subset K$ preserves $B$, by 
	\eqref{eqihb0}, $N_{K}(T)$ preserves $\mathfrak  i$. Using again 
	$N_{K}(T)$ preserves $B$, by \eqref{eqg=hic}, $N_{K}(T)$ 
	preserves $\mathfrak c$. By \eqref{eqKMic}, $N_{K}(T)$ preserves 	$\mathfrak i_{\fp},\mathfrak i_{\fk}, \mathfrak c_{\fp},\mathfrak 
	c_{\fk}$.
\end{proof}


\begin{thm}\label{thmlift}
	The adjoint action  of $T$ on $\fii_{\fp,\bC}$, $\fii_{\fk,\bC}$, 
	$\fc_{\fp,\bC}\simeq\fc_{\fk,\bC}\simeq 
	\fc_{+,\bC}\simeq\fc_{-,\bC}$ lift uniquely to virtual representations in $R(K)$. 
\end{thm}
\begin{proof} 
	By Proposition \ref{propNp},  
the characters of $T$ on $\fii_{\fp,\bC}$, $\fii_{\fk,\bC}$, and $\fc_{\fp,\bC}$ are 
$W(T:K)$-invariant. By \eqref{eqiKTunique} and \eqref{eqRKRT}, the $T$-actions on 
$\fii_{\fp,\bC}$, $\fii_{\fk,\bC}$, and $\fc_{\fp,\bC}$ lift uniquely to 
$R(K)$.
\end{proof}
	
\begin{cor}\label{corkey}
	For $i, j \in \mathbf N$, the adjoint representations of $K_M$ on 
$\Lambda^{i}(\fp_{\fm,\bC}^{*})$ and $\Lambda^{j}(\fc_{\pm,\bC}^{*})$ have unique 
lifts in $R(K)$. 
\end{cor}
\begin{proof}
This is a consequence of	 \eqref{eqKMic2} and Theorem \ref{thmlift}, 
%
\end{proof}

\begin{re}
Let $RO(K)$ be the real representation ring of $K$. By 
\cite[Proposition II.7.8]{BrockerDieck}, the complexification  $V\in RO(K)\to V\otimes_{\bR}\bC\in R(K)$ induces an injective 
morphism of rings, 
\begin{align}
	RO(K)\to R(K). 
\end{align}
	In \cite[Theorem 6.11]{Shfried}, using the classification theory on the real 
	simple Lie algebras, we have shown that when $\delta(G)=1$, the 
	real representation $\fc_{\pm}$ of $K_{M}$ has a unique lift in $RO(K)$. 
	However, the above complexified version is enough for applications 	both in \cite{Shfried} and the current paper. 
\end{re}

\section{The Ruelle dynamical  zeta functions   for arbitrary 
twist}\label{SRuelle}


The purpose of this section is to introduce the Ruelle dynamical zeta 
function for arbitrary twist.  We restate the main result  of this 
article as Theorem \ref{thmRmeo}. Its proof will be given in Sections 
\ref{Sselbergzeta} and \ref{S:rep}.  
Theorem \ref{thmRmeo} generalises author's previous result 
\cite[Theorem 1.1]{Shfried}  
where $\rho$ is unitary, as well as the results on hyperbolic 
manifolds due to  M\"uller 
\cite[Theorem 1.1, Proposition 1.3]{Muller20} and  Spilioti  \cite[Theorem 1.2]{Spilioti18}, \cite[Theorem 2]{Spilioti20}.

We use the notation in Section \ref{Sselberg}. 
Recall that $Z=\Gamma\backslash G/K$ is a locally symmetric manifold 
and $\rho:\Gamma\to \GL_{r}(\bC)$ is a representation of $\Gamma$. 


Let us recall the definition of the Ruelle dynamical zeta function introduced by 
Fried \cite[Section 5]{Friedconj}. For $[\gamma]\in [\Gamma]$, recall 
that $B_{[\gamma]}$ (see \eqref{eqBgamma}) is the locally 
symmetric space  $\Gamma(\gamma)\backslash X(\gamma)$. 
By \cite[Proposition 5.15]{DuistermaatKolkVaradarajan}, the set of  
 nontrivial closed geodesics  on  $Z$ consists of a disjoint union 
\begin{align}\label{DKVB}
\coprod_{[\gamma]\in [\Gamma_{+}]}B_{[\gamma]}.
\end{align}
Moreover, if $[\gamma]\in [\Gamma_{+}]$, all the  elements of 
$B_{[\gamma]}$ have the same length $\ell_{[\gamma]}>0$.
\index{B@$B_{[\gamma]}$} 

If $[\gamma]\in [\Gamma_{+}]$,  the geodesic flow induces a  locally free action of 
$\mathbb{S}^1$ on $B_{[\gamma]}$, so that $ 
B_{[\gamma]}/ \mathbb{S}^1$ is a closed orbifold. Let $\chi_{\rm 
		orb}\(B_{[\gamma]}/\mathbb{S}^{1}\)\in \mathbf{Q}$ be the 
		orbifold Euler characteristic number \cite{SatakeGaussB}. We 
		refer the reader to \cite[Proposition 5.1]{Shfried} for an explicit formula for $\chi_{\rm 
		orb}\(B_{[\gamma]}/\mathbb{S}^{1}\)$. In particular,  
if $\delta(G)\g 2$, or if $\delta(G)=1$ and $\gamma$ can not be 
conjugate by an element of $G$ into the fundamental Cartan subgroup $H$, then  
\begin{align}\label{eqchi0}
\chi_{\rm 
		orb}\(B_{[\gamma]}/\mathbb{S}^{1}\)=0.
\end{align}

The $\mathbb{S}^1$-action on $B_{[\gamma]}$ is not necessarily effective. Let
\begin{align}
m_{[\gamma]}=\left|\ker\big(\bbS^1\to {\rm 
Diff}(B_{[\gamma]})\big)\right|\in \mathbf{N}^{*}
\end{align}
be the generic multiplicity.\index{M@$m_{[\gamma]}$}

By \cite[Theorem 5.6]{Shfried} and by \eqref{eqgno}, there is 
$\sigma_{0}>0$ such that
\begin{align}
	\sum_{[\gamma]\in [\Gamma_{+}]}\frac{\left|\chi_{\rm 
		orb}\(B_{[\gamma]}/\mathbb{S}^{1}\)\right|}{m_{[\gamma]}} 
		\left|\Tr\[\rho(\gamma)\]\right| e^{-\sigma_{0} 
		\ell_{[\gamma]}}<\infty. 
\end{align}

\begin{defin}
	For $\Re(\sigma)\g \sigma_{0}$, set 
	\begin{align}\label{defRrho}
		R_{\rho}(\sigma)=\exp\(\sum_{[\gamma]\in 
		[\Gamma_{+}]}\frac{\chi_{\rm 
		orb}\(B_{[\gamma]}/\mathbb{S}^{1}\)}{m_{[\gamma]}}\Tr\[\rho(\gamma)\]e^{-\sigma \ell_{[\gamma]}}\).
	\end{align}
\end{defin}

\begin{re}\label{reR=1}
	By \eqref{eqchi0}, if $\delta(G)\g 2$, the 
	dynamical zeta function $R_\rho(\sigma)$ is the constant function 
	$1$. Moreover, if $\delta(G)=1$, then the sum on the right-hand 	side of \eqref{defRrho} can be reduced to 
	 a sum over 	$[\gamma]\in [\Gamma_{+}]$ such that $\gamma$ can be conjugate into $H$. 
\end{re}

Recall that $\Box^{Z}$ is the flat   Laplacian on $Z$.  We restate 
Theorem \ref{Thm1} as follows.

\begin{thm}\label{thmRmeo}
	Assume that $\dim Z$ is odd.
	The following statements hold.  
	\begin{enumerate}[i)]
	\item\label{thm1i}  The dynamical zeta function 
	$R_{\rho}(\sigma)$ has a meromorphic extension to $\sigma\in 
	\bC$.  

	\item\label{thm1ii}   There exist explicit constants $C_{\rho}\in 
	\bR^{*}$ and $r_{\rho}\in \mathbf Z$ (see 
 \eqref{eqCr111} and \eqref{eqCr1111}) such that when $\sigma\to 0$, we 
 have 
\begin{align}\label{eqFcomtorsion}
R_{\rho}(\sigma)=C_{\rho}\left\{\prod_{i=1}^{m}\({\rm 
det}^{*}\(\Box^{Z}|_{\Omega^{i}(Z,F)}\)\)^{(-1)^{i}i}\right\}\sigma^{r_{\rho}}+\cO(\sigma^{r_{\rho}+1}). 	
\end{align} 

\item\label{thm1iii} There is $k\in \mathbf N$ such that  for any   acyclic and unitary 	
representation $\rho_{0}$ of $\Gamma$, there exists  $\e>0$ such that	 if $\rho$ is $C^{k}$ $\e$-close to $\rho_{0}$, then 	$\Box^{Z}$  is invertible and 
\begin{align}\label{eqdCt11}
&C_{\rho}=1,&r_{\rho}=0,
\end{align} 
 so that 
\begin{align}
	R_{\rho}(0)=T_{\rm CM}(F). 
\end{align} 
\end{enumerate} 
\end{thm}
\begin{proof}
	Since $m$ and $\delta(G)$ have the same parity, we have 
	$\delta(G)\g1$ is odd. When $\delta(G)\g 3$, 	our theorem with 
	$k=0$ follows from  Corollary \ref{corMS0} and Remark 
	\ref{reR=1}. If $\delta(G)=1$, the proofs  of \ref{thm1i}),  
	\ref{thm1ii}) as well as the proof of \ref{thm1iii}) when $Z_{G}$ 
	is noncompact  are based on the 
	introducing of  the 	Selberg zeta functions with arbitrary 
	twist and will be given in Section  \ref{sproofthmRmeo}.  	The 
	proof of  \ref{thm1iii}) when $Z_{G}$ is compact will be given in Section \ref{secnonbetti}. 
\end{proof}

Let $\ol{R}_{\rho}$ be the meromorphic function 
defined for $\sigma\in \bC$ by
\begin{align}
	\ol{R}_{\rho}(\sigma)=\ol{R_{\rho}(\ol{\sigma})}.
\end{align}

\begin{prop}\label{propRrhob}
The following identities of meromorphic functions on $\bC$ hold, 
	\begin{align}\label{eqRrho}
		&R_{\rho^{*}}(\sigma)=R_{\rho}(\sigma),
&R_{\ol{\rho}}(\sigma)=	\ol{R}_{\rho}(\sigma).
	\end{align}
\end{prop}
\begin{proof}The second identity in \eqref{eqRrho} is trivial. Let us 
	show the first one. 	We have  
	\begin{align}\label{eqTrol}
	\Tr\[\rho^{*}(\gamma)\]=\Tr\[\rho\(\gamma^{-1}\)\].
\end{align}
Note that $\gamma\to 
\gamma^{-1}$ induces a bijection on $[\Gamma_{+}]$. Since 
$d_{\gamma^{-1}}(x)=d_{\gamma}(x)$, we have
\begin{align}\label{eqlglg1}
&\ell_{\[\gamma^{-1}\]}=\ell_{[\gamma]},&X\(\gamma^{-1}\)=X(\gamma),
\end{align} 
By \eqref{eqlglg1}, using $ \Gamma\(\gamma^{-1}\)=\Gamma(\gamma)$, we 
get \begin{align}\label{eqlglg12}
&	B_{[\gamma^{-1}]}=B_{[\gamma]},&m_{[\gamma^{-1}]}=m_{[\gamma]}. 
\end{align} 
By \eqref{defRrho} and  
\eqref{eqTrol}-\eqref{eqlglg12}, we get the first identity in 
\eqref{eqRrho}. 
\end{proof}

\section{The Selberg zeta function for arbitrary 
twist}\label{Sselbergzeta}
The purpose of this section is to extend the results of \cite[Theorems 7.6 and 7.7]{Shfried}  on the zeta functions of Ruelle 
and Selberg to arbitrary twist with the help of the M\"uller's  
Selberg trace formula (Theorem \ref{thmselnon}), inexplicitly of Bismut's 
orbital integral formula \cite[Theorem 6.1.1]{B09}, and of the lifting 
properties (Theorem \ref{thmlift}, Corollary \ref{corkey}). 

More precisely,   in \cite[Section 7]{Shfried}, if $\delta(G)=1$, we 
associate a Selberg zeta function $Z_{\eta,\rho}$ to  a representation  $\eta$ of $M$ satisfying \cite[Assumption 7.1]{Shfried} and to a unitary representation $\rho$ of $\Gamma$. Moreover, we show that $Z_{\eta,\rho}$ has a  
meromorphic extension to $\bC$ and satisfies a functional 
equation. Also, we prove that the Ruelle zeta function with a unitary twist is an alternating product of 
certain Selberg zeta functions. In this way, we obtained the 
meromorphic extension of the Ruelle zeta function. 
In this section, we extend all the above results to arbitrary $\rho$ 
and to a slightly larger class of $\eta$.


This section is organised as follows. In Section \ref{srank1}, we recall some results 
on the structure of real reductive groups with  $\delta(G)=1$ 
obtained in \cite[Section 6]{Shfried}.
%

In Sections \ref{sSelbergzeta} and \ref{spthm12}, when $\delta(G)=1$, we introduce a 
class of virtual representations $\eta$ of the group $M$. If $\rho:\Gamma\to 
\GL_{r}(\bC)$ is a representation of $\Gamma$, we introduce the Selberg zeta function $Z_{\eta,\rho}$ associated to $(\eta,\rho)$. We show the 
meromorphic extension and a functional equation for $Z_{\eta,\rho}$. 

Finally, in Section \ref{sproofthmRmeo}, we show that $R_{\rho}$ 
is an alternating product of certain Selberg zeta functions. In 
particular, we show Theorem \ref{thmRmeo} \ref{thm1i})    
\ref{thm1ii}), as well as \ref{thm1iii}) when $Z_{G}$ is non compact. 

In the whole section, we assume $\delta(G)=1$.

\subsection{The structure of the reductive group $G$ with 
$\delta(G)=1$}\label{srank1}
 We use the notation in Sections \ref{Sselberg} and  
\ref{Sfondcs}.   Assume that $\delta(G)=1$. To make this section 
readable, instead of writing  $\fc_{\pm}$, we use  notation as 
\cite{Shfried}, i.e.,
\begin{align}
&\fn=\fc_{+},&\ol{\fn}=\fc_{-}.
\end{align}

When $G$ has a non compact centre, thanks to the following proposition, $G$ has a very simple structure. 

\begin{prop}\label{propGbm}
	If $G$ has a non compact centre, then \begin{align}\label{eqGbM}
	&G=\exp(\fb)\times M,& \fz^{\bot}(\fb)=0. 
\end{align}
\end{prop}
\begin{proof}
Since $G$ has a non compact centre, we have $\dim \fz_{\fp}\g 1$. 	Since 
$\fz_{\fp}\subset \fb$ and since $\dim \fb=1$, we have 
$\fz_{\fp}=\fb$. So $\fg=\fz(\fb)$. By \eqref{eqZ(b)}, we get 
\eqref{eqGbM}. 
\end{proof}

Assume in the rest of this subsection that   $G$ has a compact 
centre. By Propositions \ref{prop38} and  \ref{prop310} 
\ref{eqp1}), $\dim \fn$ is even. Set
\begin{align}
  \ell=\frac{1}{2}\dim \fn.\index{L@$\ell$}
\end{align}
 Since $G$ has a compact centre, we have $\fb\not\subset \fz_{\fg}$. 
Therefore,  $\fz^{\bot}(\fb)\neq0$ and $\ell>0$. Recall that we have fixed a system of positive roots $R_{+}$. 

\begin{prop}\label{prop:nnam1}
There is $\alpha_{0}\in \fb^{*}$ 
such that for all $\alpha\in R_{+}^{\rm c}$, 
\begin{align}
	\alpha|_{\fb}=\alpha_{0}. 
\end{align}
Equivalently,	 $\fb$ acts on $\fn$ and $\ol{\fn}$ by 
$\pm\alpha_{0}\in \fb^{*}$, i.e., for $a\in 
\fb$, $f\in \fn$, $\ol{f}\in \ol{\fn}$, we have
\begin{align}\label{eq:ab1}
  &[a,f]=\<\alpha_{0},a\>f, &[a,\ol{f}]=-\<\alpha_{0},a\>\ol{f}.
\end{align}
In particular,  
\begin{align}\label{eq:nnam}
&[\fn,\ol{\fn}]\subset \fz(\fb),&[\fn,\fn]=\[\ol{\fn},\ol{\fn}\]=0,
\end{align}
and 
\begin{align}\label{eq:fzbot1}
&\[\fz(\fb),\fz(\fb)\]\subset 
\fz(\fb),&\[\fz(\fb),\fz^\bot(\fb)\]\subset 
\fz^{\bot}(\fb),&&\[\fz^{\bot}(\fb),\fz^\bot(\fb)\]\subset \fz(\fb).
\end{align}
\end{prop}
\begin{proof}
	This is  \cite[Propositions 6.2 and 6.3]{Shfried}. 
\end{proof}



%
Let $\fu(\fb)\subset \fu,\fu_{\fm}\subset \fu$ be the compact forms of $\fz(\fb)$ and 
$\fm$. Then 
\begin{align}
&\fu(\fb)=\sqrt{-1}\fb\oplus \fu_{\fm},&\fu_{\fm}= 
\sqrt{-1}\fp_{\fm}\oplus \fk_{\fm}. 
\end{align}
Let 
$\fu^{\bot}(\fb)$ be the orthogonal space of $\fu(\fb)$ in $\fu$ with 
respect to $-B|_{\fu}$. Then,
\begin{align}
\fu^{\bot}(\fb)= 
	\sqrt{-1}\fp^{\bot}(\fb)\oplus \fk^{\bot}(\fb). 
\end{align}
By \eqref{eq:fzbot1}, we have
\begin{align}
  \label{eq:fubot}
&\[\fu(\fb),\fu(\fb)\]\subset \fu(\fb),& \[\fu(\fb),\fu^\bot(\fb)\]\subset \fu^{\bot}(\fb),&&  \[\fu^{\bot}(\fb),\fu^\bot(\fb)\]\subset \fu(\fb).
\end{align}
Thus, $(\fu,\fu(\fb))$ is a compact symmetric pair. 

Let $U(\fb)\subset U$, $U_{M}\subset U$, $A_{0}\subset U$ be the connected subgroups 
of $U$ associated to the Lie subalgebra 
$\fu(\fb),\fu_{\fm},\sqrt{-1}\fb$ of $\fu$, so that 
	\begin{align}
		U(\fb)=A_{0}U_{M}. 
	\end{align} 
\begin{re} All of the above three groups are compact. Indeed, $U(\fb)$ is 
the stabiliser of $\sqrt{-1}\fb$ in $U$, which is compact in $U$. 
Since $\delta(M)=0$, $M$ has compact centre. Therefore, $U_{M}$ is 
the compact form of $M$. The compactness of $A_{0}$ is established in \cite[Proposition 6.6]{Shfried}
\end{re}

\subsection{The Selberg zeta function}
\label{sSelbergzeta} 
Recall    that $K_{M}$ and $K$ have the maximal torus $T$.  By \eqref{eqRKRT}, the restriction to $K_{M}$ induces an 
injective morphism of rings  
$R(K)\to R({K_{M}})$. 

\begin{ass}\label{ass}
	Assume that $\eta=\eta^{+}-\eta^{-}$ is a virtual  $M$-representation  
	on the   finite 	dimensional complex vector space 
	$E_{\eta}=E_{\eta}^{+}-E_{\eta}^{+}$ 	such that 
\begin{enumerate}
	\item  $\eta|_{K_{M}}=\eta^{+}|_{K_{M}}-\eta^{-}|_{K_{M}}\in R(K_{M})$ has a unique lift in $R(K)$.
	\item  the Casimir $C^{\fu_{\fm}}$ of $\fu_{\fm}$ acts on 
	$\eta^{\pm}$  by the same 	scalar $C^{\fu_{\fm},\eta}\in \bR$. 
\end{enumerate}
\end{ass}

\begin{re}
	If $Z_G$ is non compact, then  $K_{M}=K$.  Assumption (1) is 
	empty. 
\end{re}

\begin{re}\label{re66}
	If $Z_G$ is compact, Assumption \ref{ass}  is slightly different from	
	\cite[Assumption 7.1]{Shfried}. 
	Indeed, Assumption \ref{ass} is  the complexified  virtual 
	version \cite[Assumption 
7.1 (1) (3)]{Shfried}. Moreover, the statement (2) in  \cite[Assumption 7.1]{Shfried}, which requires 
the $\fu_{\fm}$-action on $\eta$ lifts 
to $U_{M}$, is a 
technique condition. Since $\delta(M)=0$, $M$ has a compact centre. 
By Remark \ref{reweyltrick},  the $\fu_{\fm}$-action on $\eta$ lifts 
to a finite cover of $U_{M}$. 
 Thanks to this observation, most of arguments (see also Remark 
 \ref{re69} and Section \ref{ssdg1nc}) go though up to  evident  
 modifications. 


%
\end{re}

Recall that $\rho:\Gamma\to \GL_{r}(\mathbf{C})$ is a representation 
of $\Gamma$. Following \cite[Definition 7.4]{Shfried}, let us define  the 
Selberg zeta function associated to the pair $(\eta,\rho)$.
Recall 
that $H=\exp(\fb)\times T$ is the fundamental Cartan subgroup of $G$. 
For $e^{a}k^{-1}\in H$, we write $\gamma\sim e^{a}k^{-1}$ if there is $g_{\gamma}\in G$ such 
that $\gamma=g_{\gamma}e^{a}k^{-1}g_{\gamma}^{-1}$. 
By \cite[(7-62)]{Shfried} and by \eqref{eqgno}, there is 
$\sigma_{1}>0$ such that
\begin{align}
	\sum_{\tiny\substack{[\gamma]\in 
		[\Gamma_{+}]\\ \gamma\sim e^{a}k^{-1}\in H}}\frac{\left|\chi_{\rm 
		orb}\(B_{[\gamma]}/\mathbb{S}^{1}\)\right|}{m_{[\gamma]}} 
		\big|\Tr\[\rho(\gamma)\]\big| \frac{e^{-\sigma_{1} 
		\ell_{[\gamma]}}}{\left|\det(1-\Ad(e^{a}k^{-1}))|_{\fz^{\bot}(\fb)}\right|^{1/2}}<\infty. 
\end{align}

\begin{defin}
		For $\Re(\sigma)\g \sigma_{1}$, set 
	\begin{align}\label{eqdefsel}
		Z_{\eta,\rho}(\sigma)=\exp\Bigg(-\sum_{\tiny\substack{[\gamma]\in 
		[\Gamma_{+}]\\ \gamma\sim e^{a}k^{-1}\in H}}\frac{\chi_{\rm 
		orb}(B_{[\gamma]}/\mathbb{S}^{1})}{m_{[\gamma]}}\Tr\[\rho(\gamma)\]\frac{\Trs^{E_{\eta}}[k^{-1}]}{\left|\det(1-\Ad(e^{a}k^{-1}))|_{\fz^{\bot}(\fb)}\right|^{1/2}}e^{-\sigma \ell_{[\gamma]}}\Bigg).
	\end{align}
\end{defin}

Recall that by Corollary \ref{corkey},  
$\Lambda^{\scriptscriptstyle\bullet }(\fp_{\fm,\bC}^{*}) $ has a unique 
lift
in $R(K)$.  
\begin{defin}
Let
$\widehat{\eta}\in R(K)$ be the unique 
virtual representation of $K$ on $E_{\widehat{\eta}}=E^{+}_{\widehat\eta}-E^{-}_{\widehat\eta}$
such that the following identity in $R(K_M)$  holds, \begin{align}\label{eqetahat}
	E_{\widehat{\eta}}|_{K_{M}}=\Lambda^{\scriptscriptstyle\bullet }(\fp_{\fm,\bC}^{*})\widehat{\otimes} E_\eta|_{K_{M}} \in R(K_{M}). 
\end{align}
\end{defin}

Let $C^{\fg,Z,\widehat{\eta},\rho}$ be the generalised Laplacian 
acting on $C^{\infty}(Z,\cF_{\widehat{\eta}}\otimes F)$ introduced 
after \eqref{eqC316}. For $\lambda\in \bC$, set 
\begin{align}
	m_{\eta,\rho}(\lambda)=\dim \ker 
	\(C^{\fg,Z,\widehat{\eta}^{+},\rho}-\lambda\)^{N}-\dim \ker 
	\(C^{\fg,Z,\widehat{\eta}^{-},\rho}-\lambda\)^{N},
\end{align}
where $N\gg1$. When $\lambda=0$, set 
	\begin{align}\label{eqretarho}
		r_{\eta,\rho}=m_{\eta,\rho}(0). 
	\end{align} 

Let 
\begin{align}\label{eqdetCc}
\mathrm{det}_{\rm 
  gr}\(C^{\fg,Z,\widehat{\eta},\rho}+\sigma\)=\frac{\mathrm{det}\big(C^{\fg,Z,\widehat{\eta}^{+},\rho}+\sigma\big)}{\mathrm{det}\big(C^{\fg,Z,\widehat{\eta}^{-},\rho}+\sigma\big)}
 \end{align} 
  be a 
  graded determinant of $C^{\fg,Z,\widehat{\eta},\rho}+\sigma$.  By 
  Theorem \ref{thmGLH}, \eqref{eqdetCc} is a 
  meromorphic function on $\sigma\in \bC$. Its zeros and poles belong to the 
  set $\{-\lambda: \lambda\in \Sp( C^{\fg,Z,\widehat{\eta},\rho})\}$. 
  If $\lambda\in \Sp (C^{\fg,Z,\widehat{\eta},\rho})$, the order of 
  the zero  at $\sigma=-\lambda$ is $m_{\eta,\rho}(\lambda)$.

Following in \cite[(7-60)]{Shfried}, set
\begin{align}\label{eqseta}
	\sigma_\eta=\frac{1}{8}\Tr^{\fu^\bot(\fb)}\[C^{\fu(\fb),\fu^{\bot}(\fb)}\]-C^{\fu_{\fm},\eta}. 
\end{align}
Let $P_{\eta}(\sigma)$ be the odd polynomial  defined in 
\cite[(7-61)]{Shfried}. 
When $G$ has non compact centre, we have $\fu^{\bot}(\fb)=0$, so  
\begin{align}\label{eqsen1}
	\sigma_\eta=-C^{\fu_{\fm},\eta}, 
\end{align}
and the polynomial $P_{\eta}$ is given 
by
\begin{align}
	P_{\eta}(\sigma)=-\(\dim E_{\eta}^{+}-\dim E_{\eta}^{-}\)\[e\(TX_{M},\nabla^{TX_{M}}\)\]^{\max}\sigma.
\end{align}

\begin{re}\label{re69}
In \cite[(7-61)]{Shfried}, we assume the group $U_{M}$ acts on 
$\eta$. This action  extends to $U(\fb)=A_{0}U_{M}$ by requiring $A_{0}$ acts trivially on $\eta$. To define 
$P_{\eta}(\sigma)$, we use  characteristic forms of the homogenous 
vector bundle $\cF_{\fb,\eta}=U\times_{U(\fb)} E_{\eta}$ on 
$U/U(\fb)$.  In current  situation,  the $\fu_{\fm}$-action on 
$E_\eta$ does not necessarily lift to $U_{M}$.  So the vector bundle $\cF_{\fb,\eta}$ is not well defined globally on $U/U(\fb)$.  However, it is well defined in a neighborhood of  $[e]\in 
U/U(\fb)$, so the 
right-hand sides of \cite[(7-8) (7-61)]{Shfried} are still  well 
defined. In particular, $P_{\eta}(\sigma)$ is still well defined. 
\end{re}

We have a generalisation of  \cite[Theorem 7.6]{Shfried}.

\begin{thm}\label{thm:detfor}
	The Selberg zeta function $Z_{\eta,\rho}(\sigma)$ has a 
	meromorphic extension to $\sigma\in \bC$ such that  the following identity of meromorphic functions on $\bC$ holds, 
\begin{align}\label{eq:detfor}
  Z_{\eta,\rho}(\sigma)=\mathrm{det}_{\rm 
  gr}\(C^{\fg,Z,\widehat{\eta},\rho}+\sigma_{\eta}+\sigma^2\)\exp\big(r\vol(Z) P_{\eta}(\sigma)\big).
\end{align}
The zeros and poles of $Z_{\eta,\rho}(\sigma)$
belong to the set 
$\left\{\pm i\sqrt{\lambda+\sigma_{\eta}}:\lambda\in 
\Sp\(C^{\fg,Z,\widehat{\eta},\rho}\)\right\}.
$
If $\lambda\in 
\Sp\(C^{\fg,Z,\widehat{\eta},\rho}\)$ and 
$\lambda\neq -\sigma_{\eta}$, the order of the zero  at $\sigma=\pm 
i\sqrt{\lambda+\sigma_{\eta}}$ is $m_{\eta,\rho}(\lambda)$. The order 
of the zero at $\sigma=0$ is $2m_{\eta,\rho}(-\sigma_{\eta})$.
Also,
\begin{align}\label{eq:funceq}
  Z_{\eta,\rho}(\sigma)=Z_{\eta,\rho}(-\sigma)\exp\big(2r\vol(Z) P_{\eta}(\sigma)\big).
\end{align}
\end{thm}
\begin{proof}For the first part of the theorem, 
	it is enough to show \eqref{eq:detfor} for $\sigma\in \bR$ large 
	enough, which will be given in Section 
	\ref{spthm12}. The rest parts of our 
	theorem is a consequence of Theorem \ref{thmGLH} and  
	\eqref{eq:detfor}. 
\end{proof}

\subsection{Proof of  (\ref{eq:detfor}) 
for $\sigma\gg1$}\label{spthm12}
\subsubsection{The case where 
$\delta(G)=1$ and $Z_{G}$ is non compact}
For $\gamma'\in {M}$, let $X_{M}(\gamma')$ be the symmetric space 
defined in \eqref{eqXr=ZrKr}
when $G$ and $\gamma$ are replaced by $M$ and $\gamma'$. We have a generalisation of \cite[Proposition 4.14]{Shfried} and 
\cite[Proposition 5.7]{Shen_Yu}.

\begin{prop}\label{propdgn1}
	If $\gamma=e^{a}k^{-1}\in H$ with $a\in \fb$ and $k\in T$, then 
	\begin{multline}\label{eq:dgn1}
		\Trs^{[\gamma]}\[\exp\(-\frac{t}{2}C^{\fg,X,\widehat{\eta}}\)\]\\
		=\frac{1}{\sqrt{2\pi 
		t}}\exp\(-\frac{|a|^{2}}{2t}-\frac{t}{2}C^{\fu_\fm,\eta}\)\[e\(TX_{M}(k),\nabla^{TX_{M}(k)}\)\]^{\max}\Trs^{E_{\eta}}\[k^{-1}\].
	\end{multline}
	If $\gamma$ can not be conjugate into $H$, then 
	\begin{align}\label{eqd100}
			\Trs^{[\gamma]}\[\exp\(-\frac{t}{2}C^{\fg,X,\widehat{\eta}}\)\]=0.
	\end{align}
\end{prop}
\begin{proof}
	Since $K=K_{M}\subset M$ and $\fp=\fb\oplus \fp_{\fm}$, by 
	\eqref{eqetahat}, we have 
	\begin{align}\label{eqdgnn1}
		C^{\infty}\(G,E_{\widehat{\eta}}\)^{K}=C^{\infty}(\bR){\otimes} 
		C^{\infty}\(M,\Lambda^{\scriptscriptstyle\bullet }(\fp^{*}_{\fm,\bC})\widehat{\otimes} 
		E_{\eta|{K_{M}}}\)^{K_{M}}. 
	\end{align}
	By the identification \eqref{eqdgnn1},
	if $\Delta^{\bR}$ is the usual Laplacian operator on $\bR$, 
	we have 
	\begin{align}\label{eqCCXM}
		C^{\fg,X,\widehat{\eta}}=-\Delta^{\bR}+C^{\fm,X_{M},\Lambda^{\scriptscriptstyle\bullet }(\fp_{\fm}^{*})\widehat{\otimes} {\eta_{|K_{M}}}}.
	\end{align}
	
	Let $\gamma$ be a semisimple element in $G$. Since 
	$G=\exp(\fb)\times M$, write 	
	$\gamma=e^{a}\gamma'$ where $a\in \fb$ and $\gamma'$ is 
	semisimple in $M$. By \eqref{eqCCXM}, we have
	\begin{align}\label{eq9J1}
		\Trs^{[\gamma]}\[\exp\(-\frac{t}{2}C^{\fg,X,\widehat{\eta}}\)\]=\Tr^{[e^{a}]}\[\exp\(\frac{t}{2} 
		\Delta^{\bR}\)\]\Trs^{[\gamma']}\[\exp\(-\frac{t}{2}C^{\fm,X_{M},\Lambda^{\scriptscriptstyle\bullet }(\fp_{\fm}^{*})\widehat{\otimes}\eta_{|K_{M}}}\)\].
	\end{align}
	Clearly, we have
	\begin{align}
	\Tr^{[e^{a}]}\[\exp\(\frac{t}{2} 
		\Delta^{\bR}\)\]=\frac{1}{\sqrt{2\pi t}}e^{-\frac{|a|^{2}}{2t}}. 
	\end{align}

	If $\gamma$ can not be conjugate in $H$, then $\gamma'$ is not elliptic in $M$, by \cite[(8-82)]{Shfried}, we have 
\begin{align}
	\Trs^{[\gamma']}\[\exp\(-\frac{t}{2}C^{\fm,X_{M},\Lambda^{\scriptscriptstyle\bullet }(\fp_{\fm}^{*})\widehat{\otimes}\eta_{|K_{M}}}\)\]=0.
\end{align}
If $\gamma\in H$, then $\gamma'=k^{-1}\in T$. By \cite[(8-83)]{Shfried}, we have 
\begin{multline}\label{eq9J2}
	\Trs^{[k^{-1}]}\[\exp\(-\frac{t}{2}C^{\fm, 
		X_{M},\Lambda^{\scriptscriptstyle\bullet }(\fp_{\fm}^{*})\widehat{\otimes} \eta_{|{K_{M}}}}\)\]\\
		=\[e\(TX_{M}(k),\nabla^{TX_{M}(k)}\)\]^{\max}\Trs^{E_\eta}\[k^{-1}\]\exp\({-\frac{t}{2}C^{\fu_\fm,\eta}}\). 
	\end{multline}
By \eqref{eq9J1}-\eqref{eq9J2}, we get \eqref{eq:dgn1} and \eqref{eqd100}. 
\end{proof}


	Proceeding  as in the proof of \cite[Theorem 5.6]{Shfried}, using Theorem \ref{thmselnon} and Proposition \ref{propdgn1} instead of \cite[Theorem 
 	4.10 and Proposition 4.14]{Shfried},  we get \eqref{eq:detfor} 
	for $\sigma\gg1$. 
	 \qed

\subsubsection{The case where  $\delta(G)=1$ and $Z_{G}$ is compact}	
\label{ssdg1nc}



	Although,  Assumption \ref{ass} on $\eta$ is different from  
	\cite[Assumption 7.1]{Shfried}, by Remarks \ref{re66} and \ref{re69}, the statement of \cite[Theorem 7.3]{Shfried} still holds.
Therefore, the proof of \eqref{eq:detfor} for 
$\sigma\gg1$ is identical to the one given in \cite[Theorem 7.6]{Shfried}, except 	
that we use Theorem \ref{thmselnon} instead of \cite[Theorem 4.10]{Shfried}.  \qed

\subsection{Proof of Theorem \ref{thmRmeo}}\label{sproofthmRmeo}
\subsubsection{The case where $\delta(G)=1$ and $Z_{G}$ is non 
compact}		Since $G$ has a non compact centre, by the second 
equation of \eqref{eqGbM},  the denominator  
$\left|\det(1-\Ad(e^{a}k^{-1}))\right|_{\fz^{\bot}(\fb)}|^{1/2}$ in \eqref{eqdefsel}  
disappears.
When $\eta=\mathbf{1}$ is the trivial representation, by 
\eqref{defRrho}, \eqref{eqdefsel}, \eqref{eqetahat}, and \eqref{eqsen1}, we have
\begin{align}\label{eqRZi}
	&R_{\rho}(\sigma)=\big(Z_{\mathbf{1},\rho}(\sigma)\big)^{-1}, 
	&E_{\widehat{\mathbf 
	1}}=\sum_{i=1}^{m}(-1)^{i-1}i\Lambda^{i}(\fp^{*}_{\bC})\in R(K), 
	&&\sigma_{\mathbf 1}=0. 
\end{align}
By Theorem \ref{thm:detfor}, and by \eqref{eqTs} and \eqref{eqRZi}, we see that 
$R_{\rho}(\sigma)$ has a meromorphic extension and   
 \begin{align}\label{eqmuR}
 	R_{\rho}(\sigma)=\exp\(r 
 	\vol(Z)\[e\(TX_{M},\nabla^{TX_{M}}\)\]^{\rm 
 	max}\sigma\)T_{\rho}(\sigma^{2}).
 \end{align} 
 
Set
\begin{align}\label{eqCr111}
&	C_{\rho}=1,& r_{\rho}=2\chi'_{\rm CM}(F). 
\end{align} 
By \eqref{eqTs1}, \eqref{eqmuR},  and \eqref{eqCr111}, we get \eqref{eqFcomtorsion}. 
By Proposition \ref{propclosesign}, \eqref{eqchiCM},  and \eqref{eqCr111}, we see that  \eqref{eqdCt11} holds with $k=0$. 
The proof of Theorem \ref{thmRmeo} is complete when $\delta(G)=1$ and 
$Z_{G}$ is not compact. \qed
	
\subsubsection{The case where $\delta(G)=1$ and $Z_{G}$ is  
compact}\label{ssetajj}		For $0\l j\l 2\ell$, let $\eta_{j}$ be the representation of $M$ on  $	\Lambda^{j}(\fn^{*}_{\bC}). $
By Corollary 
		\ref{corkey} (or \cite[Corollary 6.12]{Shfried}) and \cite[Proposition 6.13]{Shfried}, $\eta_{j}$ satisfies Assumption 
\ref{ass}. By Theorem \ref{thm:detfor}, we  see that 
$Z_{\eta_{j},\rho}$ has a meromorphic extension to $\bC$.   Using the relation 	(\cite[(7-65)]{Shfried}) 
	\begin{align}\label{eqRZi3}
		R_{\rho}(\sigma)=\prod_{j=0}^{2\ell} 
		Z_{\eta_j,\rho}(\sigma+(j-\ell)|\alpha_{0}|)^{(-1)^{j-1}},
	\end{align}
 we see that $R_{\rho}(\sigma)$ has a 
meromorphic extension. 

Write $r_{j}=r_{\eta_{j},\rho}$. As in \cite[(7-74),(7-75)]{Shfried}, set
\begin{align}\label{eqCr1111}
&	
C_{\rho}=\prod_{j=0}^{\ell-1}\(-4(\ell-j)^{2}|\alpha_{0}|^{2}\)^{(-1)^{j-1}r_{j}},&r_{\rho}=2\sum_{j=0}^{\ell}(-1)^{j-1}r_{j}. 
\end{align} 
Proceeding exactly as in 
\cite[(7-71)-(7-78)]{Shfried}, we get \eqref{eqFcomtorsion}. The 
proof of \eqref{eqdCt11} will be given in  Section 
\ref{secnonbetti}. \qed

\begin{re}
Note that $r_{\rho}$ is even. By \eqref{eqCr1111} and by 
\cite[(7-74), (7-80)]{Shfried}, we have 
	\begin{align}
		\begin{aligned}
		&C_{\rho}=(-1)^{\chi'_{\rm 
		CM}(X,F)+\frac{r_{\rho}}{2}}\prod_{j=0}^{\ell-1}\(4(\ell-j)^{2}|\alpha_{0}|^{2}\)^{(-1)^{j-1}r_{j}},\\
		&r_{\rho}=\chi'_{\rm 
		CM}(Z,F)+(-1)^{\ell-1}r_{\ell}=2\chi'_{\rm 
		CM}(Z,F)+2\sum_{j=0}^{\ell-1}(-1)^{j}r_{j}.
		\end{aligned}
	\end{align} 	
\end{re}

\section{A cohomological formula for 
$r_{j}$}\label{S:rep}

The purpose of this section is to establish  \eqref{eqdCt11} when 
$\delta(G)=1$ and $Z_{G}$ is compact. Its proof relies on some 
results from the representation theory of real  reductive groups. 

This section is organised as follows.
In Section \ref{sec:rep}, we recall the definition of infinitesimal 
characters of $\fg_{\bC}$-modules and some basic properties of the Harish-Chandra $(\fg_{\bC},K)$-modules. 

In Section \ref{sRightR}, we study the right regular representation 
of $G$ on $C^{\infty}\(\Gamma\backslash G,\widehat{p}^{*}F\)$. Given 
a character $\chi:\mathscr Z(\fg_{\bC})\to \bC$, we introduce   
a  Harish-Chandra 
$(\fg_{\bC},K)$-modules  $V_{\chi}(\Gamma,\rho)\subset 
C^{\infty}\(\Gamma\backslash G,\widehat{p}^{*}F\)$ with generalised infinitesimal character $\chi$.


Finally, in Section \ref{secnonbetti}, assuming that $\delta(G)=1$ and $Z_{G}$ 
is compact, we write $r_j$ as a sum of Euler characteristic numbers of 
certain cohomologies of $V_{\chi_{\mathbf 1}}(\Gamma,\rho)$, where 
$\chi_{\mathbf 1}$ is the trivial character.   We establish  \eqref{eqdCt11}. 

In Sections \ref{sec:rep} and \ref{sRightR}, we  assume neither 
$\delta(G)=1$ nor $Z_{G}$ is compact.

\subsection{Preliminary on the representation theory}\label{sec:rep}
Recall that $G$ is a linear connected real reductive group. 
We use  the 
notation in Sections \ref{sReductive} and \ref{sCasimir}. Recall that 
$\mathscr Z(\fg_{\bC})$ is the centre of the enveloping algebra 
of $\fg_{\bC}$. 

 A morphism of algebras 
$\chi:\mathscr{Z}(\fg_\bC)\to \bC$ will be called a character of 
$\mathscr Z(\fg_\bC)$. Clearly, for $a\in \bC$, we have
\begin{align}
	\chi(a)=a. 
\end{align}

\begin{defin}\label{definfch}
A complex representation of $\fg_\bC$ is said to have an  infinitesimal character $\chi$ if $z\in \mathscr Z(\fg_\bC)$ acts as a scalar $\chi(z)\in\bC$.

A complex representation of $\fg_\bC$ is said to have a  generalised 
infinitesimal character $\chi$ if there is $i\gg1$ such that 
for all  $z\in \mathscr Z(\fg_\bC)$, $(z-\chi(z))^{i}$ acts like $0$.
\end{defin}

The infinitesimal character of a trivial representation will be 
called the trivial character $\chi_{\bf 1}$. Note that in 
\cite[Section 8]{Shfried},  such a character  is denoted by $0$. We will not use this notation here.

\begin{defin}\label{defHC}
	A complex $\mathscr U(\fg_\bC)$-module $V$, 
	equipped with an  action of $K$, is called a $(\fg_\bC,K)$-module, if 
	\begin{enumerate}
	\item every $v\in V$ is $K$-finite, i.e., $\{k\cdot v\}_{k\in K}$   
  spans a finite dimensional vector space;
  \item the actions of $\fg_\bC$ and $K$ are compatible. 
\end{enumerate} 
A $(\fg_{\bC},K)$-module  $V$ is called a Harish-Chandra $(\fg_\bC,K)$-module, if
\begin{enumerate}
  \item the space $V$ is finitely generated as a $\mathscr U(\fg_\bC)$-module;
  \item each irreducible $K$-module occurs only
  for a finite number of times in $V$.
\end{enumerate}
\end{defin}

 Let $V$ be a Harish-Chandra $(\fg_\bC,K)$-module. By 
 \cite[Proposition 10.43]{Knappsemi},  submodules  or quotient 
modules of $V$ is still  a Harish-Chandra $(\fg_{\bC},K)$-module. 
When $V$ is irreducible, by \cite[Corollary 8.13]{Knappsemi},  $V$ 
has an infinitesimal character. For general $V$, by 
 \cite[Proposition 10.41]{Knappsemi}, there are finitely many  Harish-Chandra $(\fg_\mathbf{C}, 
K)$-submodules $V_{\chi}\subset V$ of $V$ with  generalised 
infinitesimal character $\chi$, so that 
	\begin{align}
			V=\bigoplus_{\chi}V_{\chi}. 
	\end{align} 
%
%
%
	
By \cite[Corollary 10.42]{Knappsemi},  $V$ has  a finite composition series, i.e.,  there 
exist finitely many Harish-Chandra $(\fg_\bC, K)$-submodules 
\begin{align}\label{eqVcomf}
	V=V_{N_{0}}\supset V_{N_{0}-1}\supset \cdots\supset V_{0}\supset V_{-1}=0,
\end{align} 
such that each quotient $V_{s}/V_{s-1}$  with  $0\l  s \l  N_{0}$ is irreducible. Moreover, the length 
$N_{0}\in \mathbf N$, the set of 
all irreducible quotients and their multiplicities are the same for 
all the composition series. Also, if $E$ is a finite dimensional 
representation of $K$, then 
\begin{align}\label{eqVKVii1}
	\dim \(V\otimes E	\)^{K}=\sum_{s=0}^{N_{0}} \dim 
	\big((V_{s}/V_{s-1})\otimes E\big)^{K}. 
\end{align} 

%


\begin{prop}\label{propVc0}
Given a character $\chi:\mathscr Z(\fg_{\bC})\to \bC$, there is a 
finite dimensional representation $E$ of $K$ such that if $V$ 
is a Harish-Chandra $(\fg_{\bC},K)$-module with generalised infinitesimal 
character $\chi$, then 
\begin{align}\label{eqV0VE0}
	V=0\iff (V\otimes E)^{K}=0. 
\end{align} 
\end{prop}
\begin{proof}It is enough to show the inverse direction. By	a fundamental theorem  of Harish-Chandra  
	\cite[Corollary 10.37]{Knappsemi}, there are only finitely many irreducible Harish-Chandra $(\fg_{\bC},K)$-modules $V_{1},\ldots,V_{k}$  with 	infinitesimal character $\chi$. Let 
	$E_{i}$ be a $K$-type of $V_{i}$. Set
	\begin{align}\label{eqdefE}
	E=\bigoplus_{i=1}^{k} E^{*}_{i}. 
\end{align} 
By \eqref{eqdefE}, for all $1\l i\l k$, 
\begin{align}\label{eqEvin0}
	\(	V_{i}\otimes E\)^{K}\neq 0. 
\end{align} 
By \eqref{eqVKVii1} and \eqref{eqEvin0}, we get the inverse direction 
in  \eqref{eqV0VE0}. 
\end{proof}

\subsection{The right regular representation on 
$C^{\infty}\(\Gamma\backslash G,\widehat{p}^{*}F\)$}\label{sRightR}
Recall that 
$\Gamma\subset G$ is a discrete  cocompact and torsion free subgroup of $G$ and that $\rho:\Gamma\to \GL_{r}(\bC)$ is a representation of 
$\Gamma$. Recall also  that $F$ is the flat vector bundle on 
$Z=\Gamma\backslash G/K$
 defined in \eqref{eqFhol1}.  We  fix a Hermitian metric  
$g^{F}$ on $F$. Note that the choice of $g^{F}$ is irrelevant. 

Let 
$g^{\widehat{p}^{*}F}$ be the pull back metric on $\widehat{p}^{*}F$.  Let $L^{2}\(\Gamma\backslash G,\widehat{p}^{*}F\)$ be the
$L^{2}$-space associated to $g^{\widehat{p}^{*}F}$ and to an 
invariant measure on $\Gamma\backslash G$. 
The group $G$ acts on the right on  $L^{2}\(\Gamma\backslash 
G,\widehat{p}^{*}F\)$. Since $\rho$ is not necessarily unitary, this $G$-action is not always unitary. Since  $g^{\widehat{p}^{*}F}$ is 
$K$-invariant, the above $G$-action restricts to a  unitary 
representation of $K$. Let $\widehat{K}$ be the set of equivalent 
classes of irreducible representations of $K$. 
If $\tau\in \widehat{K}$,  let $L^{2}_{\tau}(\Gamma\backslash 
G,\widehat{p}^{*}F)$ be the $\tau$-isotropic  subspace of $L^{2}(\Gamma\backslash 
G,\widehat{p}^{*}F)$. By \eqref{eq:Etau} and \eqref{eqFtau0}, we have   
\begin{align}\label{eqL2taurho}
	L^{2}_{\tau}\(\Gamma\backslash 
	G,\widehat{p}^{*}F\)\simeq L^{2}\(Z,\cF_{\tau}^{*}\otimes F\)\otimes E_{\tau}. 
\end{align} 
By Peter-Weyl's theorem, we have a direct sum of Hilbert spaces 
\begin{align}\label{eqPeterWeyl}
	L^{2}\(\Gamma\backslash G,\widehat{p}^{*}F\)=\bigoplus_{\tau\in 
	\widehat{K}}^{\rm Hil}L^{2}_{\tau}\(\Gamma\backslash 
	G,\widehat{p}^{*}F\). 
\end{align} 

By \eqref{eqL2taurho}, 
the Casimir operator $C^{\fg,Z,\tau^{*},\rho}$ defines an  unbounded operator  on the Hilbert space $L^{2}_{\tau}\(\Gamma\backslash 
	G,\widehat{p}^{*}F\)$. If $\lambda\in 
	\Sp(C^{\fg,Z,\tau^{*},\rho})$, 	as in \eqref{eqVpch}, let  
	$C^{\infty}_{\tau,\lambda}(\Gamma\backslash G,\widehat{p}^{*}F)$ 
	be the corresponding finite dimensional 	characteristic space. As the notation indicates,
	\begin{align}
	C^{\infty}_{\tau,\lambda}\(\Gamma\backslash 
	G,\widehat{p}^{*}F\)\subset C^{\infty}\(\Gamma\backslash 
	G,\widehat{p}^{*}F\).
\end{align} 
	
	Let $V(\Gamma,\rho)$ be the subspace of $  C^{\infty}(\Gamma\backslash 	G,\widehat{p}^{*}F)
$ defined by the algebraic sum
\begin{align}\label{eq}
V(\Gamma,\rho)=\bigoplus_{\tau\in 
	\widehat{K}}\bigoplus_{\lambda\in 
	\Sp(C^{\fg,Z,\tau^{*},\rho})}C^{\infty}_{\tau,\lambda}\(\Gamma\backslash G,\widehat{p}^{*}F\).
\end{align} 
Clearly, the group $K$ acts on $V(\Gamma,\rho)$ and the elements of 
$V(\Gamma,\rho)$ are $K$-finite.  

By \eqref{eqvp}, \eqref{eqPeterWeyl}, and \eqref{eq}, we have
	\begin{align}
	L^{2}\(\Gamma\backslash G,\widehat{p}^{*}F\)=\ol{V(\Gamma,\rho)}. 
\end{align} 

\begin{prop}\label{propVgk}
The algebra $\mathscr U(\fg_{\bC})$ acts on $V(\Gamma,\rho)$, so that 
$V(\Gamma,\rho)$ is a $(\fg_{\bC},K)$-module. 
\end{prop}	
\begin{proof}We have seen in the arguments below  
\eqref{eq} that elements of $V(\Gamma,\rho)$ are $K$-finite. We need to show that  $\fg_\bC$ acts on $V(\Gamma,\rho)$, or equivalently, for  
any pair $(\tau,\lambda)$, 
	\begin{align}\label{eqFM3}
		\fg \cdot  C^{\infty}_{\tau,\lambda}\(\Gamma\backslash 
G,\widehat{p}^{*}F\)\subset V(\Gamma,\rho).
	\end{align}

Let us follow the proof of  \cite[Proposition 
8.5]{Knappsemi}.	
 Since the space $ C^{\infty}_{\tau,\lambda}(\Gamma\backslash 
G,\widehat{p}^{*}F)\subset C^{\infty}(\Gamma\backslash 
G,\widehat{p}^{*}F)$ is $K$-invariant,  for $k\in K$ and $a\in 
\fg$, we have
\begin{multline}\label{eqadCtl}
	k \cdot a \cdot C^{\infty}_{\tau,\lambda}\(\Gamma\backslash 
G,\widehat{p}^{*}F\)\subset   (\Ad(k)a) \cdot 
C^{\infty}_{\tau,\lambda}\(\Gamma\backslash 
G,\widehat{p}^{*}F\)
\subset \fg \cdot C^{\infty}_{\tau,\lambda}\(\Gamma\backslash 
G,\widehat{p}^{*}F\). 
\end{multline} 
By \eqref{eqadCtl}, the finite dimensional  space   $\fg \cdot  C^{\infty}_{\tau,\lambda}(\Gamma\backslash 
G,\widehat{p}^{*}F)$ is $K$-invariant.   In particular, the elements of $\fg\cdot C^{\infty}_{\tau,\lambda}(\Gamma\backslash 
G,\widehat{p}^{*}F)$ are $K$-finite. Therefore,
\begin{align}\label{eq810}
	\fg \cdot  C^{\infty}_{\tau,\lambda}\(\Gamma\backslash 
G,\widehat{p}^{*}F\)\subset \bigoplus_{\tau\in 
	\widehat{K}}C^{\infty}_{\tau}\(\Gamma\backslash 
	G,\widehat{p}^{*}F\).
\end{align} 
Since the Casimir commutes with  the $\fg$-actions, by \eqref{eq810},  we have 
	\begin{align}\label{eqFM2}
		\fg \cdot  C^{\infty}_{\tau,\lambda}\(\Gamma\backslash 
G,\widehat{p}^{*}F\)\subset \bigoplus_{\tau\in 
	\widehat{K}}C^{\infty}_{\tau,\lambda}\(\Gamma\backslash 
	G,\widehat{p}^{*}F\).
	\end{align} 
	By \eqref{eq} and \eqref{eqFM2}, we get \eqref{eqFM3}. 
\end{proof}
	
Note that $\mathscr Z(\fg_{\bC})$ acts on the finite dimensional 
space $C^{\infty}_{\tau,\lambda}(\Gamma\backslash 
G,\widehat{p}^{*}F)$. For a character  $\chi$ of $\mathscr Z(\fg_{\bC})$, set
\begin{align}\label{eqVchi}
	V_{\chi}(\Gamma,\rho)=\left\{v\in V(\Gamma,\rho): 
	 \exists i\in \mathbf{N}, \forall z\in 
	\mathscr Z(\fg_{\bC}),	\big(z-\chi(z)\big)^{i}v=0\right\}. 
\end{align} 
By \eqref{eq} and \eqref{eqVchi}, we have an  infinite algebraic sum 
\begin{align}
	V(\Gamma,\rho)=\bigoplus_{\chi} V_{\chi}(\Gamma,\rho). 
\end{align} 
By Proposition \ref{propVgk}, $V_{\chi}(\Gamma,\rho)$ is a 
$(\fg_{\bC},K)$-module. 

\begin{prop}
	Each $V_{\chi}(\Gamma,\rho)$ is a Harish-Chandra $(\fg_{\bC},K)$-module. 
\end{prop}
\begin{proof}
	By \eqref{eqVchi}, we have
		\begin{align}\label{eqVchiCchiC}
		V_{\chi}(\Gamma,\rho)\subset \bigoplus_{\tau\in 
		\widehat{K}}C^{\infty}_{\tau,\chi(C^{\fg})}\(\Gamma\backslash 
G,\widehat{p}^{*}F\).
	\end{align} 
In particular,  the multiplicity of each $K$-type in $V_{\chi}(\Gamma,\rho)$ is finite. 

It remains to show 
that $V_{\chi}(\Gamma,\rho)$ is a finitely  generated $\mathscr U(\fg_\bC)$-module. Otherwise,
%
we 
can find a sequence  $(v_{n})_{n\in \mathbf N}$ in 
$V_{\chi}(\Gamma,\rho)$
such that 
\begin{align}
0\varsubsetneq \mathscr U(\fg_{\bC})v_{1}\varsubsetneq \ldots \varsubsetneq 
\sum_{i=1}^{n-1}\mathscr 
U(\fg_{\bC})v_{i}\varsubsetneq \sum_{i=1}^{n}\mathscr 
U(\fg_{\bC})v_{i}\varsubsetneq \cdots
\end{align} 
%
For $n\g 1$,  $\sum_{i=1}^{n}\mathscr U(\fg_{\bC})v_{i}$ is a 
Harish-Chandra $(\fg_{\bC},K)$-module.  Since there are only finitely many 
irreducible Harish-Chandra  $(\fg_{\bC},K)$-modules with infinitesimal character $\chi$ (\cite[Corollary 
10.37]{Knappsemi}), there is at least one, denoted by $V$,   occurs an infinite number of times  in $\sum_{i=1}^{n}\mathscr U(\fg_{\bC})v_{i}/\sum_{i=1}^{n-1}\mathscr 
U(\fg_{\bC})v_{i}$ with $n\g 1$. If $\tau\in \widehat K$ is a $K$-type 
in $V$, then $\tau$ occurs infinitely many times in $V_{\chi}(\Gamma,\rho)$, which is in contradiction with \eqref{eqVchiCchiC}. 
\end{proof}

The following proposition is an immediate consequence of \eqref{eq} 
and \eqref{eqVchi}. 

\begin{prop}
If $(\tau,E_{\tau})$ is a representation of $K$,  then we have a finite sum
\begin{align}\label{eqchsp}
C^{\infty}_{\lambda}\(Z,\cF_{\tau}\otimes 
F\)=\bigoplus_{\chi:\chi(C^{\fg})=\lambda}\(V_{\chi}(\Gamma,\rho)\otimes 
E_{\tau}\)^{K}. 
\end{align} 
\end{prop} 

We use the notation in Section \ref{sAT}. Recall that ${\rm 
Rep}(\Gamma,\bC^{r})$ is the set of all $r$-dimensional representations of $\Gamma$. 

\begin{prop}\label{propMBK}
	There is $k_{0}\in \mathbf N$ such that 
	if $\chi$ is a character of $\mathscr Z(\fg_{\bC})$ and if  
	$\rho_{0}\in {\rm Rep}(\Gamma,\bC^{r})$ is such that 
	$V_{\chi}(\Gamma,\rho_{0})=0$, then there is $\e>0$ such that  if 
	$\rho\in {\rm Rep}(\Gamma,\mathbf C^{r})$ is $C^{k_{0}}$ $\e$-close to 	$\rho_{0}$, then 
\begin{align}\label{eqVchi0}
	V_{\chi}(\Gamma,\rho)=0. 
\end{align} 
\end{prop}
\begin{proof}
	We fix a character $\chi$ of $\mathscr Z(\fg_{\bC})$. Let $E_{\tau}$ be 
	the $K$-representation constructed  in Proposition \ref{propVc0}, so that 
	\eqref{eqVchi0} is equivalent to  
\begin{align}\label{eqFM3A}
\(	V_{\chi}(\Gamma,\rho)\otimes E_{\tau}\)^{K}=0.
\end{align} 

By \cite[Theorem 8.19]{Knappsemi},  $\mathscr Z(\fg_{\bC})$ is a 
polynomial algebra of $\dim \fh$ variables. Let $z_{1}=C^{\fg},z_{2},\ldots,z_{\dim \fh}$ be a set of generators.  If 
$z\in \mathscr Z(\fg_{\bC})$, 
denote by $z^{Z,\tau,\rho}$ the action of $z$ on 
$C^{\infty}(Z,\cF_{\tau}\otimes F_{\rho})$. Clearly, 
$z^{Z,\tau,\rho}$ is a differential operator. Recall   that we have choosen a Hermitian metric on $F_{\rho}$ and it induces an  
$L^{2}$-product on $C^{\infty}(Z,\cF_{\tau}\otimes F_{\rho})$. Take 
$d\in \mathbf N$ such that $2d$ is strictly larger  than the order of  
the differential operators $z_{i}^{Z,\tau,\rho}$ with $2\l i\l \dim \fh$. 
Set
\begin{multline}\label{eqLrho1}
L_{\rho}=	
\left\{\(C^{\fg,Z,\tau,\rho}-\chi(C^{\fg})\)^{*}\(C^{\fg,\tau,E,\rho}-\chi(C^{\fg})\)\right\}^{d}\\
+\sum_{i=2}^{\dim 
	\fh}\(z^{Z,\tau,\rho}_{i}-\chi(z_{i})\)^{*}\(z_{i}^{Z,\tau,\rho}-\chi(z_{i})\),
\end{multline} 
where ${}^*$ is the  adjoint with respect to the $L^{2}$-product. Then $L_{\rho}$ is a self-adjoint elliptic differential operator of 
order $4d$. 

We claim that 
	\begin{align}\label{eqLrho15}
		\(V_{\chi}(\Gamma,\rho)\otimes E_{\tau}\)^{K} =0\iff  \ker 
		L_{\rho}=0. 
	\end{align} 
Indeed, by \eqref{eqLrho1}, we have
	\begin{align}\label{eqLrho2}
\ker L_{\rho}=\bigcap_{i=1}^{\dim \fh}\ker 
		\(z_i^{Z,\tau,\rho}-\chi(z_{i})\)\subset  \(V_{\chi}(\Gamma,\rho)\otimes 
		E_{\tau}\)^{K},
	\end{align} 
which implies the direction $\implies$. For the other direction, if $\(V_{\chi}(\Gamma,\rho)\otimes 
		E_{\tau}\)^{K}\neq0$, since the operators $z_{i}^{Z,\tau,\rho}$ commute and act on the finite dimensional vector space $\(V_{\chi}(\Gamma,\rho)\otimes 
		E_{\tau}\)^{K}$, they have  a common eigenvector. This means 	$\ker L_{\rho}\neq0$.

		Fix now $\rho_{0}\in {\rm Rep}(\Gamma,\bC^{r})$ such that 
		$V_{\chi}(\Gamma,\rho_{0})=0$. By \eqref{eqLrho15}, $L_{\rho_{0}}$ is 
		invertible.  As in 
		\eqref{eqFNrho1}, if $\rho$ is in a small  $C^{0}$-neighborhood  of 
		$\rho_{0}$, we can  identity $(F_{\rho},\nabla^{F_{\rho}})$ 
		with $(F_{\rho_{0}},\nabla^{F_{\rho_{0}}}+A_{\rho})$. We take $g^{F_{\rho}}=g^{F_{\rho_{0}}}$. 
%
Then, $L_{\rho}$ and $L_{\rho_{0}}$ act on the same space 
$C^{\infty}(Z,\cF_{\tau}\otimes F_{\rho_{0}})$ with the same principal symbol. Set
	\begin{align}\label{eqBrho}
		B_{\rho}=L_{\rho}-L_{\rho_{0}}. 
	\end{align} 
Then, $B_{\rho}$ is a differential operator of order $4d-1$. By \eqref{eqFNrho}, there is $C>0$ such that 	
\begin{align}\label{eq629}
		\|B_{\rho}\|_{\cH^{4d-1},L^{2}}\l C\|A_{\rho}\|_{C^{4d-1}}. 
	\end{align} 

By \eqref{eqBrho}, since $L_{\rho_{0}} $	is invertible, we have 
	\begin{align}
		L_{\rho}=\(1+B_{\rho}L^{-1}_{\rho_{0}}\)L_{\rho_{0}}. 
	\end{align} 
Take $k_{0}=4d-1$. By \eqref{eq629}, if 
$\|A_{\rho}\|_{C^{k_{0}}}$ is small enough, then  
$\|B_{\rho}L_{\rho_{0}}^{-1}\|_{L^{2},L^{2}}<1$ and  
 $L_{\rho}$ is invertible. By \eqref{eqLrho15}, we get \eqref{eqFM3A}. 
\end{proof}


We have an analogue \cite[(8-68)]{Shfried}. 

\begin{thm}\label{thml3}
	Let $\rho_{0}\in {\rm Rep}(\Gamma,\bC^{r})$ be an acyclic and unitary representation of $\Gamma$. There is $\e>0$ such that if 
$\rho$ is $C^{k_{0}}$ $\e$-close to $\rho_{0}$, then 
\begin{align}\label{eqV1=0}
	V_{\chi_{\mathbf 1}}(\Gamma,\rho)=0.
\end{align} 
\end{thm}
\begin{proof}
In \cite[Proposition 8.12]{Shfried}, using  fundamental results of 
	Vogan-Zuckerman \cite{VoganZuckerman}, Vogan \cite{Vogan2}, and 
	Salamanca-Riba \cite{Salamanca}, we have shown that  $\rho_{0}$ 
	is unitary and acyclic, if and only if 
\begin{align}\label{eqVVZSR}
	V_{\chi_{\bf 1}}(\Gamma,\rho_{0})=0. 
\end{align} 
By Proposition \ref{propMBK}, if $\rho$ is close enough to 
$\rho_{0}$ in the sense of $C^{k_{0}}$, we get \eqref{eqV1=0}. 
%
\end{proof}

\subsection{Formulas for $r_{\eta,\rho}$ and $r_{j}$}\label{secnonbetti}
Assume now $\delta(G)=1$ and $Z_{G}$ is compact. Let $\rho:\Gamma\to 
\GL_{r}(\bC)$ be a representation of $\Gamma$, and let  $\eta$ be a 
virtual representation of $M$ satisfying  Assumption \ref{ass}. 
Recall  that $\widehat{\eta}$ is defined in \eqref{eqetahat}.

If $V$ is a 
Harish-Chandra $(\fg_{\bC},K)$-module, denote by 
$H^{\scriptsize\bullet}(\fg,K;V)$ the associated  $(\fg_{\bC},K)$-cohomology and by $H_{\scriptsize\bullet}(\fn,V)$ the 
$\fn$-homology. By \cite[Proposition 2.24]{HechtSchmid}, 
$H_{\scriptsize\bullet}(\fn,V)$ is a Harish-Chandra 
$(\fm_{\bC},K_{M})$-module. Let $\chi(K/K_{M})$ be the Euler 
characteristic number 
of $K/K_{M}$. 

\begin{prop}\label{prop810}
	The following identity holds,  
\begin{multline}\label{eqretar1}
r_{\eta,\rho}=\frac{1}{\chi(K/K_{M})}
\sum_{\tiny{\substack{\chi:\chi(C^{\fg})=0\\ 0\l i\l \dim \fp_{\fm}\\ 
0\l k \l 2\ell}}}
(-1)^{i+k}
 \Big\{\dim H^{i}\(\fm,K_{M};H_{k}\(\fn,V_{\chi}(\Gamma,\rho)\)\otimes 
E_{\eta}^{+}\)\\
-\dim 
H^{i}\(\fm,K_{M};H_{k}\(\fn,V_{\chi}(\Gamma,\rho)\)\otimes 
E_{\eta}^{-}\)\Big\},
\end{multline} 
where the above sum is finite. 
\end{prop}
\begin{proof}
We claim that	it is enough  to show that \begin{multline}\label{eqreta10}
\dim	\(V_{\chi}(\Gamma,\rho)\otimes  
E_{\widehat{\eta}}^{+}\)^{K}-\dim \(V_{\chi}(\Gamma,\rho)\otimes  
E_{\widehat{\eta}}^{-}\)^{K}\\
=\frac{1}{\chi(K/K_{M})}\sum_{\tiny{\substack{ 0\l i\l \dim \fp_{\fm}\\ 
0\l k \l 2\ell}}} (-1)^{i+k} \Big\{\dim 
\(\Lambda^{i}(\fp_{\fm}^{*})\otimes H_{k}\(\fn,V_{\chi}(\Gamma,\rho)\)\otimes 
E_{\eta}^{+}\)^{K_{M}}
 \\
-\dim 
\(\Lambda^{i}(\fp_{\fm}^{*})\otimes H_{k}\(\fn,V_{\chi}(\Gamma,\rho)\)\otimes 
E_{\eta}^{-}\)^{K_{M}}
\Big\}. 
\end{multline} 
Indeed, by \eqref{eqretarho} and \eqref{eqchsp},  we have a finite sum 
	\begin{align}\label{eqrt1}
r_{\eta,\rho}=\sum_{\chi:\chi(C^{\fg})=0}\left\{\dim \(V_{\chi}(\Gamma,\rho)\otimes 
E_{\widehat{\eta}}^{+}\)^{K}-\dim \(V_{\chi}(\Gamma,\rho)\otimes 
E_{\widehat{\eta}}^{-}\)^{K}\right\}.
\end{align} 
Equation \eqref{eqretar1} follows from  
\eqref{eqreta10}, \eqref{eqrt1},  and the Euler formula \cite[(8-85)]{Shfried}. 

If $V_{\chi}(\Gamma,\rho)$ is unitary and 
irreducible, then \eqref{eqreta10} is just  \cite[Theorem 
8.14]{Shfried}. However, the proof  extends to any irreducible 
$(\fg_{\bC},K)$-module without any change. In particular, 
if $V_{\chi}(\Gamma,\rho)$ has a composition series \eqref{eqVcomf}, 
for $0\l s\l N_{0}$, we have
\begin{multline}\label{eqreta1}
\dim	\((V_{s}/V_{s-1})\otimes  
E_{\widehat{\eta}}^{+}\)^{K}-\dim \((V_{s}/V_{s-1})\otimes  
E_{\widehat{\eta}}^{-}\)^{K}\\
=\frac{1}{\chi(K/K_{M})}\sum_{\tiny{\substack{ 0\l i\l \dim \fp_{\fm}\\ 
0\l k \l 2\ell}}} (-1)^{i+k} \Big\{\dim 
\(\Lambda^{i}(\fp_{\fm}^{*})\otimes H_{k}\(\fn,V_{s}/V_{s-1}\)\otimes 
E_{\eta}^{+}\)^{K_{M}}
 \\
-\dim 
\(\Lambda^{i}(\fp_{\fm}^{*})\otimes H_{k}\(\fn,V_{s}/V_{s-1}\)\otimes 
E_{\eta}^{-}\)^{K_{M}}
\Big\}. 
\end{multline}

Using the exact sequence of the Harish-Chandra $(\fg_{\bC},K)$-modules,
\begin{align}\label{eqexaVii1}
	0\to V_{s-1}\to V_{s}\to V_{s}/V_{s-1}\to0,
\end{align} 
 we have 
 \begin{align}\label{eqreta2}
 \dim \((V_{s}/V_{s-1})\otimes  
 E_{\widehat{\eta}}^{\pm}\)^{K}=\dim \(V_{s}\otimes  
 E_{\widehat{\eta}}^{\pm}\)^{K}-\dim \(V_{s-1}\otimes  
 E_{\widehat{\eta}}^{\pm}\)^{K}.
 \end{align} 

 On the other hand, by \eqref{eqexaVii1},
we have the long exact sequence of Harish-Chandra 
$(\fm_{\bC},K_{M})$-modules, 
\begin{align}\label{eqexnV}
\cdots \to	H_{k}(\fn,V_{s-1})\to H_{k}(\fn,V_{s})\to 
	H_{k}(\fn,V_{s}/V_{s-1})\to  H_{k-1}(\fn,V_{s-1})\to\cdots
\end{align} 
By \eqref{eqexnV}, for $0\l i\l \dim \fp_{\fm}$ and $0\l s\l N_{0}$, we have
\begin{multline}\label{eqreta3}
\sum_{k=0}^{2\ell}(-1)^{k}\dim 
\(\Lambda^{i}(\fp_{\fm}^{*})\otimes H_{k}\(\fn,V_{s}/V_{s-1}\)\otimes 
E_{\eta}^{\pm}\)^{K_{M}}\\
=\sum_{k=0}^{2\ell}(-1)^{k}\Big\{\dim 
\(\Lambda^{i}(\fp^{*}_{\fm})\otimes H_{k}\(\fn,V_{s}\)\otimes 
E_{\eta}^{\pm}\)^{K_{M}}\\
-\dim 
\(\Lambda^{i}(\fp_{\fm}^{*})\otimes H_{k}\(\fn,V_{s-1}\)\otimes 
E_{\eta}^{\pm}\)^{K_{M}}\Big\}.
\end{multline} 
By \eqref{eqreta1}, \eqref{eqreta2}, and \eqref{eqreta3}, we get 
\eqref{eqreta10}. 
\end{proof} 

Recall that for $0\l j\l 2\ell$, $\eta_{j}$ is the adjoint representation of $M$ on 
$\Lambda^{j}(\fn_{\bC}^{*})$. In Section \ref{ssetajj}, we 
have seen that $\eta_{j}$ satisfies Assumption \ref{ass}. 
\begin{prop}\label{prop811}
 If $V$ is an irreducible Harish-Chandra 	$(\fg_{\bC},K)$-module with infinitesimal character $\chi$ such that  $\chi(C^{\fg})=0$ 
 and 
\begin{align}
\bigoplus_{\tiny{\substack{ 0\l i\l \dim \fp_{\fm}\\ 
0\l k \l 2\ell}}} 
H^{i}\({\fm},K_{M}; H_{k}(\fn,V)\otimes 
\Lambda^{j}(\fn^{*}_{\bC})\)\neq0, 	
\end{align} 
then 
 \begin{align}
 	\chi=\chi_{\mathbf 1}.
 \end{align} 
\end{prop}
\begin{proof}
	The proof given in \cite[Proposition 
8.17]{Shfried} where $V$ is unitary extends to any 
irreducible Harish-Chandra 	$(\fg_{\bC},K)$-module without any 
change. 
\end{proof}

\begin{cor}\label{cor812}
If $V$ is Harish-Chandra 
	$(\fg_{\bC},K)$-module with generalised infinitesimal character 
	$\chi$ such that $\chi(C^{\fg})=0$ and 
\begin{align}\label{eq837}
	\sum_{\tiny{\substack{ 0\l i\l \dim \fp_{\fm}\\ 
0\l k \l 2\ell}}}(-1)^{i+k} \dim H^{i}\({\fm},K_{M}; 
H_{k}(\fn,V)\otimes \Lambda^{j}(\fn^{*}_{\bC})\)\neq0, 	
\end{align} 
then
\begin{align}\label{eqchi00}
	\chi=\chi_{\mathbf 1}.
\end{align} 
\end{cor} 
\begin{proof}
	If $V$ has a composition series \eqref{eqVcomf}, by 
	\eqref{eqreta3} and the Euler formula \cite[(8-85)]{Shfried}, we have
	\begin{multline}\label{eq812}
	\sum_{\tiny{\substack{ 0\l i\l \dim \fp_{\fm}\\ 
0\l k \l 2\ell}}}(-1)^{i+k} \dim 
	H^{i}\({\fm},K_{M}; H_{k}(\fn,V)\otimes 
	\Lambda^{j}(\fn^{*}_{\bC})\)\\
	=\sum_{s=0}^{N_{0}}\sum_{\tiny{\substack{ 0\l i\l \dim \fp_{\fm}\\ 
0\l k \l 2\ell}}}(-1)^{i+k} \dim 
H^{i}\({\fm},K_M; H_{k}(\fn,V_{s}/V_{s-1})\otimes 
	\Lambda^{j}(\fn^{*}_{\bC})\).
\end{multline} 
By \eqref{eq837}, \eqref{eq812}, there is $0\l s\l N_{0}$ such that  
\begin{align}
\bigoplus_{\tiny{\substack{ 0\l i\l \dim \fp_{\fm}\\ 
0\l k \l 2\ell}}} 
H^{i}\({\fm},K_{M}; H_{k}(\fn,V_{s}/V_{s-1})\otimes 
\Lambda^{j}(\fn^{*}_{\bC})\)\neq0.
\end{align} 
By Proposition \ref{prop811}, we get \eqref{eqchi00}. 
\end{proof}
We have a generalisation of \cite[Corollary 8.18]{Shfried}. 
\begin{thm}\label{thmcor818}
	For $0\l j\l 2\ell$, we have 
\begin{align}\label{eqrjma}
r_{j}=\frac{1}{\chi(K/K_{M})}\sum_{\tiny{\substack{ 0\l i\l \dim \fp_{\fm}\\ 
0\l k \l 2\ell}}} (-1)^{i+k}\dim 
H^{i}\(\fm,K_{M};H_{k}\(\fn,V_{\chi_{\bf 1}}(\Gamma,\rho)\)\otimes 
\Lambda^{j}(\fn^{*}_{\bC})\). 	
\end{align} 
In particular, if $V_{\chi_{\bf 1}}(\Gamma,\rho)=0$, for $0\l j\l 
2\ell$, we have 
\begin{align}\label{eqrj0}
	r_{j}=0. 
\end{align}
\end{thm}
\begin{proof}
	This is a consequence of  Proposition 
\ref{prop810} and  Corollary \ref{cor812}. 
\end{proof}

%

\begin{re}\label{re613}
	By \eqref{eqCr1111}, Theorems \ref{thml3} and  \ref{thmcor818}, we get \eqref{eqdCt11} in the case $\delta(G)=1$ and 
	$Z_{G}$ is compact. We complete the proof of Theorem \ref{thmRmeo} in full generality. 
\end{re}

\section{An extension to orbifolds}\label{Snontorsionfree}
 In this section, we  no longer assume that  $\Gamma\subset G$ is 
torsion free. Then $Z=\Gamma\backslash G/K$ is a closed  orbifold. 
The purpose of this section is to extend Theorem \ref{thmRmeo} to orbifolds. 

This section is organised as follows. In Section \ref{sobsle}, we establish  M\"uller's Selberg trace formula for orbifolds. 

In Section \ref{sorbRuelle}, we indicate the essential steps in 
generalising  Theorem \ref{thmRmeo} to orbifolds.

\subsection{M\"uller's Selberg trace formula on locally symmetric orbifolds}\label{sobsle}
We use the notation in Section \ref{sec:Gamma}. Recall that  
$\Gamma\subset G$ is a discrete cocompact subgroup of $G$. Then, 
$Z=\Gamma\backslash G/K$ is a closed orbifold.  Recall also 
that 
$\rho:\Gamma\to \GL_{r}(\bC)$ is a representation of $\Gamma$, and 
that $(\tau,E_{\tau})$ is a representation of 
$K$. Define $F$ and 
$\cF_{\tau}$  as in \eqref{eqFhol1} and \eqref{eqFtau0}. Then $F$ is 
a flat  orbifold vector bundle and 
$\cF_{\tau}$  is a Hermitian orbifold vector bundle.  As in the case  where $\Gamma$ is torsion free, the Casimir operator $C^{\fg}$ induces a generalised  Laplacian $C^{\fg,Z,\tau,\rho}$ acting on 
$C^{\infty}(Z,\cF_{\tau}\otimes F)$. 

For $\gamma\in \Gamma$, we have seen in Section \ref{sec:Gamma} 
that   $\gamma$ is semisimple.  The group $K(\gamma)$ acts on the right on $\Gamma(\gamma)\backslash 
Z(\gamma)$. For $h\in \Gamma(\gamma)\backslash 
Z(\gamma)$, let $K(\gamma)_h$ be the stabiliser of $h$ in 
$K(\gamma)$.  Since $\Gamma(\gamma)\backslash X(\gamma)$ is 
connected,  the cardinality of a generic 
stabiliser $K(\gamma)_{h}$ is well defined (see \cite[Section 
3.1]{BismutLabourie99}) and depends only on the 
conjugacy class of $\gamma$ in $\Gamma$. We denote it by 
$n_{[\gamma]}$. By \cite[(3.10)]{BismutLabourie99},  we have 
\begin{align}
\frac{	\vol(\Gamma(\gamma)\backslash 
Z(\gamma))}{\vol(K(\gamma))}=\frac{\vol(\Gamma(\gamma)\backslash 
X(\gamma))}{n_{[\gamma]}}. 
\end{align}
By \cite[Proposition 5.3]{Shen_Yu} and by \eqref{eqzpkgamma},\eqref{eqXr=ZrKr}, we have
\begin{align}
	n_{[\gamma]}=\left|K\cap \Gamma(\gamma)\cap Z(\fp(\gamma))\right|=\left|\ker\(\Gamma(\gamma)\to {\rm Diffeo} 
	\big(X(\gamma)\big)\)\right|.
\end{align} 

Recall that $[\Gamma_{+}]\subset [\Gamma]$ is the set of non elliptic  conjugacy 
classes of $\Gamma$. For $[\gamma]\in [\Gamma]$, $B_{[\gamma]}$ is 
defined in \eqref{eqBgamma}. We have a generalisation of  \cite[Theorem 5.4]{Shen_Yu} and Theorem 
\ref{thmselnon}. 

\begin{thm}\label{thmselnonob}
There exist $c>0$, $C>0$ such that
  for $t>0$, we have
  \begin{align}\label{eq:h11expob}
    \sum_{[\gamma]\in [\Gamma_{+}]} 
	\frac{\vol\big(B_{[\gamma]}\big) 
	}{n_{[\gamma]}}
	\big|\Tr[\rho(\gamma)]\big| 
	\left|\Tr^{[\gamma]}\[\exp\(-tC^{\fg,X,\tau}\)\]\right|
	\l C\exp\(-\frac{c}{t}+Ct\).
  \end{align}
  For $t>0$, the following identity holds,
\begin{align}\label{eq:selob}
  \Tr\[\exp\(-tC^{\fg,Z,\tau,\rho}\)\]=\sum_{[\gamma]\in [\Gamma]} 
  \frac{\vol\big(B_{[\gamma]}\big)}{n_{[\gamma]}} \Tr[\rho(\gamma)] 
  \Tr^{[\gamma]}\[\exp\(-tC^{\fg,X,\tau}\)\].
\end{align}
\end{thm}
\begin{proof}
The proof is similar to the one given in \cite[Theorem 5.4]{Shen_Yu} and Theorem \ref{thmselnon}.
\end{proof}

By Theorem \ref{thmselnonob}, proceeding as in Corollary 
\ref{corMS0}, we get the following corollary. 

\begin{cor}\label{corMS0orb}
	The statement of Corollary \ref{corMS0}  holds for orbifolds. 
\end{cor}

\subsection{The Ruelle zeta functions on locally symmetric orbifolds}\label{sorbRuelle}
Let us follow 
\cite[Section 5.6]{Shen_Yu}. 
By \cite[Remark 
5.6, (5.59)]{Shen_Yu}, as in \eqref{DKVB}, the  closed geodesics (see \cite{GuruprasadHaefiger06} or \cite[Remark 2.26]{Shen_Yu}) on the 
orbifold $Z$ with positive length are given by 
\begin{align}
	\coprod_{[\gamma]\in [\Gamma_{+}]} B_{[\gamma]}. 
\end{align} 
Moreover, if $[\gamma]\in [\Gamma_{+}]$, all the elements of $B_{[\gamma]}$ have the same length 
$\ell_{[\gamma]} > 0$. Also,  the group $\mathbb{S}^{1}$ acts locally freely on the orbifold 
$B_{[\gamma]}$ by rotation. So, 
$B_{[\gamma]}/\mathbb{S}^{1}$ is still a closed  orbifold.   Set
\begin{align}\label{eqmggorb}
	m_{[\gamma]}=n_{[\gamma]} \left|\ker\(\bbS^{1}\to 
	\rm{Diffeo}\big(B_{[\gamma]}\big) \)\right|\in \mathbf{N}^{*}. 
\end{align}

Following \cite[Definition 5.10]{Shen_Yu}, for $\Re(\sigma)\gg1$ large enough, we define the Ruelle dynamical zeta 
function by the same formula   
	\begin{align}\label{defRrhoob}
		R_{\rho}(\sigma)=\exp\(\sum_{[\gamma]\in 
		[\Gamma_{+}]}\frac{\chi_{\rm 
		orb}(B_{[\gamma]}/\mathbb{S}^{1})}{m_{[\gamma]}}\Tr\[\rho(\gamma)\]e^{-\sigma \ell_{[\gamma]}}\),
	\end{align}
as in \eqref{defRrho}.  
By \cite[Proposition 5.9]{Shen_Yu}, the statement of  Remark \ref{reR=1} still holds for orbifolds. 

\begin{thm}
The statements of Theorem \ref{thmRmeo} holds for orbifolds. 
\end{thm}
\begin{proof}
	As previous, by Corollary \ref{corMS0orb},  we need only consider the case 		$\delta(G)=1$. If $\eta$ satisfies Assumption \ref{ass}, we 	
	can define the Selberg zeta function by the same formula 
	\eqref{eqdefsel} with $m_{[\gamma]}$ defined in \eqref{eqmggorb}. 
	Using Theorem \ref{thmselnonob}
	instead of Theorem \ref{thmselnon}, by Remark \ref{re17}, we can deduce that the statements of 
	Theorem \ref{thm:detfor}, \eqref{eqRZi}, \eqref{eqRZi3} still 
	hold when $\Gamma$ is not  torsion free. Moreover, all the 
	results in Section \ref{S:rep} hold when $\Gamma$ is not torsion 
	free. In this way, we get our theorem. 
\end{proof}

\def\cprime{$'$}
\providecommand{\bysame}{\leavevmode\hbox to3em{\hrulefill}\thinspace}
\providecommand{\MR}{\relax\ifhmode\unskip\space\fi MR }
\providecommand{\MRhref}[2]{%
  \href{http://www.ams.org/mathscinet-getitem?mr=#1}{#2}
}
\providecommand{\href}[2]{#2}

\end{document}